\documentclass{article}
\usepackage{graphicx,amsmath,amssymb,geometry,bbm,amsthm,color,physics,bm,comment,url}
\newtheorem{theorem}{Theorem}
\newtheorem{lemma}{Lemma}
\newtheorem{corollary}{Corollary}
\newtheorem{definition}{Definition}

\newtheorem{proposition}[lemma]{Proposition}
\newcommand{\lra}{\leftrightarrow}
\newcommand{\Zd}{\mathbb{Z}^{d}}
\newcommand{\E}{\mathbb{E}}
\newcommand{\Pp}{\mathbb{P}}
\newcommand{\nuIIC}{\nu}
\newcommand{\eps}{\varepsilon}

\newcommand{\diam}{\operatorname{diam}}
\newcommand{\dist}{\operatorname{dist}}

\newcommand{\sa}[1]{\ensuremath{\,{\buildrel #1 \over \longleftrightarrow}\,}}
\title{Limiting distribution of the chemical distance in high dimensional critical percolation}

\begin{document}
   \author{Shirshendu Chatterjee \thanks{Email: shirshendu@ccny.cuny.edu.} \\ \small{The City College of New York and CUNY Graduate Center}  \and Pranav Chinmay \thanks{Email: pchinmay@gradcenter.cuny.edu.} \\ \small{CUNY Graduate Center}  \and Jack Hanson \thanks{Email: jack.hanson@uni-hamburg.de.}\\ \small{Universit\"at Hamburg, The City College of New York, and CUNY Graduate Center} \and Philippe Sosoe \thanks{Email: ps934@cornell.edu.} \\ \small{Cornell University} }
\maketitle

\begin{abstract}We identify the asymptotic distribution of the chemical distance and other natural metrics (including the resistance) in high-dimensional critical percolation. When rescaled by the square of the Euclidean distance, each of these metrics converges in distribution to a multiple of the hitting time $T$ of a Brownian motion to hit $\mathbf{e}_1$ conditional on $\{T < \infty\}$. We extend these results to the near-critical regime, where the limiting distribution is now the analogue of $T$ for a killed Brownian motion.

These results are intended to be the foundation for an understanding of the metric space structure of high-dimensional clusters. They follow from a general theorem we find interesting in its own right, a ``law of large numbers'' for local functions summed along the backbone of a long open connection. A mixing result for open clusters \cite{CCHS} in the form of a robust convergence to the incipient infinite cluster measure plays a key role in the proofs.
\end{abstract}

\section{Introduction}
Bernoulli percolation on $\Zd$ has long been used as a canonical model of a random metric space; much work on percolation has focused on its metric space properties. Percolation clusters provide a natural and well-studied random environment for random walks and transport phenomena. The model at criticality is a particularly important case; it has been believed since the foundational work of de Gennes and Alexander-Orbach that random walk on critical clusters scales non-diffusively, though in general dimensions the precise scaling and limiting process remain elusive.

In high dimensions, by contrast, there are natural and longstanding conjectures for many geometric properties of clusters. It has been known in various senses since the pioneering work of Hara-Slade that the two-point function --- the probability that two vertices lie in the same open cluster --- behaves like the Green's function for simple random walk. Longstanding and widely held conjectures relate many other properties of critical high-dimensional percolation to those of critical branching random walk (BRW). See \cite{HH}, especially Chapter 2, for more on this analogy. For instance, in BRW, the ``chemical'' or graph distance between two distant vertices $x, y$ in the same cluster grows quadratically in the Euclidean distance between $x$ and $y$. 

Several authors \cite{KN2, HHH2} have proved that chemical distances in high-dimensional critical percolation also grow quadratically. We write $\dist(x,y)$ for the chemical distance between vertices $x$ and $y$ of $\Zd$ on the event $\{x \lra y\}$ that $x$ and $y$ are connected by an open path, with $\dist(x,y) = \infty$ when no such path exists. In an early breakthrough work \cite{KN2}, it was shown that 
\[ \E[\dist(x,y) \mid x \lra y] \asymp |x-y|^2, \]
where ``$\asymp$'' means that the two sides are bounded by constant multiples of each other as $|x - y| \to \infty$. In \cite{CHS}, the first, third and fourth authors of the present paper refined the understanding of the distribution of the chemical distance on the scale $n^2$. Among other results, \cite{CHS} showed the lower tail estimates
\begin{equation}\label{eqn: lower}
-\log \mathbb{P}_{p_c}(\dist(0, \partial [-n,n]^d) \le\lambda n^2\mid 0\leftrightarrow \partial [-n, n]^d)\asymp \lambda^{-1},
\end{equation}
as well as the upper tail bound
\[\mathbb{P}_{p_c}(\dist(0, \partial [-n,n]^d)>\lambda n^2)\le \exp(-c\lambda).\]
For the lower tail, van der Hofstad and Sapozhnikov had previously obtained
\begin{equation}\label{eqn: upper}
\mathbb{P}_{p_c}(\dist(0, \partial [-n,n]^d) \le \lambda n^2\mid 0\leftrightarrow \partial [-n, n]^d)\le \exp(-\lambda^{-1/2}).
\end{equation}
We note that the estimates \eqref{eqn: lower} and \eqref{eqn: upper} are consistent with the behavior of the exit time of a Brownian motion from the a ball of radius $n$, as one would expect from the BRW picture.

The above results do not precisely characterize the leading-order behavior of the chemical distance in the limit; at best, they give ``up-to-constants'' comparisons. The main result of the present paper, Theorem~\ref{thm: main} is a sharp characterization of the leading-order behavior of $d(0, x)$, conditional on $\{0 \lra x\}$. As one might expect from the BRW picture, we show that the conditional moments of $d(0,x)$ behave like the expected hitting time of simple random walk on $x$, conditional on the walk eventually visiting $x$. It follows that in the limit $|x| \to \infty$, the variable $d(0,x) / |x|^2$ is distributed as a constant multiple of a variable we denote by $T_d$. This variable $T_d$ is the hitting time of a standard $d$-dimensional Brownian motion on a unit vector, conditional (via an $h$-transform) on such hitting occurring in finite time. 

We furthermore identify this ``constant multiple'' scaling between the walk hitting times and the chemical distance explicitly in terms of other quantities in the percolation model. In high dimensions, when $\{0 \lra x\}$ occurs, any open path from $0$ to $x$ typically has a positive density of pivotal edges (cut edges); this idea has long been key to the analysis of the high-dimensional regime, and it underlies the relationship to BRW. The chemical distance between $0$ and $x$ is a sum of (one plus) the distances between the successive cut edges; see \eqref{eqn: series-law} below. Using a form of convergence to the incipient infinite cluster \cite{CCHS}, we show that the distances between a successive pair $e_1, f_1$ of cut edges and another successive pair $e_2, f_2$ of cut edges decorrelate, as long as these pairs are distant from each other. This implies that the chemical distance $\dist(0, x)$ is approximately a deterministic constant multiple of the number of pivotal edges. Such a result holds for resistances as well; this leads to the full strength of Theorem~\ref{thm: main}, which proves a convergence result for the joint law of the number of pivotals, the chemical distance, and the resistance. In fact, we prove a still more general limit theorem about moments of sums of ``local variables'' along the chain of pivotals connecting $0$ and $x$ as $|x| \to \infty$; see Theorems~\ref{thm: local-limit} and \ref{thm: local-mixed} below.

One key input to our results already mentioned is the convergence to the IIC proved by the present authors in \cite{CCHS}, which provides a sort of mixing bound for large open clusters. The other key input to our results is the assumption, stated at \eqref{eqn: two-pt} below, that $\Pp(0 \lra x)$ scales as $(\mathbf{c} + o(1)) |x|^{2-d}$. Such a result is expected for any $d > 6$; our results hold assuming this two-point asymptotic on $\Zd$ for any $d > 6$, when one chooses the edge set of $\Zd$ to be either the nearest-neighbor or spread-out edges. The result \eqref{eqn: two-pt} provides the key missing ingredient to the sketch from the last paragraph: it allows us to show that the pivotal chain has a random walk structure, which connects the number of pivotals to the hitting time of a random walk. Weaker versions of \eqref{eqn: two-pt} also underlie existing key estimates that allow us to control errors in our calculations. The assumption \eqref{eqn: two-pt} is in fact a result in many cases: it is known to hold \cite{FH} on $\Zd$ with nearest-neighbor edges with $d \geq 11$, or with sufficiently spread-out edges for all $d > 6$. See also \cite{SL1, SL2} for recent work on the error terms in \eqref{eqn: two-pt} and \cite{DMPS} for a recent novel approach to the study of the two-point function.

Our result is intended to provide the foundation for a deeper study of the metric space structure of critical percolation clusters. Through the connection to BRW, the scaling limits of critical clusters under various conditionings are related to objects arising from super-Brownian motion \footnote{After this work was posted, preprints of Blanc-Renaudie and Hutchcroft \cite{BRH} and Cabezas-Croydon-Fribergh-Kawamoto made certain limiting behavior explicit; see Section~\ref{sec:motiv}}. See \cite{HH} for a survey. In particular, conditioning on the cluster of the origin having a large cardinality, we obtain an object which has a metric structure approaching that of the integrated super-Brownian excursion (ISBE). Major work of Ben Arous-Cabezas-Fribergh gave a roadmap to a scaling limit for random walk on these conditioned clusters, assuming one could made the connection to the ISBE sufficiently precise. The present paper will be the basis of ongoing work studying the joint distances of multiple vertices in an open cluster, with the goal of a full understanding of the limiting metric space.

Our methods also extend into the near-critical regime, where edges are open with probability $p$ smaller than but close to the critical probability $p_c$. The main result here is a limiting law for the same metrics as above in this regime. As a consequence, we get a precise rate of exponential decay for the two-point function in this near-critical regime. See Section~\ref{sec: subcritical} below for a full statement of these results.

\subsection{Main Results}
\begin{definition}
    The set of (unordered) bonds on the lattice is denoted by $\mathcal{E}(\mathbb{Z}^d)$:
     \begin{equation}
        \mathcal{E}(\mathbb{Z}^d) = \{\{u,v\}: \|u-v\|_1 = 1\}.
    \end{equation}
    The set of ordered bonds on the lattice is the set of ordered pairs
    \begin{equation}
        \overline{\mathcal{E}}(\mathbb{Z}^d) = \{(u,v): \|u-v\|_1 = 1\}.
    \end{equation}
\end{definition}
Unless explicitly mentioned otherwise, ``bond" will henceforth always refer to ``ordered bond"; we use ``bond'' and ``edge'' interchangeably. \\

We recall the definition of independent Bernoulli bond percolation on $\mathbb{Z}^d$ at the critical density $p=p_c$. Formally, we consider the product probability space $(\Omega,\mathcal{F},\mathbb{P})$, where $\Omega$ is the set of configurations 
\[\Omega=\{0,1\}^{\mathcal{E}(\mathbb{Z}^d)},\]
with $\mathcal{F}$ the product $\sigma$-algebra and $\mathbb{P}$ the infinite Bernoulli product measure with parameter $p=p_c(\mathbb{Z}^d)$, the critical parameter for Bernoulli bond percolation. 
We will almost exclusively consider the case $p = p_c$, but, in the statement and proof of Corollary~\ref{cor: subcritical} below, we will also consider Bernoulli product measures with other parameter values $p$; we denote a typical such more general measure by $\mathbb{P}_p$.
For $\omega\in \Omega$, we let $\omega_e\in \{0,1\}$ be the status of the edge $e$. If $\omega_e = 1$, we call the edge $e$ ``open", and if $\omega_e = 0$, we call the edge $e$ ``closed".

 Some of our main results describe the behavior of several geometric quantities defined as a function of the random open subgraph of the percolation model. These are the number of pivotal (cut) edges, the chemical (graph) distance in the open cluster, and the electrical resistance (where each open edge is a resistor of resistance one). For formal definitions, see Section~\ref{sec: additive}.

\begin{theorem}\label{thm: main}
    Let $\mathbf{e}_1:=(1,0,\ldots,0)$ be the unit vector in the first coordinate direction in $\mathbb{Z}^d$. When there is no risk of
    confusion, we denote the origin $(0,\ldots,0)\in \mathbb{Z}^d$ simply by $0$. Define the following quantities:
    \begin{enumerate}
        \item $P_n$ denotes the number of pivotal edges for the connection from $0$ to $n\mathbf{e}_1$,
        \item $R_n$ denotes the effective resistance from $0$ to $n\mathbf{e}_1$ and
        \item $D_n$ denotes the chemical distance from $0$ to $n\mathbf{e}_1$.
    \end{enumerate}
     Let 
    \begin{equation}\label{eqn: zd-def}
        z_d:= \frac{\Gamma(\frac{d}{2}-1)}{2\pi^{\frac{d}{2}}},
    \end{equation}
    and suppose \eqref{eqn: two-pt} holds. (By \cite{HHS} and \cite{FH}, this last assumption is satisfied whenever $d\geq 11$.)
    
    Then there are constants $\mathbf{c}$, $\beta$, $\alpha_p$, $\alpha_r$ and $\alpha_d$ such that, conditioned on the existence of a connection from $0=(0,\ldots,0)$ to $n\mathbf{e}_1$, if we define the three sequences of random variables
    \begin{equation}
    \label{eq:XYZ} X_n=\frac{z_d P_n}{2d\alpha_p\mathbf{c} \beta n^2},\quad Y_n=\frac{ z_d R_n}{2d\alpha_r \mathbf{c}\beta n^2}, \quad Z_n=\frac{z_d D_n}{2d \alpha_d \mathbf{c}\beta n^2}.\end{equation}
    Then the vector $(X_n,Y_n,Z_n)$ converges in distribution to $(T_d,T_d,T_d)$, where $T_d$ is a random variable with density
    \begin{equation}\label{eqn: limit-density}
    f_d(t)=\frac{1}{2^{\frac{d}{2}-1}\Gamma(\frac{d}{2}-1)}t^{-\frac{d}{2}}e^{-1/2t}, \quad t\ge 0.
    \end{equation}
    The density $f_d$ is that of the time it takes a $d$-dimensional Brownian motion started at $0$ and conditioned to hit a fixed unit vector to reach said vector.
\end{theorem}
The constant $\beta=p_c/(1-p_c)$ appearing in the statement of the result is defined in terms of the 
critical probability $p_c$, while the constants $\alpha$ and $\mathbf{c}$ are defined in the discussion following the equations \eqref{eqn: two-pt} and \eqref{iic-constant-def}.
 As already mentioned, a key tool used in proving Theorem~\ref{thm: main} is a strong form of convergence to the incipient infinite cluster measure or ``IIC''. This is the probability measure $\nu$ defined as the weak limit
\begin{equation} \label{eq:firstiic} \nu = \lim_{|x| \to \infty} \mathbb{P}(\cdot \mid 0 \lra x)\ , \end{equation}
where the existence of the limit was proved in \cite{vdHJ}. See Theorem~\ref{thm:IICexistences} below for the aforementioned strong form of convergence.

For clarity, we have stated our main result as Theorem \ref{thm: main}. It follows from the following result which includes a finite obstacle near the origin. This is used crucially in the updated version of \cite{CCHS1}.
\begin{theorem}\label{thm: avoid-dist}
Let $F\subset\Zd$ be finite, and suppose that $0$ lies in the infinite component of $\Zd\setminus F$. For $x\in\Zd$ let
\begin{equation*}
A_x^F:=\{0\longleftrightarrow x\text{ in }\Zd\setminus F\},
\end{equation*}
 let 
\[D^F(0,x):=\mathrm{dist}_{\mathbb{Z}^d\setminus F}(0,x)\] be the chemical distance between $0$ and $x$ defined using only open edges of $\mathcal{E}(\Zd) \setminus F$,
and define
\begin{equation*}
\Xi(F):=\{0\longleftrightarrow\infty\text{ in }W\setminus F\}
\end{equation*}
under the IIC measure $\nuIIC$. For every $\lambda\ge 0$ and $\eps>0$,
\begin{equation*}
\sup_{\eps n\le |x|\le \eps^{-1}n}
\left|
\E\bigl[\exp(-\lambda a n^{-2}D^F(0,x))\mid A_x^F\bigr]
-\nuIIC(\Xi(F))\,\psi_d\bigl(\lambda |x/n|^2\bigr)
\right|\longrightarrow 0.
\end{equation*}
Here
\begin{equation*}
a=\frac{z_d}{2d\alpha_d \bm{c} \beta}
\end{equation*}
and $\psi_d$ is the Laplace transform of the distribution $T_d$ in \eqref{eqn: limit-density}. Note that the normalization by $a / n^2$ is the same as in the statement of Theorem~\ref{thm: main}.
\begin{equation}\label{eqn: laplace-Td}
\psi_d(s)=\E\exp(-sT_d), \qquad  s \geq 0.
\end{equation}
\end{theorem}

{\bf Remark 1.} For clarity of notation, we have formulated Theorem \ref{thm: main} in terms of the vector $\mathbf{e}_1$, but it will be clear from the proof that the same result holds with only minor notational adjustments if $\mathbf{e}_1$ is replaced by a general unit vector $\mathbf{v}$ and  $n\mathbf{e}_1$ replaced by the lattice site nearest to $n\mathbf{v}$.

{\bf Remark 2} The conclusion of Theorem \ref{thm: main} also holds for spread-out percolation models with sufficiently large spread parameter, whenever $d>6$. This will be clear from the proof, and the fact that the key inputs from \cite{CCHS} and \cite{HHS} are also valid for that model.

{\bf Remark 3} The density \eqref{eqn: limit-density} has  $\lfloor \frac{d}{2}-1\rfloor$ finite moments, while the higher moments are infinite. It is easy to see that the pre-limiting quantity, the chemical distance, does not have high finite moments either: the two-point function scaling implies that one can find pivotal edges for the connection between $0$ and $n\mathbf{e}_1$ outside of a box of size $M$ at a polynomial cost in $M$, so sufficiently high moments diverge.

{\bf Remark 4} The degeneracy of the joint limit $(X_n,Y_n,Z_n)$ was noted in a recent preprint of Hutchcroft and Blanc-Renaudie \cite{BRH}. This work appeared after the first version of the present work (see Section~\ref{sec:motiv}), which only proved convergence of the quantities separately. The observation is thus properly attributed to those authors. We have added it to Theorem \ref{thm: main} to show that it follows readily from our framework to compute moments of local quantities over pivotals.

\subsection{Application: two-point function and chemical distance in the scaling window}\label{sec: subcritical}
Here we note an a corollary of our result to \emph{near}-critical percolation. It is well-known that the correlation length exponent $\nu$ in high dimensional percolation is $1/2$ (the traditional symbol $\nu$ is not to be confused with the IIC measure from \eqref{eq:firstiic}). For example, see \cite{HS0}. 

At a heuristic level, this follows from Russo's formula: letting 
\[\tau_p(0,x)=\mathbb{P}_p(0\lra x)\]
denote the two-point function at density $p$, we have
\begin{equation}\label{eqn: russo}
\frac{\mathrm{d}}{\mathrm{d}p}\tau_p(0,n\mathbf{e}_1)=\frac{1}{p}\mathbb{E}_p[P_n\mid 0\lra n\mathbf{e}_1],
\end{equation}
where $P_n$, as above, denotes the pivotal count for $0\lra n\mathbf{e}_1$. Assuming the right side is of size about $n^2$ in $[p,p_c]$, we obtain
\[\log \frac{\tau_{p_c}(0,n\mathbf{e}_1)}{\tau_{p}(0,n\mathbf{e}_1)}\approx \frac{c}{p}(p_c-p)n^2.\]
The argument is somewhat circular, since it requires control of the pivotals for a range of $p$ (the scaling window). See \cite{CHS} for an argument along these lines giving the upper and lower bounds
\[c_1\tau_{p_c}(0,n\mathbf{e}_1) \exp(-C_2(p_c-p)n^2) \le \tau_p(0,n\mathbf{e}_1) \le  C_1\tau_{p_c}(0,n\mathbf{e}_1) \exp(-c_2(p_c-p)n^2)  \tau_{p_c}(0,n\mathbf{e}_1).\]

Theorem \ref{thm: main} shows that $P_n/n^2$ has a limiting distribution. Replacing the expected pivotal count in \eqref{eqn: russo} by a Laplace transform, we obtain an exact asymptotic for the distribution of the chemical distance in  near-critical percolation, where $p$ is tuned to lie in the scaling window adapted to the connection distance; this implies an asymptotic for the the two-point function as well.
We consider connections in the subcritical setting inside the near-critical window. For a vertex $y \in \Zd \setminus \{0\}$, we set $p = p(y, m) = p_c - m |y|^{-2}$; here $m > 0$ is fixed relative to $y$ and controls the depth of $p$ inside the near-subcritical window.

For the ease of our exposition, we couple the model at $p$ and the model at $p_c$ in a usual and natural way: first generating the set of open edges at $p_c$ using independent Bernoulli$(p_c)$ random variables, then using a new set of i.i.d.~Bernoulli$(p/p_c)$ random variables to determine which edges remain open in the model at $p$. This coupling is monotone: edges open at parameter value $p$ are also open at parameter value $p_c$. We write $\overline\Pp$ and $\overline\E$ for probabilities and expectations with respect to this coupling. Similarly, we write $\{0 \sa{p} y\}$, $\dist_p(0, y)$, and $P_p(0, y)$ for the connection event, the chemical distance, and the number of open pivotal edges in the $p$-open component of the coupled configuration, with analogous notation for the quantities in the $p_c$-open component.
\begin{corollary}\label{cor: subcritical}
    Fix $m > 0$, $\lambda \ge0$ and define $p(y, m)$ as a function of $y \neq 0$ as above. Let $a$ be defined as in the statement of Theorem~\ref{thm: avoid-dist}. We have
    \begin{equation}\lim_{|y| \to \infty} \overline\E\left[ \exp(-\lambda a |y|^{-2} \dist_p(0, y)) \mid 0 \sa{p_c} y \right] = \E_{p_c}\left[\exp\left(\left[- \lambda -  \frac{m \alpha_p}{\alpha_d a p_c}\right] T_d\right) \right],  \label{eq:lapsubchem} \end{equation}
    where $T_d$ is a random variable whose distribution has the density \eqref{eqn: limit-density} with respect to Lebesgue measure. 

    The quantity appearing in \eqref{eq:lapsubchem} is the Laplace transform of the defective distribution with (Lebesgue) density
    \[\frac{\exp\left[\frac{-1}{2t} - \frac{m \alpha_p t}{\alpha_d a p_c} \right]}{2^{d/2 - 1}\ \Gamma(d/2 - 1)\ t^{d/2}}, \qquad t > 0\ . \]
    Integrating, it follows that, for $p = p(y, m)$:
    \begin{equation} \label{eq:ratiopc} \lim_{|y| \to \infty} \frac{\tau_p(0,y)}{\tau_{p_c}(0,y)} = \psi_d(m \alpha_p / \alpha_d a p_c), \end{equation}
    where $\psi_d$ is the Laplace transform of $T_d$ as defined in the statement of Theorem~\ref{thm: avoid-dist}.
    Therefore, the limiting distribution of $\dist_p(0,y)$ conditional on $\{0 \sa{p} y\}$ has (Lebesgue) density
    \[\frac{\exp\left[\frac{-1}{2t} - \frac{m \alpha_p t}{\alpha_d a} \right]}{2^{d/2 - 1}\ \Gamma(d/2 - 1)\ t^{d/2}\ \psi_d(m \alpha_p / \alpha_d a p_c)}, \qquad t > 0\ . \]
\end{corollary}

     Using \eqref{eq:lapsubchem} and \eqref{eq:ratiopc} gives that the conditional distribution of $\dist_p(0, y)$ under $\Pp_p( \cdot \mid 0 \lra y)$ (for $p = p(y, m)$) converges to that of a random variable $D$ whose distribution has Laplace transform
    \[ \E[\exp(-s D)] = \frac{\psi_d(s +m \alpha_p / \alpha_d a p_c)}{\psi_d(m \alpha_p / \alpha_d a p_c)}\ . \]
    This implies that $D$ is distributed as the time $T_d^*$  for a Brownian motion killed at rate $m \alpha_p / \alpha_d a p_c$  to hit $\mathbf{e}_1$, conditional on such hitting happening before killing. See \eqref{eq:lapstar} and the preceding discussion below for more information.

    We note that the corollary also gives the precise exponential decay rate of $\tau_p$ in the scaling window. Indeed,
    \begin{align*}
        \psi_d(z) &= \int_0^\infty  \frac{\exp\left[\frac{-1}{2t} -zt \right]}{2^{d/2 - 1}\ \Gamma(d/2 - 1)\ t^{d/2}}\, \mathrm{d}t\\
        &= \frac{\sqrt{2}}{\Gamma(d/2-1)}\left(\frac{z}{2} \right)^{(d-2)/4} K_{\frac{d-2}{2}}(\sqrt{2z}),
    \end{align*}
    where $K_\nu$ is a modified Bessel function of the second kind.
    The asymptotic 
    \[K_{\frac{d-2}{2}}(z)\sim \left(\frac{\pi}{2z}\right)^{\frac{1}{2}}e^{-z}\]
    gives that, for a dimension-dependent constant $c$,
    \begin{align*}
        \tau_{p(y, m)}(0, y) &= c_1 (1 + o_{|y|}(1)) m^{(d-4)/4} \exp\left( \sqrt{\frac{2 m \alpha_p}{a \alpha_d p_c}}\right)
    \end{align*} 
    and so, if we evaluate the below limit as $|y| \to \infty$ and $p \to p_c^-$ with $|y|^2 (p_c - p)$ constant, we have
    \[\lim_{|y|\rightarrow \infty} \frac{\log\,  \tau_p(0, y)}{|y|} = \sqrt{\frac{2 (p_c - p) \alpha_p}{a \alpha_d p_c}} \]
    which gives the exact asymptotic behavior of the correlation length inside the scaling window. For related results, see \cite{HMS, SL3}.

\subsection{Ongoing and Related Work \label{sec:motiv}}
The main result of this paper, in the form appearing in Theorem \ref{thm: avoid-dist} plays a key in our study of the joint convergence of distances in large percolation clusters. Since the original version of this article was posted, the present authors have developed new techniques which, in parallel with refined versions of the one in the present paper, give precise characterizations of finer geometric properties of open clusters. In the manuscript \cite{CCHS1}, we prove convergence of the $k$-point function decorated by appropriately rescaled distances between pivotals along a canonical tree encoding the connectivity structure of the cluster. From this, we deduce the joint convergence of the rescaled matrix of all pairwise distances between $k$ points. 

Since the submission of the original version of this article, other authors have also posted related results. In a concurrent but entirely independent work, Blanc-Renaudie and Hutchcroft \cite{BRH} obtained super-Brownian limits for critical percolation clusters as metric-measure spaces in the Gromov-Hausdorff-Prohorov topology. That result rests on a computation of the $k$ point function with additional marked points, which allows them to obtain the joint distribution of the pairwise pivotal counts, distances and resistances between $k$ points of the clusters, and thus also implies the main result here in case $k=2$. A scaling limit result for the cluster measure in long-range percolation had been established earlier in spectacular work of Hutchcroft \cite{Hutch1,Hutch2,Hutch3}. Cabezas, Croydon, Fribergh and Kawamoto \cite{CCFK} also used our \cite{CCHS1} (first revision) to give a different proof of GHP convergence of the cluster measure  to the canonical measure of Super-Brownian motion.

\subsection{Outline of the proof}
We begin by establishing some notation and collecting key results from the literature we use in this paper in Section \ref{sec: defs}. We then introduce a notion of additive (over pivotals) quantities and notation for the bubble of a pivotal edge, the set of edges doubly connected to the far end of the pivotal. These are central to our argument in Section \ref{sec: additive}. The distance between two vertices is an example of an additive quantity in our sense: it can be decomposed into a sum over pivotal edges: each pivotal contributes one plus the distance between the pivotal and the last vertex in its bubble, which we refer to as the \emph{head} of the bubble (defined precisely in Proposition \ref{prop: head}). To find the distribution of the distance, we must then compute moments of sums over pivotals. 

The computation in Section \ref{sec: moments} shows that the neighborhoods of pivotals at large distance (the typical case) can be decoupled using convergence to the IIC (Theorem \ref{thm:IICexistences}) if the summands are random variables that depend only on the neighborhood of the pivotal. For this purpose, we define a class of collections of random variables that depend only on local information near a pivotal in Section \ref{sec: local-v}. The main moment computation of this paper, Theorem \ref{thm: local-limit}, concerns sums over pivotals of such quantities, and is stated in Section \ref{sec: local-limit}.

The proof of Theorem \ref{thm: local-limit} occupies all of Section \ref{sec: moments}. We first show that we can concentrate on sums over distant pivotal edges and compute the moments by a recursion detailed in Section \ref{sec: recursive}, using the three central Lemmas \ref{rec1}, \ref{rec2} and \ref{rec3}. These lemmas constitute the technical core of this paper. The leading order terms in the asymptotics given in those lemmas are extracted in Section \ref{sec: leading}. It is here that we apply convergence to the IIC to successively decouple the contribution of each pivotal factor to the moments. Most of the error terms treated in Section \ref{sec: LOT} appear due to our truncating the expectations to fixed neighborhoods of the pivotals before applying Theorem \ref{thm:IICexistences}.

In Section \ref{sec: control}, we derive several truncation results to show that one can with high probability replace the distance and effective resistance across bubbles, quantities which in principle could depend on the status of edges at an arbitrary distance, by suitable local approximations in the sense of Definition \ref{def: local-var}, to which we can apply Theorem \ref{thm: local-limit}.

Section \ref{section: BM} relates the limiting moments $I_k(M)$ appearing in Theorem \ref{thm: local-limit} to a Brownian Motion conditioned to hit a given unit vector, suitably defined. The proof of Theorem \ref{thm: main}, which combines the approximation by local quantities, the moment computations, and the identification of the limiting moments from previous sections, appears in Section \ref{sec: final-proof}.

\section{Preliminaries}

\subsection{Definitions and basic estimates} \label{sec: defs}
We introduce a few definitions used throughout the paper.

\begin{definition}
    For two events $A,B\in \mathcal{F}$, we denote by $A\circ B$ the event that $A$ and $B$ occur disjointly. See \cite[Section 1.3]{HH} for the precise definition of disjoint occurrence. 
\end{definition}
    
    We will repeatedly use the van den Berg-Kesten inequality \cite{BK}: if $A$ and $B$ are increasing, then:
    \begin{equation}\label{eqn: BK}
    \mathbb{P}(A\circ B)\le \mathbb{P}(A)\mathbb{P}(B).
    \end{equation}

    In the rest of the paper, we refer to \eqref{eqn: BK} as ``the BK inequality.''

\begin{definition}
    Given a bond $f = (u,v)$, we notate
    \begin{equation}
        \underline{f}:= u, \quad \text{and} \quad 
        \overline{f}: = v,
    \end{equation}
    the near-side and far-side of $f$ respectively.
\end{definition}

\begin{definition} \label{def: Cf}
    For $x\in \mathbb{Z}^d$ we define the cluster of $x$, denoted $C(x)$ as the vertices of $\mathbb{Z}^d$ connected to $x$ by a path of open edges. For a bond $f$, we denote by $C^f(x)=C^f(x,\omega)$ the cluster of $x$ in the configuration where $\omega_f=0$. 
\end{definition}

\begin{definition}\label{def: local-CA}
    Let $A\in \mathbb{Z}^d$ be a set of vertices, and $x,y\in \mathbb{Z}^d$. We write $x\stackrel{A}{\leftrightarrow} y$ if $x$ and $y$ are connected by a path of open edges, all of whose edges have both endpoints in $A$. If $A=\mathbb{Z}^d$, we write $x\leftrightarrow y$. We let 
    \[C_A(x)=\{y: \,x\stackrel{A}{\leftrightarrow} y \}.\]
\end{definition}

\begin{definition}\label{def: piv}
We say the ordered bond $f=(u,v)$ is \emph{pivotal} for the connection from $x$ to $y$ if
$x\leftrightarrow u$, $y\leftrightarrow v$ and $y\notin C^{(u,v)}(x)$.
\end{definition}

\begin{definition}
    Let $x,y\in \mathbb{Z}^d$, and $f_1,\ldots,f_k$ pivotal bonds for the connection from $x$ to $y$. We write $f_1<\ldots<f_k$ and say $(f_1,\ldots,f_k)$ are ordered pivotals for the connection from $x$ to $y$ if, for each $1\le i\le k$, we have 
    \[f_j\in C^{f_i}(x), \quad 1\le j\le i-1,\] 
    and for each $1 \le i \le k-1$,
    \[\overline{f_i}\stackrel{\mathbb{Z}^d\setminus C^{f_i}(x)}{\leftrightarrow} \underline{f_{i+1}}.\]
\end{definition}

We rely extensively on the following result proved in our previous work.
\begin{theorem}
			\label{thm:IICexistences} Define $B(n) := [-n, n]^d$ for each positive integer $n$; assume the asymptotic behavior $\mathbb{P}(0 \lra x) \asymp |x|^{2-d}$ from \eqref{eqn: HS-scaling} below.
           Let $V_n, D_n$ be such that, for each $n$, $(V_n \cup D_n) \cap B(n) = \varnothing$ and such that the origin $0=(0,\ldots,0)$ and some vertex of $V_n$ are in the same connected component of $\mathbb{Z}^d \setminus D_n$. Then
			\[\lim_{n \to \infty} \mathbb{P}(A \mid 0 \sa{\mathbb{Z}^d \setminus D_n} V_n )  =: \nu(A)\]
            for every cylinder event $A$, where $\nu$ is the \emph{incipient infinite cluster} (IIC) measure based at the origin. In particular, the conditional measure $\mathbb{P}(\cdot \mid 0 \sa{\mathbb{Z}^d \setminus D_n} V_n )$ converges weakly to $\nu$.
			
\end{theorem}
As an immediate consequence of the above theorem, if $h$ depends only on the status of bonds in a finite region (and hence bounded), then
\begin{equation}\label{eqn: IIC-exp}
\lim_{n\rightarrow \infty} \mathbb{E}[h\mid 0\sa{\mathbb{Z}^d \setminus D_n} V_n]=\mathbb{E}_{\nu}[h].
\end{equation}
Note that since the convergence holds for arbitary admissible $(V_n,D_n)$, the convergence must be uniform:
\[\eta(n)=\sup_{V,D\subset B(n)^c}\big|\mathbb{E}[h\mid 0\sa{\mathbb{Z}^d \setminus D} V]-\mathbb{E}_{\nu}[h]\big|\rightarrow 0.\]

Under the IIC measure, the origin $0=(0,\ldots,0)$ is almost surely contained in an infinite cluster. When writing $\nu$ probabilities, we denote this cluster by $W$. So, for example, for $x\in \mathbb{Z}^d$, $\nu(x\in W)$ denotes the IIC probability that $x$ lies in the (infinite) cluster of the origin.

Below, we also encounter shifted versions of $\nu$. For a vertex $u\in \mathbb{Z}^d$, we denote by $\nu_u$ the IIC measure shifted by $u$:
\[\nu_u(A):=\nu(\rho_{-u}(A)),\]
where $\rho_v:\Omega\rightarrow\Omega$ shifts all variable indices by $v$:
\[\big(\rho_{v}(\omega)\big)_{\{x,y\}}=\omega_{\{x+v,y+v\}}.\]
We denote the cluster of the vertex $v$ under $\nu_v$ by $W_v$.
\\

Throughout, we denote by $|x|$ the Euclidean norm of a vector $x=(x_1,\ldots,x_d)\in \mathbb{R}^d$:
\[|x|=\sqrt{\sum_{i=1}^d x_i^2}.\]
We use the customary notation for the $\ell^\infty$ norm:
\[\|x\|_\infty= \max\{|x_i|: \, 1\le i \le d\}.\]
For a vertex $x\in \mathbb{Z}^d$, we denote by $B(x,r)$ the $\ell^\infty$ ball centered at the vertex:
\[B(x,r):=\{y\in\mathbb{Z}^d: \|x-y\|_\infty \le r\},\]
 with $B(r) := B(0, r)$.
A key estimate we use repeatedly is the one-arm asymptotic of Kozma and Nachmias \cite{KN}:
\begin{equation}\label{eqn: KN}
\mathbb{P}(0\leftrightarrow \partial B(0,R))\le CR^{-2},       
\end{equation}
for $R\in \mathbb{Z}_+$.

We also use the so-called ``Japanese bracket" notation:
\[\langle x\rangle:= (1+|x|^2)^{\frac{1}{2}}.\]
The point is that $|x|/\langle x\rangle\rightarrow 1$ at infinity but $\langle x\rangle$ does not vanish.

We denote the indicator of an event $A$ by $\mathbbm{1}_A$. For an integrable random variable $X$ and events $A_1,\ldots, A_k$, we use the notation
\[\mathbb{E}[X,A_1,\ldots, A_k]=\mathbb{E}\left[X\prod_{i=1}^k\mathbbm{1}_{A_i}\right].\]
We may switch back and forth between the two notations above as per what is typographically convenient.

This paper contains several asymptotic statements. We use the $o$ and $O$ notations, writing 
\[A=O_{P\rightarrow \infty}(B)\] 
if there is a constant $C$ independent of the asymptotic parameter $P$ such that $|A|\le CB$. We also write
\[A=o_{P\rightarrow\infty}(B)\]
if the limit of $A/B$ tends to zero as the parameter $P$ tends to its limit. If it is clear what the underlying parameter is, we may omit the subscript on $o$ or $O$. 

We conclude this subsection with a simple lemma that allows approximation of connectivity events by local analogues.
\begin{lemma}
    \label{lem:prelocalize}
    There is a $C>0$ such that, for all $K \geq 1$ and all $x \notin B(K^d)$, we have
\[\mathbb{P}\left(B(K) \cap C(0) \neq B(K) \cap C_{B(K^d)}(0), \, 0 \lra x \right) \leq C K^{-d} |x|^{2-d}. \]
In particular,
\[ \nu\left(\exists y \in W \cap B(K) \text{ which is not connected to $0$ by an open path in $B(K^d)$} \right)  \leq C K^{-d}. \]
\end{lemma}
\begin{proof}
    We first claim that if
    \[B(K) \cap C(0) \neq B(K) \cap C_{B(K^d)}(0) \text{ and } 0 \lra x,\]
    then there is a $z\in B(K)$ such that
    \begin{equation}\label{eqn: decouple-z} \{z\leftrightarrow \partial B(z,K^d/2)\}\circ \{0\leftrightarrow x\},
    \end{equation}
    or there is a $z\in B(K)$ such that
    \begin{equation}
         \label{eqn: decouple-z-2} \{0 \leftrightarrow \partial B(K^d/2)\}\circ \{z\leftrightarrow x\}.
    \end{equation}
    Indeed, let $z$ be an arbitrarily chosen vertex of 
    \[B(K) \cap [C(0) \setminus C_{B(K^d)}(0)]\ ,\]
    and suppose that \eqref{eqn: decouple-z} does not occur. 
    
    Let $\gamma$ be an open path realizing the connection from $0$ to $x$, and let $\eta$ be an open path from $z$ to its first intersection $w$ with $\gamma$. Then $\eta$ cannot exit $B(z, K^d/2)$ because otherwise it and $\gamma$ would be witnesses for the event in \eqref{eqn: decouple-z}. Then the initial portion of $\gamma$ from $0$ to $w$ must exit  $B(K^d/2)$, since otherwise this portion and $\eta$ would witness a connection from $0$ to $z$ lying entirely in $B(K^d)$.
    The path consisting of $\eta$ and the terminal segment of $\gamma$ from $w$ to $x$ realizes the second event in \eqref{eqn: decouple-z-2}, and  the initial segment of $\gamma$ clearly realizes the first event. 
    
    From this, we find
\begin{equation}
\begin{split}
&\mathbb{P}\left( \left[ B(K) \cap C(0) \right] \neq \left[ B(K) \cap C_{B(K^d)}(0) \right], \, 0 \lra x \right)\\
\le~&\sum_{z\in B(K)}(\mathbb{P}(z\leftrightarrow \partial B(z,K^d/2))\mathbb{P}(0 \lra x)+\mathbb{P}(0\lra \partial B( K^d/2))\mathbb{P}(z\leftrightarrow x))\\
\le~& CK^d K^{-2d}\tau(0,x)\\
\le~& CK^{-d}|x|^{2-d}.
\end{split}
\end{equation}
This proves the first part of the lemma.

For the second, we write
\begin{align*}
&\nu\left(\exists y \in W \cap B(K) \text{ which is not connected to $0$ by an open path in $B(K^d)$} \right)\\
= &\lim_{R \to \infty} \nu\left(\exists \text{ an open path in $W$ from $0$ to $B(K)$ which exits $B(K^d)$ but not $B(R)$}\right)\\
= &\lim_{R \to \infty} \lim_{|x| \to \infty}  \mathbb{P}\left(B(K) \cap C_{B(R)}(0) \neq B(K) \cap C_{B(K^d)}(0) \mid 0 \lra x\right)\\
\leq &\lim_{|x| \to \infty}  \mathbb{P}\left(B(K) \cap C(0) \neq B(K) \cap C_{B(K^d)}(0) \mid 0 \lra x\right)\\
\leq &C K^{-d}
\end{align*}
by the first part of the lemma.
\end{proof}

\subsection{Additive Quantities}\label{sec: additive}
The quantities whose asymptotic distribution we investigate here in addition to the number of pivotal edges, are the \emph{chemical distance} and \emph{effective resistance} from the origin to the vertex $n\mathbf{e}_1$.

\begin{definition} The \emph{chemical distance} $\mathrm{dist}(x,y)$ between vertices $x$ and $y$ is set to be infinite unless $x$ and $y$ are connected by a path of open edges. Otherwise, $\mathrm{dist}(x,y)$ is the minimal number of edges in $\gamma$, taken over all open paths $\gamma$ connecting $x$ to $y$.
\end{definition}

\begin{definition} Given two vertices $x, y$, the effective resistance $R_{\mathrm{eff}}(x,y)$ between $x$ and $y$ is  set to be $R_{\mathrm{eff}}(x, y) = \infty$ if they are not connected by an open path. 

Otherwise, we let $\Theta$ be the class of antisymmetric functions $\theta$ on the ordered edges of $\mathbb{Z}^d$ satisfying $\theta(\vec{e}) = 0$ when $\omega_e = 0$, where $e$ is the unordered edge corresponding to $\vec{e}$. We then write
\[R_{\mathrm{eff}}(x,y) = \inf_{\substack{\theta \in \Theta:\\ \theta \text{ a unit flow } x \to y}} \frac{1}{2}\sum_{\vec{e}} \theta(\vec{e})^2\ . \]
Here, we say that $\theta$ is a unit flow from $x$ to $y$ if, for any $z\in \mathbb{Z}^d$,
\[\sum_{q} \theta((z, q)) = \mathbf{1}_{z = x} - \mathbf{1}_{z = y}\ . \]
See \cite{LP} for more information.
\end{definition}

The key property of the chemical distance and effective resistance that we use is the following additivity:
\begin{equation} \label{eqn: series-law}
\begin{gathered}
    \text{if $(w, z)$ is an open pivotal bond for the open connection from $x$ to $y$, then}\\
    \delta(x, y) = 1 + \delta(x, w) +\delta(y, z),\
\end{gathered}
\end{equation}
where $\delta$ denotes $\mathrm{dist}$ or $R_{\mathrm{eff}}$.

\subsubsection{Bubbles}\label{sec: bubble-def}
Let $f=(\underline{f},\overline{f})$ be an ordered edge. Suppose $f$ is pivotal for the connection between the origin $0$ and $n\mathbf{e}_1$ in the sense of Definition \ref{def: piv}. Recall from Definition \ref{def: Cf} that $C^f(0)$ represents the sites connected to $0$ without using the edge $f$. The \emph{bubble} of $\overline{f}$, denoted $bubble(\overline{f})$, is the set of edges with two edge-disjoint connections to $\overline{f}$ off $C^f(0)$ (i.e. in $\mathbb{Z}^d\setminus C^f(0)$). 

Denote
\[\mathcal{S}(\overline{f},C^f(0))=\{x\in \mathbb{Z}^d: x\leftrightarrow n\mathbf{e}_1 \text{ off } bubble(\overline{f})\}.\]
 
The quantity $D(bubble(\overline{f}))$ is defined as the distance from $\overline{f}$ to the set $\mathcal{S}(\overline{f},C^f(0))$:
\[D(bubble(\overline{f})):=\mathrm{dist}\big(\overline{f},\mathcal{S}(\overline{f},C^f(0))\big).\]
That is, the minimal number of edges in any open path from $\overline{f}$ to $\mathcal{S}(\overline{f},C^f(0))$.

\begin{proposition}\label{prop: head}
    Suppose $\overline{f}$ does not have two edge-disjoint connections to $n\mathbf{e}_1$. Then, there is a unique vertex $z\in \mathcal{S}(\overline{f},C^f(0))$ that is an endpoint of an edge in $bubble(\overline{f})$. We call $z$ the \emph{head} of the bubble $bubble(\overline{f})$ and denote it by $head(bubble(\overline{f}))$. 

\begin{proof}
    If $\overline{f}$ does not have two edge-disjoint connections to $n\mathbf{e}_1$, then $n\mathbf{e}_1$ is not an endpoint of an edge in $bubble(\overline{f})$. Choosing an open path $\gamma$ from $\overline{f}$ to $n\mathbf{e}_1$, the last vertex of $\gamma$ that is an endpoint of an edge in $bubble(\overline{f})$ belongs to $\mathcal{S}(\overline{f},C^f(0))$.
    
    Suppose there are two distinct vertices $z$, $z'$ as in the proposition. $z$ has two disjoint connections to $\overline{f}$, one of which does not contain $z'$. Similarly, $z'$ has a connection to $\overline{f}$ which does not contain $z$. These connections lie in $bubble(\overline{f})$. Since $z$ and $z'$ are distinct, they have connections $\gamma$ and $\gamma'$ off $bubble(\overline{f})$ to $n\mathbf{e}_1$ which do not coincide. By considering the first vertex along $\gamma$ which belongs to $\gamma'$, we find that the edges of $\gamma$ must in fact belong to $bubble(\overline{f})$, a contradiction, so $z=z'$.
\end{proof}
\end{proposition}
If $\overline{f}$ is doubly connected to $n\mathbf{e}_1$, we set $head(bubble(\overline{f}))=n\mathbf{e}_1$.
Note that 
\begin{equation}\label{eqn: dist-to-head}
D(bubble(\overline{f}))=\mathrm{dist}(\overline{f},head(bubble(\overline{f}))).
\end{equation}
We also define
\[R_{\mathrm{eff}}(bubble(\overline{f})):=R_{\mathrm{eff}}\big(\overline{f},head(bubble(\overline{f}))\big).\]

\subsubsection{Additivity}

 Let $\delta(0,n\mathbf{e}_1)$ denote $D_n=\mathrm{dist}(0,n\mathbf{e}_1)$ or $R_n=R_{\mathrm{eff}}(0,n\mathbf{e}_1)$. 
 Assuming that $0$ is connected but not doubly connected to $n\mathbf{e}_1$, we let $\{e_i=(\underline{e_i},\overline{e_i}\}))^q_{i=1}$ be the sequence of pivotal edges for a connection from $0$ to $n \mathbf{e}_1$. 
    Let the quantity $\delta(e_i)$ denote the chemical distance or the electrical resistance across the bubble of $\overline{e_i}$:
    \[\delta(e_i)=D(bubble(\overline{e_i})) \text{ or } R_{\mathrm{eff}}(bubble(\overline{e_i})).\]
    Then, if there are $q$ pivotal edges, we have by \eqref{eqn: series-law}
    \begin{equation}
       \label{eq:deltasum}
        \delta(0, n\mathbf{e}_1) = \delta(0, \underline{e_1}) + q + \sum_{i=1}^{q} \delta(e_i)\ = \sum_{i=0}^q h(e_i),
    \end{equation}
    where $h(e_i)=1+\delta(e_i)$ for all $i \geq 1$, and $h(e_0):= \delta(0,\underline{e_1})$. We compute the asymptotic distribution of $\delta(0,n\mathbf{e}_1)$ using the method of moments, introducing a number of approximations along the way.

\subsection{Local variables} \label{sec: local-v}
The quantities $h(e_i)$ appearing in the sum \eqref{eq:deltasum} a priori depend on edges at arbitrary distance from the base pivotal edge $e_i$. However, this happens only if $bubble(\overline{f})$ is large, an unlikely event. We approximate $h(e_i)$ by local variables suitable for our moment computations.

We associate to each bond of $\mathbb{Z}^d$ and set $V\subset \mathbb{Z}^d$, a random variable with certain properties that encapsulate this locality. Let $L \geq 1$ be fixed. \\
\begin{definition}\label{def: local-var}
    Let $(g(e,V))$, where $e$ ranges over $\overline{\mathcal{E}}(\mathbb{Z}^d)$ and $V$ over subsets of $\mathbb{Z}^d$, be a collection of random variables. Suppose that $g$ has the following properties uniformly on its domain: \\
    
    (i) Denote by $C^V(\overline{e})$ the cluster of vertices connected to $\overline{e}$ by open edges with no endpoint in $V$. We assume that, for each lattice animal $\mathcal{C}$
    \[g(e,V)\mathbbm{1}_{C^V(\overline{e})=\mathcal{C}}\]
    is measurable with respect to the edges in $\mathcal{C}$.
    
    (ii) For any $e$ and $V$,
    \begin{equation} 
    g(e,V) \in \sigma(B(e,L)),
    \label{rv-cyl}
    \end{equation}
    and, if $V\cap B(\overline{e},L)=V'\cap B(\overline{e},L)$, then
    \begin{equation}\label{eqn: V-locality}
    g(e,V)=g(e,V').
    \end{equation}
    
    (iii) there is a constant $C$ such that \begin{equation}
    g(e,V) \leq CL^d \  a.s.
    \label{rv-bdd}
    \end{equation}

    (iv) For any $\epsilon>0$, for large enough $n$, for any $k$-tuple of edges $(f_1,\dots,f_k)$ such that
    \[\min_{i=1,\dots,k+1} \|\underline{f_i} - \overline{f_{i-1}}\|_\infty \geq \epsilon n,\] 
    and for any sets $V_i$ such that $f_i\notin V_i$, we have that
    \begin{equation}
        \prod_{i=1}^{k} g(f_i,V_i),
        \label{rv-indep}
    \end{equation}
    is independent of $\sigma(f_1,\dots,f_k)$. \\

    (v) Translation invariance: for any $f$,
    \begin{equation}
        \mathbb{E}_{\nu_{\overline{f}}}[g(f,V)] = \mathbb{E}_{\nu}[g((\underline{f} - \overline{f}, 0),\rho_{-\overline{f}}V)],
        \label{rv-trans}
    \end{equation}
    where $\nu_{\bar e}$ represents the IIC measure based at the endpoint $v$ of the ordered edge $e=(u,v)$.

    Then we call $g$ a family of \textbf{local variables}. 
\end{definition}

\subsubsection*{Truncated Quantities}
Our main examples of local variables are truncated versions of the chemical distance and effective resistance across the bubble at one end of a pivotal edge. 

For a set of \emph{edges} $A\subset \mathcal{E}(\mathbb{Z}^d)$, and two sets $B,C\subset \mathbb{Z}^d$, we define the chemical distance between $B$ and $C$ as the minimal number of edges in any lattice path consisting of open edges of $A$ joining a vertex in $B$ to a vertex in $C$. We denote this by $\mathrm{dist}_A(B,C)$.

Let $f$ be an ordered edge of $\mathbb{Z}^d$ and $V\subset \mathbb{Z}^d$ be a lattice animal with $\underline{f}\in V$. We recall from Definition~\ref{def: local-var} above that $C^V(\overline{f})$ is the cluster of $\overline{f}$ in $\mathbb{Z}^d\setminus V$: the vertices that can be reached by a path of edges both of whose endpoints lie outside of $V$. When treating an edge $f$ as an unordered set $\{\underline{f}, \overline{f}\}$, this agrees with the notation in Definition~\ref{def: Cf}, and so we trust there will be no confusion. The $V$-\emph{bubble} of $\overline{f}$ denotes the edges in $C^V(\overline{f})$ doubly connected to $\overline{f}$ in $\mathbb{Z}^d\setminus V$. We denote this set by $bubble_V(\overline{f})$. Note that if $V\subset U$, then $bubble_U(\overline{f})\subset bubble_V(\overline{f})$.

Finally, we define 
\begin{equation}\label{eqn: SLVf}
S_L(f,V)=\{x\in C^V(\overline{f}): x \leftrightarrow \partial B(\overline{f},L)\text{ off } bubble_V(\overline{f})\}.
\end{equation}
We also define
\begin{equation}\label{eqn: SinftyV}
S_\infty(f,V):=\cap_{L\ge 1} S_L(f,V).
\end{equation}

\begin{definition}
The chemical distance across the bubble of $\overline{f}$, truncated at Euclidean distance $L$, is defined as
\begin{equation} \label{eqn: L-bubble}
g_{d,L}(f,V):= (1+\mathrm{dist}_{bubble_V(\overline{f})}\big(\overline{f}, S_L(f,V)\big))\mathbbm{1}_{\mathrm{diam}(bubble_V(\overline{f}))\le L/2}.
\end{equation}

The truncated effective resistance accross the $V$-bubble of $\overline{f}$ is defined by
\begin{equation}\label{eqn: L-resistance}
    g_{r,L}(f,V):=(1+R_{\mathrm{eff}}\big(\overline{f}, S_L(f,V)\big))\mathbbm{1}_{\mathrm{diam}(bubble_V(\overline{f}))\leq L/2}.
\end{equation}
\end{definition}
These approximations converge to the corresponding untruncated quantities as the parameter $L$ tends to $\infty$.

\subsection{Limit Theorem for local variables}\label{sec: local-limit}

Here, we state our main result for sums over local variables, after introducing a few quantities that appear in the statement.
First, let $\bm{c}$ denote the van der Hofstad-Hara-Slade two-point asymptotic constant \cite{HHS}:
\begin{equation}\label{eqn: two-pt}
\bm{c}:=\lim_{|x|\rightarrow\infty} |x|^{d-2}\mathbb{P}(0\leftrightarrow x)
\end{equation}
The existence of the limit is due to Hara in the case of nearest-neighbor percolation \cite{Hara}, with recent improvements due to Slade and Liu \cite{SL1,SL2}. For many purposes, the scaling of the two-point function, an immediate consequence of \eqref{eqn: two-pt} suffices:
\begin{equation}\label{eqn: HS-scaling}\tau(x,y):=\mathbb{P}(x\leftrightarrow y)\asymp |x-y|^{-d+2}.
\end{equation}

For a collection of local variables $g$, we define:
\begin{equation} 
    \alpha_g:= \mathbb{E}_{\nu_{\mathbf{e}_1}} \otimes \mathbb{E}_{\widetilde{\nu_{0}}}\left[\mathbbm{1}_{\exists\,\text{a path to}\,\infty\,\text{in}\,W_{\mathbf{e}_1}\,\text{off}\,\widetilde{W}_{0}} \cdot g\big((0,\mathbf{e}_1),\widetilde{W}_0\big)\right].
    \label{iic-constant-def}
\end{equation}
The expectation is with respect to the product measure $\mathbb{E}_{\nu_{\mathbf{e}_1}} \otimes \mathbb{E}_{\widetilde{\nu_{0}}}$ of the two IIC measures $\nu_{\mathbf{e}_1}$ and $\nu_0$. Recall that we denote the cluster of a vertex $v$ under $\nu_v$ by $W_v$. The notation $\tilde{\nu}_0$ is used to indicate that $\widetilde{W}_0$ is the IIC corresponding to this factor of the product measure.
For notational convenience, we set
\[\alpha_{d,L}:=\alpha_{g_{d,L}},\quad \alpha_{r,L}:=\alpha_{g_{r,L}}, \]
where $g_{d,L}$ and $g_{r,L}$ are the truncated distance \eqref{eqn: L-bubble} and truncated resistance \eqref{eqn: L-resistance}. These quantities will have limits as $L \rightarrow \infty$, which will define $\alpha_d$ and $\alpha_r$ respectively, as discussed in Proposition \ref{prop: alphaL}. We also define
\[ \alpha_p:= \mathbb{E}_{\nu_{\mathbf{e}_1}} \otimes \mathbb{E}_{\widetilde{\nu_{0}}}\left[\mathbbm{1}_{\exists\,\text{a path to}\,\infty\,\text{in}\,W_{\mathbf{e}_1}\,\text{off}\,\widetilde{W}_{0}}\right], \]
which corresponds to the case when $g$ is identically $1$. Throughout, we may simply write $\alpha$ for the generic $\alpha_g$ when the context is clear. \\

Next, we set
\begin{equation}\label{eqn: beta-def}
    \beta:= \frac{p_c}{1-p_c}.
\end{equation}
For each integer $k\ge 1$ and $M>0$, we define the quantities
\begin{equation}
\label{eqn: IM}
I_k(M) := \int_{([-M,M]^d)^k} \prod_{i=0}^{k} |x_{i+1}-x_i|^{2-d}\,\mathrm{d}x_1\cdots\mathrm{d}x_k,
\end{equation}
where $x_0 = (0, \ldots,0)$, $x_{k+1} = \mathbf{e}_1$, and $x_i$ for $i = 1,\dots,k$ are vectors in $\mathbb{R}^d$. Note that the integral is absolutely convergent. We also encounter the following integrals:
\begin{equation}
I_k(\epsilon,M) := \int_{([-M,M]^d)^k \setminus (B(0,\epsilon)\cup B(\mathbf{e}_1;\epsilon))} \mathbbm{1}_{\min_{i=1,\dots,k-1} |x_{i+1}-x_i| \geq \epsilon} \prod_{i=0}^{k} |x_{i+1}-x_i|^{2-d}\,\mathrm{d}x_1\cdots\mathrm{d}x_k,
    \label{farregime-int}
\end{equation}
By the Dominated Convergence Theorem, we have
\begin{equation}\label{eqn: DCT-app}
\lim_{\epsilon\rightarrow 0} I_k(\epsilon,M)=I_k(M).
\end{equation}

Let $f_1,\ldots, f_k$ be a collection of edges in $\mathbb{Z}^d$. Let $\overline{f_0}:= 0=(0,\ldots,0)$ and $\underline{f_{k+1}}:=n\mathbf{e}_1$, and define for all $0 \leq i \leq m-1$, 
\begin{equation}
    \zeta_i:= \bigcup_{j=0}^{i}C(\overline{f_j}),
\end{equation}
and the restricted union of clusters
\begin{equation}
    \zeta^*_i:= \bigcup_{j=0}^{i} C^{f_{j+1}}(\overline{f_j}),
\end{equation}
with $\zeta_{-1} = \zeta^*_{-1} = \varnothing$.
\begin{figure}[htbp]  
  \centering
  \includegraphics[width=0.55\textwidth]{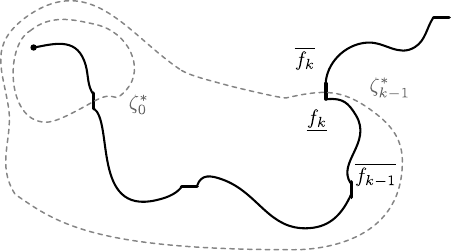}
  \caption{An illustration of the cluster $\zeta_{k-1}^*$}
  \label{fig:zeta}
\end{figure}

The following is our main limit theorem for local variables. Fix $K_H>0$, and let $H$ be a bounded random variable such that, for every $A\subset B(K_H)$
\begin{equation} \label{eq:Hdef} H\mathbf{1}[C_{B(K_H)}=A]\end{equation}
is measurable with respect to the status of edges in $\mathcal{E}(B(K_H))$ with at least one endpoint in $A$. Below, be abbreviate this by saying that $H$ is measurable with respect to $\sigma(C_{B(K_H))}$.
\begin{theorem}\label{thm: local-limit}
Let $g$ be a family of local variables and let $H$ be a variable as described at and above \eqref{eq:Hdef}. For all $M \in (2,\infty)$,
\begin{equation}
\begin{split}
    &\lim_{n \rightarrow \infty} \frac{\mathbb{E}\left[H\sum_{(f_1,\dots,f_k) \in B(Mn)^k} \prod_{i=1}^{k} g(f_i,\zeta^*_{i-1}), (f_1,\dots,f_k)\,\text{ordered open pivotals for }\,0\leftrightarrow n\mathbf{e}_1\right]}{n^{2(k+1)-d}}\\
    =~& \bm{c}^{k+1} \cdot (\alpha_g \beta)^k \cdot (2d)^k  \cdot \mathbb{E}_{\nu}[H] \cdot I_k(M),
    \end{split}
    \label{mainthm-eq}
\end{equation}
where $I_k(M)$ is defined in \eqref{eqn: IM}. The extra factor of $\mathbf{c}$ in \eqref{mainthm-eq} disappears under the natural conditioning on the event $\{0\leftrightarrow n\mathbf{e}_1\}$. We recall that $\mathbb{E}_\nu$ is expectation with respect to the IIC in Theorem \ref{thm:IICexistences}.
\label{mainthm}
\end{theorem}

We record the following slight generalization of Theorem \ref{thm: local-limit}, whose proof is identical after minor notational changes.

\begin{theorem}\label{thm: local-mixed}
    Let $g^{(1)},\ldots g^{(k)}$ be families of local variables. For all $M\in (2,\infty)$,
   \begin{align*} \lim_{n\to\infty}
&\frac{\E\Big[H,\sum_{(f_1,\dots,f_k)\in B(Mn)^k}
\prod_{i=1}^{k} g^{(i)}(f_i,\zeta^*_{i-1}),\
(f_1,\dots,f_k)\text{ ordered open pivotals for } 0\lra ne_1\Big]}
{n^{2(k+1)-d}}\\
&= c^{k+1}\beta^k(2d)^k\mathbb{E}_\nu[H]\Big(\prod_{i=1}^k\alpha_{g^{(i)}}\Big)\, I_k(M).
\end{align*}
\end{theorem}

\section{Moment estimates}\label{sec: moments}
Here we prove Theorem \ref{thm: local-limit}. Throughout this section, we let $k \geq 1$ be fixed, and let $(f_1,\dots,f_k)$ denote an ordered $k$-tuple of bonds, with the convention that $\overline{f_0}:= 0$, and $\underline{f_{k+1}}:= n\mathbf{e}_1$. Let $M > 2$ be fixed.

\begin{definition}
    The set of far-regime edges $F(\epsilon, M)$ is defined by
    \begin{equation}
    F(\epsilon,M):=\{(f_1,\dots,f_k) \in B(Mn)^k: \min_{i=1,\dots,k+1} |\underline{f_i} - \overline{f_{i-1}}| \geq \epsilon n\}.
    \label{farregime-def}
    \end{equation}
\end{definition}

\begin{definition}
    The near-regime $N(\epsilon,M)$ is defined by
    \begin{equation}
    N(\epsilon,M):= B(Mn)^k \setminus F(\epsilon,M).
    \label{nearregime-def}
\end{equation}
\end{definition}

\begin{proposition}[Near-regime estimate]
    Let $g$ be a family of local variables. There exists a constant $C$ so for all $M \in (2,\infty)$ and all $\epsilon$ sufficiently small, 
    \begin{equation}
        \limsup_{n \rightarrow \infty} n^{d-2(k+1)} \sum_{N(\epsilon,M)} \mathbb{E}\left[|H|\prod_{i=1}^{k} g(f_i,\zeta^*_{i-1}), (f_1,\dots,f_k)\,\text{ordered open pivotals for }\, 0 \leftrightarrow n\mathbf{e}_1\right] \leq C \cdot\|H\|_\infty L^{dk} \cdot \epsilon^2.
        \label{nearregime-eq1}
    \end{equation}
    \label{nearregime-thm}
\end{proposition}
The proof of this result depends on the following lemma, proved in Section \ref{sec: convolution}.
\begin{lemma}[Convolution bound]
There exist a $C > 0$ such that, for each $k \geq 1$ and each $1 \leq i \leq k+1$, we have
\begin{equation*}
\begin{gathered}
 \sum_{\substack{x_1, \dots, x_k \in B(2 \|y\|_\infty) \\ |x_i - x_{i-1}| < \epsilon \|y\|_{\infty}}} \langle x_1\rangle^{2-d} \langle x_2 - x_1\rangle^{2-d} \dots \langle x_k - x_{k-1}\rangle^{2-d} \langle y - x_k\rangle^{2-d} \leq C^k  \epsilon^2\langle y\rangle ^{2(k+1) - d},
\end{gathered}
\end{equation*}
for all $i$ and $y$, where $x_0 = 0$, and $x_{k+1} = y$.  
\label{convolution-bound}
\end{lemma}

\begin{proof}[Proof of Proposition \ref{nearregime-thm} given Lemma \ref{convolution-bound}]
    Bounding all $g$ factors using \eqref{rv-bdd}, and then applying the van den Berg-Kesten inequality \eqref{eqn: BK} and the two-point upper bound \eqref{eqn: HS-scaling}, the sum in \eqref{nearregime-eq1} is bounded above by
    \begin{equation}
        C \cdot\|H\|_\infty L^{dk}\sum_{N(\epsilon,M)}  \prod_{i=1}^{k+1} |\underline{f_{i}}-\overline{f_{i-1}}|^{2-d}.
        \label{nearregime-eq2}
    \end{equation}
    We use the convolution bound Lemma \ref{convolution-bound} to bound the sum over the $f_i$ in \eqref{nearregime-eq2} by
    \begin{equation}
        C \cdot \|H\|_\infty  L^{dk} \cdot \epsilon^2 \cdot n^{2(k+1)-d}.
    \end{equation}
    Divide by $n^{2(k+1)-d}$ and let $n \rightarrow \infty$ to conclude. 
\end{proof}
Next, we state the result for the far-regime. 
\begin{proposition}[Far-regime estimate]
Let $g$ be a family of local variables. For all $M \in (2,\infty)$,
\begin{equation}
\begin{split}
    &\lim_{n \rightarrow \infty} n^{d-2(k+1)} \sum_{F(\epsilon,M)} \mathbb{E}\left[H\prod_{i=1}^{k} g(f_i,\zeta_{i-1}^* ),(f_1,\dots,f_k)\,\text{ordered open pivotals for}\,~ 0 \leftrightarrow n\mathbf{e}_1\right]\\
    =~&\bm{c}^{k+1} \cdot (\alpha\beta)^k  \cdot (2d)^k \cdot \mathbb{E}_\nu[H]\cdot I_k(\epsilon,M).
\end{split}
    \label{farregime-thm-eq}
\end{equation}
where  $I_k(\epsilon,M)$ is defined in \eqref{farregime-int}.
\label{farregime-thm}
\end{proposition}
The proof of this proposition appears in Section \ref{sec: far-regime-proof} below, following a sequence of intermediate results that contain the main computation in this paper.

\subsubsection{Proof of Theorem \ref{mainthm}}
Putting together the near- and far-regime estimates in Propositions \ref{nearregime-thm} and \ref{farregime-thm}, we  obtain the main result.
\begin{proof}[Proof of Theorem \ref{mainthm}]
Apply Proposition \ref{nearregime-thm} and Proposition \ref{farregime-thm} to 
\begin{equation}
    \frac{\mathbb{E}\left[H\ \sum_{(f_1,\dots,f_k)\in B(Mn)^k} \prod_{i=1}^{k} g(f_i,\zeta^*_{i-1}),(f_1,\dots,f_k)\,\text{ordered open pivotals for}\, 0\leftrightarrow n\mathbf{e}_1\right]}{n^{2(k+1)-d}},
\end{equation}
to obtain the far region asymptotic
\begin{equation}
    (1+o_{n\rightarrow\infty}(1))\cdot\bm{c}^{k+1} \cdot  (\alpha\beta)^k \cdot (2d)^k\cdot  I_k(\epsilon,M)+ O(L^{kd}\cdot \|H\|_{\infty} \cdot \epsilon^2)
\end{equation}
Let $n \rightarrow \infty$ and $\epsilon \rightarrow 0$ to conclude using \eqref{eqn: DCT-app}.
\end{proof}

\subsection{Far-regime}
In this section, we introduce the ingredients for the proof of Proposition \ref{farregime-thm}. We assume that the condition \eqref{farregime-def} holds throughout.
\subsubsection{Recursive setup for the far-regime}
For some fixed family of local variables $g$ and $(f_1,\dots,f_m) \in F(\epsilon,M)$, we seek to describe the limiting behavior of
\begin{equation}
    \mathbb{E}\left[H\prod_{i=1}^{m} g(f_i,\zeta_{i-1}^*),(f_1,\dots,f_{m})\,\text{ordered open pivotals for}\,0\leftrightarrow n\mathbf{e}_1\right].
    \label{qty}
\end{equation}
Now that we are only concerned with a fixed $(f_1,\dots,f_m)$, we abbreviate:
\begin{equation}
\begin{split}
X_j&:=\prod_{i=1}^{j} g(f_i, \zeta_{i-1}),\\
    X_j^*&:=\prod_{i=1}^{j} g(f_i, \zeta_{i-1}^*).
\end{split}
\end{equation}

\begin{lemma}
    For any ordered $m$-tuple $(f_1,\dots,f_m)$, it holds that \eqref{qty} equals
    \begin{equation}
        \beta^m \cdot \mathbb{E}\left[HX_m, \bigcap_{i=0}^{m} \left\{\overline{f_i} \leftrightarrow \underline{f_{i+1}}\,\text{off}\,\zeta_{i-1}\right\}\right],  
        \label{open-closed}
    \end{equation}
    where we have $\overline{f_0}:= 0$, $C(\underline{f_0}):=\varnothing$, and $\underline{f_{m+1}}:=n\mathbf{e}_1$.
    \label{open-closed-lem}
\end{lemma}
\begin{proof}
    By definition, \eqref{qty} equals
    \begin{equation}
        \mathbb{E}\left[HX_m^*,\bigcap_{i=0}^{m} \left\{f_1,\dots,f_m\,\text{open}, \overline{f_i} \leftrightarrow \underline{f_{i+1}}\,\text{off} \,\zeta^*_{i-1}\right\}\right].
    \end{equation}
    For $n$ sufficiently large, $f_i\notin B(K_H)$, so $H$ does not depend on the status of any $f_i \in F(\epsilon,M)$. 
    This in turn equals
    \begin{equation}
        \mathbb{E}\left[HX_m^*,\bigcap_{i=0}^{m} \left\{f_1,\dots,f_m\,\text{open}, \overline{f_i} \leftrightarrow \underline{f_{i+1}}\,\text{off}\,\{f_1,\dots,f_i\}\cup\zeta^*_{i-1}\right\}\right].
    \end{equation}
    Since $H$ and the restricted connection events are independent of $\sigma(f_1,\dots,f_m)$, and so is $X_m$, by the comment below \eqref{rv-indep}, we pull out the constant cost to obtain
    \begin{equation}
        \left(\frac{p_c}{1-p_c}\right)^m \cdot \mathbb{E}\left[HX_m^*, \bigcap_{i=0}^{m} \left\{f_1,\dots,f_m\,\text{closed}, \overline{f_i} \leftrightarrow \underline{f_{i+1}}\,\text{off}\,\{f_1,\dots,f_i\}\cup\zeta^*_{i-1}\right\}\right].
    \end{equation}
    The above event equals
    \begin{equation}
        \beta^m  \cdot \mathbb{E}\left[HX_m, \bigcap_{i=0}^{m} \left\{\overline{f_i} \leftrightarrow \underline{f_{i+1}}\,\text{off}\,\zeta_{i-1}\right\}\right],
    \end{equation}
as desired.
\end{proof}

Now, define, for any $1 \leq k \leq m$,
\begin{equation}
    \Lambda_k:=\bigcap_{i=0}^{k} \left\{\overline{f_i} \leftrightarrow \underline{f_{i+1}}\,\text{off}\,\zeta_{i-1}\right\},
\end{equation}
Let $K>0$ and define, for a disjoint union of lattice animals $\mathcal{D}_k \ni \underline{f_k}$, 
\begin{equation}
    \mathcal{L}(\mathcal{D}_k) = \left\{\zeta_{k-1}\cap B(\overline{f_k},K)=\mathcal{D}_k\right\}.
\end{equation}
There are three steps to the recursive argument we use. First,
\begin{lemma}\label{rec1}
    For $k \geq 1$,
    \begin{equation}\label{eqn: rec1}
    \begin{split}
        &n^{(d-2)(k+1)}\cdot \mathbb{E}\left[HX_k,\Lambda_{k}\right]\\
    =~&\frac{\tau(\overline{f_{k}},\underline{f_{k+1}})}{n^{2-d}}
        \cdot \sum_{\mathcal{D}_k} \mathbb{E}_{\nu_{\overline{f_k}}}\left[g(f_k,\mathcal{D}_k),\exists\,\text{a path to}\,\infty\,\text{in}\,W_{\overline{f_k}}\,\text{off}\,\mathcal{D}_k\right]n^{(d-2)k}\mathbb{E}\left[HX_{k-1}, \Lambda_{k-1},\mathcal{L}(\mathcal{D}_{k})\right]\\
        &+\varepsilon_1(n,K).
    \end{split}
    \end{equation}
    Here, the error $\varepsilon_1$ satisfies
    \[\lim_{K\rightarrow\infty}\limsup_{n\rightarrow\infty} \varepsilon_1(n,K)=0,\]
    uniformly in the choice of $k$ edges $f_1,\ldots,f_k\in F(\epsilon,M)$ for fixed $M$.
\end{lemma}
Recall that, for a vertex $v\in \mathbb{Z}^d$,   $W_v$ denotes the incipient infinite cluster under the relevant measure $\nu_v$. 

The next lemma provides an asymptotic expression for the second factor in the sum in \eqref{eqn: rec1}.
\begin{lemma} 
For $k \geq 2$,
     \begin{equation}
         n^{(d-2)k}\cdot\mathbb{E}\left[HX_{k-1},\Lambda_{k-1},\mathcal{L}(\mathcal{D}_{k})\right]
    \end{equation}
    equals
    \begin{equation}
        \begin{split}
        \label{eq:fix53}
         \frac{\tau(\overline{f_{k-1}},\underline{f_{k}})}{n^{2-d}}\cdot &\nu_{\underline{f_k}}(W_{\overline{f_k}} \cap B(\overline{f_k},K)=\mathcal{D}_k)
         \\
         \times\sum_{\mathcal{D}_{k-1}}&\mathbb{E}_{\nu_{\overline{f_{k-1}}}}\left[g(f_{k-1},\mathcal{D}_{k-1}),\exists\,\text{a path to}\,\infty\,\text{in}\,W_{\overline{f_{k-1}}}\,\text{off}\,\mathcal{D}_{k-1}\right] n^{(d-2)(k-1)}\mathbb{E}\left[HX_{k-2},\Lambda_{k-2},\mathcal{L}(\mathcal{D}_{k-1})\right]\\
         &+\varepsilon_2(n,K,\mathcal{D}_k).
        \end{split}
     \end{equation}
     \label{rec2}
    The error $\varepsilon_2$ satisfies
    \[\lim_{K\rightarrow\infty}\limsup_{n\rightarrow\infty} \sum_{\mathcal{D}_k}\varepsilon_2(n,K,\mathcal{D}_k)=0.\]
    The limit is uniform in the choice of edges $f_1,\ldots,f_k\in F(\epsilon,M)$ for fixed $M$.
\end{lemma}
Finally, we have

 \begin{lemma} \label{rec3}
Fix $K,K_H<\infty$, and let $H$ be a bounded random measurable with respect to $\sigma(C_{B(K_H))}$. (Recall the precise definition preceding the statement of Theorem \ref{thm: local-limit}.)  Let $f_1$ range in the far regime:
\[ \varepsilon n \le |\underline{f_1}|\le Mn . \]
For every local realization $\mathcal{D}\subset B(\overline{f_1},K)$,
\begin{equation} \label{eq:fix54}
\sup_{f_1,D_1}
\frac{ \left|
 \mathbb E\bigl[H,\Lambda_0,\mathcal{L}(\mathcal{D}_1)\bigr]
 -
 \mathbb E_{\nu}[H]\,
 \tau(0,\underline{f_1})\,
 \nu_{\underline{f_1}}\!\left(W_{f_1}\cap B(f_1,K)=\mathcal{D}\right)
\right|
}{\tau(0,\underline{f_1})}
        \longrightarrow 0 ,
\end{equation}
where the supremum is over the above far-regime \(f_1\)'s and over the
finitely many possible \(D_1\)'s. 
    \label{rec3}
 In particular, for every \(D_1\) with
positive \(\nu_{f_1}\)-mass,
\[ \mathbb E\bigl[H,\Lambda_0,\mathcal{L}(\mathcal{D})\bigr] = \mathbb E_{\nu}[H]\,
\tau(0,\underline{f_1})\, \nu_{f_1}\left(W_{\underline{f_1}}\cap B(\overline{f_1},K)=D_1\right)+o(\tau(0,\underline{f_1})). \]
Taking $H\equiv 1$, we get
    \begin{equation}
        \mathbb{P}(\Lambda_0,\mathcal{L}(\mathcal{D}_1)) = \tau(\overline{f_0},\underline{f_1})\cdot \nu_{\underline{f_1}}(W_{\underline{f_1}}\cap B(\overline{f}_1,K)=\mathcal{D}_1)+o(\tau(\overline{f_0},\underline{f_1})).
    \end{equation}
\end{lemma}

Putting these together, we get
\begin{proposition}
As $n \rightarrow \infty$,
    \begin{equation}
    \begin{split}
        &\mathbb{E}\left[HX_m,\Lambda_{m}\right]\\
        =~&\mathbb{E}_{\nu}[H]\cdot\left(\mathbb{E}_{\nu_{\mathbf{e}_1}} \otimes \mathbb{E}_{\widetilde{\nu_0}}\left[g((0,\mathbf{e}_1), \widetilde{W}_0),\exists\,\text{a path to}\,\infty\,\text{in}\,W_{\mathbf{e}_1}\,\text{off}\,\widetilde{W}_{0}\right]\right)^m \cdot \prod_{j=0}^{m}\bm{c}|\underline{f_{j+1}}-\overline{f_{j}}|^{2-d}\\
        &+o(n^{(2-d)(m+1)}).
    \end{split}
        \label{thm-eq0}      
    \end{equation}
    Here, $o_n(1)$ denotes a quantity tending to 0 as $n\rightarrow\infty$.
    \label{rec4}
\end{proposition}

\subsubsection{Proof of the far-regime estimate}
\label{sec: far-regime-proof}
Here, we derive the estimate in Proposition \ref{farregime-thm} given Proposition \ref{rec4}.
\begin{proof}[Proof of Proposition \ref{farregime-thm}]
   Applying Lemma \ref{open-closed-lem} and Proposition~\ref{rec4}, we obtain, for sufficiently large $n$, 
    \begin{equation}
    \begin{split}
        &\sum_{F(\epsilon,M)} \mathbb{E}\left[H\prod_{i=1}^{k} g(f_i,\zeta_{i-1}^*), (f_1,\dots,f_k)\,\text{ordered open pivotals for}\, 0 \leftrightarrow n\mathbf{e}_1\right]\\
        =~ &\mathbb{E}_\nu[H]\cdot   \bm{c}^{k+1} (\alpha\beta)^k \sum_{F(\epsilon,M)} \prod_{j=0}^{k} |\underline{f_{j+1}}-\overline{f_j}|^{2-d}+o(n^{(2-d)(k+1)}).
        \end{split}
    \end{equation}
    Rescaling, we find the quantity
    \begin{equation}
    \bm{c}^{k+1} \cdot (\alpha\beta)^k \cdot \mathbb{E}_\nu[H]\cdot (n^{2-d})^{k+1} \cdot n^{kd}\sum_{F(\epsilon,M)} \frac{1}{n^{kd}} \prod_{j=0}^{k} \bigg\lvert \frac{\underline{f_{j+1}}}{n}-\frac{\overline{f_j}}{n}\bigg\rvert^{2-d}+o(n^{(2-d)(k+1)}.
    \end{equation}
    We interpret the sum as a Riemann sum to obtain
    \begin{equation}
   \bm{c}^{k+1} \cdot (\alpha\beta)^k \cdot \mathbb{E}_\nu[H] \cdot (n^{2(k+1)-d}) \cdot  (2d)^k\cdot I_k(\epsilon,M)+o\left(n^{(2-d)(k+1)}\right).
    \end{equation}
    The factor $2d$ comes from each interior vertex appearing $2d$ times in the sum.
    As in the proof of Proposition \ref{nearregime-thm}, dividing by $n^{2(k+1)-d}$ and taking the limit as $n \rightarrow \infty$ completes the proof.
\end{proof}

\subsubsection{Proof of Proposition \ref{rec4}}\label{sec: recursive}
Here, we prove Proposition \ref{rec4}, given Lemma \ref{rec1}, Lemma \ref{rec2} and Lemma \ref{rec3}. To understand the relative size of the terms appearing in this section and the next two, it is useful to note that from \eqref{rv-bdd} and the BK inequality, we have
\begin{equation}
    \label{eqn: baseline}
    \mathbb{E}[HX_k,\Lambda_k]= O(\|H\|_\infty L^{d(k+1)}(n^{2-d})^{k+1}).
\end{equation}
Terms that are asymptotically smaller are error terms.
\begin{proof}
The proof is by induction. The case $m = 1$ can be obtained by combining Lemma \ref{rec1} with Lemma \ref{rec3}. We review it separately in the case $H\equiv 1$ for the benefit of the reader.
We begin by writing
\begin{align*}
\mathbb{E}[X_1,\Lambda_1]&=\mathbb{E}[g(f_1,C(0)), 0\leftrightarrow \underline{f_1}, \overline{f_1}\leftrightarrow n\mathbf{e}_1 \text{ off } C(0)]\\
&=\sum_{\underline{f_1}\in \mathcal{C}}\mathbb{E}[g(f_1,\mathcal{C}),\, C(0)=\mathcal{C}, \overline{f_1}\leftrightarrow n\mathbf{e}_1 \text{ off } \mathcal{C}]\\
&=\sum_{\underline{f_1}\in \mathcal{C}}\mathbb{P}(C(0)=\mathcal{C})\mathbb{E}[g(f_1,\mathcal{C}),\overline{f_1}\leftrightarrow n\mathbf{e}_1 \text{ off } \mathcal{C}]
\end{align*}
Introducing a product space with measure $\mathbb{P}\otimes \tilde{\mathbb{P}}$, we rewrite this as 
\begin{equation}\label{eqn: split}
\tilde{\mathbb{E}}\big[\mathbb{E}[g(f_1,\tilde{C}(0)),\ \overline{f_1}\leftrightarrow n\mathbf{e}_1 \text{ off } \widetilde{C}(0)]\big].
\end{equation}
Here $\tilde{\mathbb{E}}$, represents the percolation measure on one factor of the product space, and $\tilde{C}(0)$ is the cluster of the origin in that factor. Next, we approximate the expectation in \eqref{eqn: split} as
\begin{align*}
&\tilde{\mathbb{E}}\big[\mathbb{E}[g(f_1,\tilde{C}(0)),\ \overline{f_1}\leftrightarrow n\mathbf{e}_1 \text{ off } \widetilde{C}(0)]\big]\\
=~&\tilde{\mathbb{E}}\big[\mathbb{E}[g(f_1,\tilde{C}(0)),\ \overline{f_1}\leftrightarrow n\mathbf{e}_1 \text{ off } B(\overline{f_1},K)\cap \tilde{C}(0)]\big]+ o_{K\rightarrow \infty}(1)\cdot \epsilon^{2-d} L^{2d} (n^{2-d})^2,
\end{align*}
where $o_{K\rightarrow \infty}(1)$ denotes a quantity tending to zero as $K\rightarrow \infty$, uniformly in $f_1$ for fixed $M$. This approximation will be argued in general and in detail in Proposition \ref{prop: E1}. We omit the details at this point and proceed.

Let $K\ge L$. Then, \eqref{rv-cyl} and \eqref{eqn: V-locality} imply that
\begin{align*}
&\mathbb{E}[g(f_1,\tilde{C}(0)),\ \overline{f_1}\leftrightarrow n\mathbf{e}_1 \text{ off } B(\overline{f_1},K)\cap \tilde{C}(0)]\\
=~&\mathbb{E}[g(f_1,B(\overline{f_1},K)\cap \tilde{C}(0)),\ \overline{f_1}\leftrightarrow n\mathbf{e}_1 \text{ off } B(\overline{f_1},K)\cap \tilde{C}(0)].
\end{align*}
We decompose the expectation according to the value of 
\[\mathcal{D}= B(\overline{f_1},K)\cap \tilde{C}(0).\]
Conditioning on $\{\overline{f}_1\leftrightarrow n\mathbf{e}_1\}$, we rewrite \eqref{eqn: split} as
\[\tau(\overline{f}_1,n\mathbf{e}_1)\sum_{\mathcal{D}}\tilde{\mathbb{P}}(\tilde{C}(0)\cap B(\overline{f_1},K)=\mathcal{D},0\lra \underline{f_1})\mathbb{E}[g(f_1, \mathcal{D}), \overline{f}_1\leftrightarrow n\mathbf{e}_1  \text{ off } \mathcal{D}\mid \overline{f}_1\leftrightarrow n\mathbf{e}_1 ].\]
We make one more approximation, replacing the last quantity by
\[\tau(\overline{f}_1,n\mathbf{e}_1)\sum_{\mathcal{D}}\tilde{\mathbb{P}}(\tilde{C}(0)\cap B(\overline{f_1},K)=\mathcal{D},0\lra \underline{f_1})\mathbb{E}[g(f_1,\mathcal{D}), \overline{f}_1\leftrightarrow \partial B(\overline{f_1},2^K)  \text{ off } \mathcal{D}\mid \overline{f}_1\leftrightarrow n\mathbf{e}_1 ]+2^{-K}L^{d}(n^{2-d})^{2}.\]
This approximation is proved in detail in Proposition \ref{prop: E2}.

By convergence to the IIC in Theorem \ref{thm:IICexistences} and \eqref{rv-cyl}, we have
\begin{align*}
&\mathbb{E}[g(f_1,\mathcal{D}), \overline{f}_1\leftrightarrow \partial B(\overline{f_1},2^K)  \text{ off } \mathcal{D}\mid \overline{f}_1\leftrightarrow n\mathbf{e}_1 ]\\
=~&\mathbb{E}_{\nu_{\overline{f_1}}}[g(f_1,\mathcal{D}), \overline{f_1}\leftrightarrow \partial B(\overline{f_1},2^K)  \text{ off } \mathcal{D}]+o_n(1).
\end{align*}
Inserting this into the sum above, we obtain using \eqref{rv-bdd}
\begin{equation}\label{eqn: sum-over-D}
\begin{split}\mathbb{E}[X_1,\Lambda_1]&=\tau(\overline{f}_1,n\mathbf{e}_1)\sum_{\mathcal{D}}\tilde{\mathbb{P}}(\tilde{C}(0)\cap B(\overline{f_1},K)=\mathcal{D}, 0\lra \underline{f_1})\mathbb{E}_{\nu_{\overline{f_1}}}[g(f_1,\mathcal{D}), \overline{f_1}\leftrightarrow \partial B(\overline{f_1},2^K)  \text{ off } \mathcal{D}]\\
&+ L^{d} (n^{2-d})^2(o_n(1)+o_K(1)\epsilon^{2-d}+2^{-K}).
\end{split}
\end{equation}
Resolving the sum over $\mathcal{D}$ in the sum in \eqref{eqn: sum-over-D}, we write the first term in \eqref{eqn: sum-over-D} as
\[\tau(0,\underline{f_1})\tau(\overline{f}_1,n\mathbf{e}_1)\tilde{\mathbb{E}}\big[ \mathbb{E}_{\nu_{\overline{f_1}}}[g(f_1,\tilde{C}(0)\cap B(\overline{f_1},K)), \overline{f_1}\leftrightarrow \partial B(\overline{f_1},2^K)  \text{ in } W_{\overline{f_1}} \text{ off }  \tilde{C}(0)\cap B(\overline{f_1},K)] \mid 0\leftrightarrow \underline{f_1}\big].\]
Since the quantity
\[\mathbb{E}_{\nu_{\overline{f_1}}}[g(f_1,\mathcal{D}), \overline{f_1}\leftrightarrow \partial B(\overline{f_1},2^K)  \text{ off } \mathcal{D}]\]
is a function of the finite set $\mathcal{D}$, \eqref{eqn: IIC-exp}  gives
\begin{align*}
&\tilde{\mathbb{E}}\big[ \mathbb{E}_{\nu_{\overline{f_1}}}[g(f_1,\tilde{C}(0)\cap B(\overline{f_1},K)), \overline{f_1}\leftrightarrow \partial B(\overline{f_1},2^K)  \text{ in } W_{\overline{f_1}} \text{ off }  \tilde{C}(0)\cap B(\overline{f_1},K)] \mid 0\leftrightarrow \underline{f_1}\big]\\
=~& \tilde{\mathbb{E}}_{\nu_{\underline{f_1}}}\big[ \mathbb{E}_{\nu_{\overline{f_1}}}[g(f_1,\widetilde{W}_{\underline{f_1}}), \overline{f_1}\leftrightarrow \partial B(\overline{f_1},2^K)  \text{ in } W_{\overline{f_1}} \text{ off }  \widetilde{W}_{\underline{f_1}}\cap B(\overline{f_1},K)]\big]+o_n(1)\\
=~& \tilde{\mathbb{E}}_{\nu_{\underline{f_1}}}\big[ \mathbb{E}_{\nu_{\overline{f_1}}}[g(f_1,\widetilde{W}_{\underline{f_1}}), \overline{f_1}\leftrightarrow \infty  \text{ in } W_{\overline{f_1}} \text{ off }  \widetilde{W}_{\underline{f_1}}\cap B(\overline{f_1},K)]\big]+o_n(1)+o_K(1),
\end{align*}
for $K$ large.

We have thus shown that the first term on the right in \eqref{eqn: sum-over-D} equals
\[(1+o_K(1)+o_n(1))\cdot \tau(0,\underline{f_1})\tau(\overline{f}_1,n\mathbf{e}_1)\tilde{\mathbb{E}}_{\nu_{\underline{f_1}}}\big[ \mathbb{E}_{\nu_{\overline{f_1}}}[g(f_1,\widetilde{W}_{\underline{f_1}}), \overline{f_1}\leftrightarrow \infty  \text{ in } W_{\overline{f_1}} \text{ off } \widetilde{W}_{\underline{f_1}}]].\]
Combined with \eqref{eqn: HS-scaling}, this gives the result when $m=1$ and $H\equiv 1$.

Now let $m \geq 2$ and assume \eqref{thm-eq0} holds for $m-1$. By Lemma \ref{rec1}, 
\begin{equation}
\begin{split}
    &\mathbb{E}\left[HX_m,\Lambda_m\right]\\
    =&~\tau(\overline{f_m},\underline{f_{m+1}})\cdot \sum_{\mathcal{D}_m}\mathbb{E}_{\nu_{\overline{f_m}}}\left[g(f_m,\mathcal{D}_m), \exists\,\text{a path to}\,\infty\,\text{in}\,W_{\overline{f_m}}\,\text{off}\,\mathcal{D}_m \right]\mathbb{E}\left[HX_{m-1},\Lambda_{m-1},\mathcal{L}(\mathcal{D}_m)\right]\\
    +~&n^{(2-d)(m+1)}o_{K\rightarrow\infty}(1).
    \label{thm-eq1}
    \end{split}
\end{equation}
We treat the second expectation $\mathbb{E}$ as a separate, independent copy and expand it using Lemma \ref{rec2}. Then \eqref{thm-eq1} equals
\begin{equation}
   \tau(\overline{f_{m-1}},\underline{f_m})\tau(\overline{f_m},\underline{f_{m+1}})\sum_{\mathcal{D}_m}\mathbb{E}_{\nu_{\overline{f_m}}}\left[g(f_m,\mathcal{D}_m),\exists\,\text{a path to}\,\infty\,\text{in}\,W_{\overline{f_m}}\,\text{off}\,\mathcal{D}_m\right]\alpha(\mathcal{D}_m)
    \label{thm-eq-2}
\end{equation}
where $\alpha(\mathcal{D}_m)$ equals
\begin{equation}
    \begin{split}
    &\widetilde{\nu}_{\underline{f_m}}(\mathcal{D}_m=\widetilde{W}_{\underline{f_m}} \cap B(\overline{f_m},K) )\\
    \times&\sum_{\mathcal{D}_{m-1}}\mathbb{E}_{\widetilde{\nu}_{\overline{f_{m-1}}}}\left[g(f_{m-1},\mathcal{D}_{m-1}), \exists\,\text{a path to}\,\infty\,\text{in}\,W_{\overline{f_{m-1}}}\,\text{off}\,\mathcal{D}_{m-1}\right]\widetilde{\mathbb{E}}\left[HX_{m-2},\Lambda_{m-2},L(\mathcal{D}_{m-1})\right]\\
    \\
   &+~n^{(d-2)m}\varepsilon(\mathcal{D}_m,K),
    \end{split}
    \label{thm-eq-3}
\end{equation}
and
\[\lim_{K\rightarrow\infty} \limsup_{n\rightarrow\infty}\sum_{\mathcal{D}_m}\varepsilon(\mathcal{D}_m,K)=0.\]
By Lemma \ref{rec1}, \eqref{thm-eq-2} can be rewritten as
\begin{equation}
\begin{split}
    &\mathbb{E}\left[HX_{m-1},\Lambda_{m-1}\right]\tau(\overline{f_m},\underline{f_{m+1}})\\
\times&\sum_{\mathcal{D}_m}\mathbb{E}_{\nu_{\overline{f_m}}}\left[g(f_m,\mathcal{D}_m),\exists\,\text{a path to}\,\infty\,\text{in}\,W_{\overline{f_m}}\,\text{off}\,\mathcal{D}_m\right]\widetilde{\nu}_{\underline{f_m}}(\mathcal{D}_m  = \widetilde{W}_{\underline{f_m}} \cap B(\overline{f_m},K))
    \end{split}
    \label{thm-eq-4}
\end{equation}
with an error
\begin{equation}\label{eqn: error-m+1}
n^{(2-d)(m+1)}L^d \|H\|_{\infty} \cdot (o_K(1)+o_n(1)).
\end{equation}
The $o_K(1)$ is uniform in $n$. We have used \eqref{eqn: HS-scaling} and \eqref{rv-bdd} for the error term.

Resolving the partition in the main term \eqref{thm-eq-4}, we obtain
\begin{align*}&\mathbb{E}\left[HX_{m-1},\Lambda_{m-1}\right]\tau(\overline{f_m},\underline{f_{m+1}})\\
\times~&\tilde{\mathbb{E}}_{\nu_{\underline{f_m}}}\left[\mathbb{E}_{\nu_{\overline{f_m}}}\left[g(f_m,\widetilde{W}_{\underline{f_m}} \cap B(\overline{f_m},K)),\,\exists\,\text{a path to}\,\infty\,\text{in}\,W_{\overline{f_m}}\,\text{off}\,\widetilde{W}_{\underline{f_m}} \cap B(\overline{f_m},K)\right]\right].
\end{align*}
By \eqref{eqn: V-locality} we have, for $K\ge L$ 
\begin{align*}
&\mathbb{E}_{\nu_{\overline{f_m}}}\left[g(f_m,\widetilde{W}_{\underline{f_m}} \cap B(\overline{f_m},K)),\,\exists\,\text{a path to}\,\infty\,\text{in}\,W_{\overline{f_m}}\,\text{off}\,\widetilde{W}_{\underline{f_m}} \cap B(\overline{f_m},K)\right]\\
=~&\mathbb{E}_{\nu_{\overline{f_m}}}\left[g(f_m,\widetilde{W}_{\underline{f_m}}),\,\exists\,\text{a path to}\,\infty\,\text{in}\,W_{\overline{f_m}}\,\text{off}\,\widetilde{W}_{\underline{f_m}} \cap B(\overline{f_m},K)\right]\\
=~&(1+o_K(1)+o_n(1))\mathbb{E}_{\nu_{\overline{f_m}}}\left[g(f_m,\widetilde{W}_{\underline{f_m}}),\,\exists\,\text{a path to}\,\infty\,\text{in}\,W_{\overline{f_m}}\,\text{off}\,\widetilde{W}_{\underline{f_m}} \right]
\end{align*}

We use \eqref{rv-trans} to rewrite the main term as an expectation with respect to product of IIC measures
\begin{equation}
    \mathbb{E}\left[HX_{m-1},\Lambda_{m-1}\right]\cdot \tau(\overline{f_m},\underline{f_{m+1}}) \cdot \mathbb{E}_{\nu_{\mathbf{e}_1}} \otimes \mathbb{E}_{\widetilde{\nu_{0}}}\left[g(\mathbf{e}_1,\widetilde{W}_0),\exists\,\text{a path to}\,\infty\,\text{in}\,W_{\mathbf{e}_1}\,\text{off}\,\widetilde{W}_0\right].
    \label{thm-eq-5}
\end{equation}
Since $|\underline{f_{k+1}}-\overline{f_k}| > \epsilon n$ for all $k$, we may apply the two-point asymptotic \eqref{eqn: two-pt}, giving
\begin{equation}
    \begin{split}&\mathbb{E}\left[HX_m,\Lambda_m\right]\\
    =~&\mathbb{E}\left[HX_{m-1},\Lambda_{m-1}\right] \cdot (1+o_n(1)+o_K(1)) \\
    &\times \mathbb{E}_{\nu_{\mathbf{e}_1}} \otimes \mathbb{E}_{\widetilde{\nu_{0}}}\left[g(\mathbf{e}_1,\widetilde{W}_0),\exists\,\text{a path to}\,\infty\,\text{in}\,W_{\mathbf{e}_1}\,\text{off}\,\widetilde{W}_0\right] \cdot \bm{c}|\underline{f_{m+1}}-\overline{f_m}|^{2-d}\\
    +~&n^{(2-d)(m+1)}L^d \cdot (o_K(1)+o_n(1)).
    \end{split}
\end{equation}
Normalizing by $n^{(2-d)(m+1)}$ and taking the $n\rightarrow\infty$, and then the  $K\rightarrow\infty$ limits completes the proof.
\end{proof}

\subsubsection{Leading order terms}\label{sec: leading}
In this section, we isolate the leading order terms in \eqref{eqn: rec1}, \eqref{eq:fix53}, and \eqref{eq:fix54}; the error terms \eqref{eqn: E1def}, \eqref{eqn: E2def}, \eqref{eqn: E3def}, \eqref{eqn: E4def}, \eqref{eqn: E5def} and \eqref{eqn: E6def} can be neglected. These error terms are then treated separately, in Section \ref{sec: LOT}.
\begin{proof}[Leading order term for Lemma \ref{rec1}]
First partition along admissible configurations up to $\underline{f_{k-1}}$. 
\begin{equation}
    \mathbb{E}[HX_k,\Lambda_{k}] = \sum_{\mathcal{C}} \mathbb{E}\left[g(f_k,\mathcal{C}) H X_{k-1},\ \Lambda_{k-1},  \zeta_{k-1}=\mathcal{C}, \overline{f_{k}} \leftrightarrow \underline{f_{k+1}}\,\text{off}\,\mathcal{C}\right].
    \label{lem-rec1-1}
\end{equation}
Since $\Lambda_{k-1}$ depends only on bonds with at least one site in $\mathcal{C}$, we rewrite the sum as 

\begin{equation}
    \sum_{\mathcal{C}} \mathbb{E}\left[HX_{k-1},\Lambda_{k-1},  \zeta_{k-1}=\mathcal{C}\right]\mathbb{E}[g(f_k,\mathcal{C}),\overline{f_{k}} \leftrightarrow\underline{f_{k+1}}\,\text{off}\,\mathcal{C}].
    \label{lem-rec1-2}
\end{equation}
Unfold into two copies as in \eqref{eqn: split}:
\begin{equation}
    \sum_{\mathcal{C}} \widetilde{\mathbb{E}}\left[HX_{k-1}\mathbbm{1}_{\Lambda_{k-1}}\mathbbm{1}_{\widetilde{\zeta}_{k-1}=\mathcal{C}}\mathbb{E}[g(f_k,\mathcal{C}), \overline{f_k} \leftrightarrow\underline{f_{k+1}}\,\text{off}\,\mathcal{C}]\right].
    \label{lem-rec1-3}
\end{equation}
Resolve the partition and upper bound \eqref{lem-rec1-3} by
\begin{equation}
    \widetilde{\mathbb{E}}\left[HX_{k-1}\mathbbm{1}_{\Lambda_{k-1}}\mathbb{E}\big[g(f_k,\widetilde{\zeta}_{k-1}),\overline{f_k}\leftrightarrow\underline{f_{k+1}}\,\text{off}\, \widetilde{\zeta}_{k-1} \cap B(\overline{f_k},K)\big]\right].
    \label{lem-rec1-4}
\end{equation}
We prove in Proposition \ref{prop: E1} that the difference
\begin{equation}\label{eqn: E1def}
E_1:=\eqref{lem-rec1-4} - \eqref{lem-rec1-3}
\end{equation}
is of lower order than the main term \eqref{eqn: rec1}:
\[\lim_{K\rightarrow\infty}\limsup_{n\rightarrow\infty} \frac{E_1}{(n^{2-d})^{k+1}}=0.\]
Recall from \eqref{eqn: baseline} that the main term in \eqref{lem-rec1-1} is order $n^{(2-d)(k+1)}$.

We now partition \eqref{lem-rec1-4} over admissible $\mathcal{D}_k\subset B(\overline{f_k},K)$:
\begin{equation}
    \sum_{\mathcal{D}_k} \widetilde{\mathbb{E}}[HX_{k-1}\mathbbm{1}_{\Lambda_{k-1}}\mathbbm{1}_{\mathcal{L}(\mathcal{D}_k)}]\mathbb{E}[g(f_k,\mathcal{D}_k),\overline{f_k}\leftrightarrow\underline{f_{k+1}}\,\text{off}\,\mathcal{D}_k].
    \label{lem1-rec1-5}
\end{equation}
Condition on the connection event $\{\overline{f_k}\leftrightarrow\underline{f_{k+1}}\}$:
\begin{equation}
    \tau(\overline{f_k},\underline{f_{k+1}}) \sum_{\mathcal{D}_k} \mathbb{E}[g(f_k,\mathcal{D}_k),\overline{f_k} \leftrightarrow \underline{f_{k+1}}\,\text{off}\,\mathcal{D}_k \mid \overline{f_k} \leftrightarrow \underline{f_{k+1}}] \widetilde{\mathbb{E}}[HX_{k-1}\mathbbm{1}_{\Lambda_{k-1}}, \mathcal{L}(\mathcal{D}_k)].
    \label{lem-rec1-6}
\end{equation}
For $R\ge K$, this is bounded above by
\begin{equation}
    \tau(\overline{f_k},\underline{f_{k+1}}) \sum_{\mathcal{D}_k} \mathbb{E}[g(f_k, \mathcal{D}_k),\overline{f_k} \leftrightarrow \partial B(f_k;2^R)\,\text{off}\,\mathcal{D}_k \mid \overline{f_k} \leftrightarrow \underline{f_{k+1}}] \widetilde{\mathbb{E}}[HX_{k-1},\Lambda_{k-1}\cap \mathcal{L}(\mathcal{D}_k)],
    \label{lem-rec1-7}
\end{equation}
with the difference being given by
\begin{equation}\label{eqn: E2def}
E_2:=\eqref{lem-rec1-7}-\eqref{lem-rec1-6}.
\end{equation}
We show in Proposition \ref{prop: E2} that this is a lower order term when $R\rightarrow \infty$
\[\lim_{R\rightarrow\infty} \limsup_{n\rightarrow\infty}\frac{E_2}{(n^{2-d})^{k+1}}=0, \]
uniformly in $K$.
For $n$ sufficiently large and $K\le R$ fixed, we approximate the conditional expectation in \eqref{lem-rec1-7} using the IIC limit:
\begin{equation}
    (1+o_{n\rightarrow\infty}(1))\cdot \tau(\overline{f_k},\underline{f_{k+1}}) \sum_{\mathcal{D}_k} \mathbb{E}_{\nu_{\overline{f_k}}}[g(f_k,\mathcal{D}_k),\overline{f_k}\leftrightarrow\partial B(f_k;2^R)\,\text{off}\,\mathcal{D}_k]\widetilde{\mathbb{E}}[HX_{k-1},\Lambda_{k-1}\cap  \mathcal{L}(\mathcal{D}_k)].
    \label{lem-rec1-8}
\end{equation}
Now let $R$ and then $K$ be sufficiently large to obtain the desired
\begin{equation}
    (1+o_{K}(1)+o_n(1))\tau(\overline{f_k},\underline{f_{k+1}}) \sum_{\mathcal{D}_k} \mathbb{E}_{\nu_{\overline{f_k}}}[g(f_k,\mathcal{D}_k),\exists\,\text{a path to}\,\infty\,\text{in}\,W_{\overline{f}_k}\,\text{off}\,\mathcal{D}_k]\widetilde{\mathbb{E}}[HX_{k-1},\Lambda_{k-1}\cap \mathcal{L}(\mathcal{D}_k)],
    \label{lem-rec1-9}
\end{equation}
where $W_{\overline{f_k}}$ is the IIC based at $\overline{f_k}$. 
\end{proof}

\begin{proof}[Leading order term for Lemma \ref{rec2}]
Here, we use the stronger version of the IIC convergence in Theorem \ref{thm:IICexistences} including avoidance of a far-away set $D_n$. First, write 
\begin{equation}
    \mathbb{E}[HX_{k-1},\Lambda_{k-1},\mathcal{L}(\mathcal{D}_k)] = \mathbb{E}\left[HX_{k-2}\cdot g(f_{k-1},\zeta_{k-2}), \Lambda_{k-2}, \overline{f_{k-1}} \leftrightarrow \underline{f_{k}}\,\text{off}\,\zeta_{k-2},\mathcal{D}_k = \zeta_{k-1}\cap B(\overline{f_k},K)  \right].
    \label{lem-rec2-1}
\end{equation}
Partitioning over admissible values $\mathcal{C}$ of the $\zeta_{k-2}$ cluster, the expectation \eqref{lem-rec2-1} equals
\begin{equation}
    \sum_{\mathcal{C}}  \mathbb{E}\left[HX_{k-2} \cdot g(f_{k-1},\mathcal{C}),   \zeta_{k-2}=\mathcal{C}, \Lambda_{k-2}, \overline{f_{k-1}} \leftrightarrow \underline{f_{k}}\,\text{off}\,\mathcal{C}, \mathcal{D}_k =  \zeta_{k-1} \cap B(\overline{f_k},K) \right].
    \label{lem-rec2-2}
\end{equation}
Note that $\mathcal{C} \cap \mathcal{D}_k = \varnothing$. If this were false, then $\overline{f_{k-1}}$ would be a site of $\mathcal{C}$, violating the condition
\[\{\overline{f_{k-1}} \leftrightarrow \underline{f_k}\,\text{off}\,\mathcal{C}\},\] and making the above probability zero. We cannot in general make the stronger statement 
\begin{equation}\label{eqn: intersect}
(\mathcal{C}\cup\partial\mathcal{C}) \cap (\mathcal{D}_k \cup \partial\mathcal{D}_k) = \varnothing.
\end{equation}
However, if \eqref{eqn: intersect} holds, the events
\[\{\mathcal{C} =  \zeta_{k-2}\} \text{ and } \{\overline{f_{k-1}} \leftrightarrow \underline{f_{k}}\,\text{off}\,\mathcal{C}, \mathcal{D}_k = B(\overline{f_k},K) \cap \zeta_{k-1}\}\]
are independent. Letting $\mathcal{S}(\mathcal{D}_k)$ be the collection of $\mathcal{C}$ that satisfy \eqref{eqn: intersect}, we have
\begin{equation}
\begin{split} 
&\sum_{\mathcal{C}\in\mathcal{S}(\mathcal{D}_k)}  \mathbb{E}\left[HX_{k-2} \cdot g(f_{k-1},\mathcal{C}),   \zeta_{k-2}=\mathcal{C}, \overline{f_{k-1}} \leftrightarrow \underline{f_{k}}\,\text{off}\,\mathcal{C}, \mathcal{D}_k = B(\overline{f_k},K) \cap \zeta_{k-1} \right]\\
    =&\sum_{\mathcal{C}\in \mathcal{S}(\mathcal{D}_k)} \mathbb{E}\big[HX_{k-2}, \zeta_{k-2}=\mathcal{C}  \big]\mathbb{E}[g(f_{k-1},\mathcal{C}),\overline{f_{k-1}} \leftrightarrow \underline{f_{k}}\,\text{off}\,\mathcal{C}, \mathcal{D}_k = B(\overline{f_k},K) \cap \zeta_{k-1}].
\end{split}
    \label{lem-rec2-3}
\end{equation}
The difference
\begin{equation}\label{eqn: E3def}
    E_3=E_3(\mathcal{D}_k):= \eqref{lem-rec2-3}-\eqref{lem-rec2-2},
\end{equation}
is a sum over $\mathcal{C}\in \mathcal{S}(\mathcal{D}_k)^c$. That is, $\mathcal{C}$ such that $(\mathcal{C}\cup\partial\mathcal{C}) \cap (\mathcal{D}_k \cup \partial\mathcal{D}_k) \neq \varnothing$. This sum is a lower order term, as shown in Proposition \ref{prop: E3}:
\[\lim_{n\rightarrow\infty} \frac{E_3}{(n^{2-d})^k}=0,\]
uniformly in $\mathcal{D}_k$, whereas the main term in \eqref{lem-rec2-2} is of order $(n^{2-d})^k$.

We unfold \eqref{lem-rec2-3} into an expectation with respect to a product measure $\mathbb{P}\otimes \tilde{\mathbb{P}}$: 
\begin{equation}
    \widetilde{\mathbb{E}}\left[\sum_{\mathcal{C}\in\mathcal{S}(\mathcal{D}_k)}HX_{k-2}\mathbbm{1}_{\mathcal{C}=\widetilde{\zeta}_{k-2}}\mathbb{E}\left(g(f_{k-1},\mathcal{C}),\overline{f_{k-1}}\leftrightarrow\underline{f_k}\,\text{off}\,\widetilde{\zeta}_{k-2},\mathcal{D}_k=B(\overline{f_k},K) \cap \zeta_{k-1}\right)\right].
    \label{lem-rec2-3pt5}
\end{equation}
We now replace the sum over $\mathcal{S}(\mathcal{D}_k)$ with a sum over all $\mathcal{C}$ admissible that are also compatible with the event $\Lambda_{k-2}$. Resolving the partition, we obtain: 
\begin{equation}
    \widetilde{\mathbb{E}}\left[HX_{k-2}\mathbbm{1}_{\Lambda_{k-2}}\mathbb{E}\Big(g(f_{k-1},\tilde{\zeta}_{k-2}),\overline{f_{k-1}}\leftrightarrow\underline{f_k}\,\text{off}\,\widetilde{\zeta}_{k-2},\mathcal{D}_k=B(\overline{f_k},K) \cap \zeta_{k-1}\Big)\right].
    \label{lem-rec2-4}
\end{equation}
Using \eqref{rv-bdd}, the error incurred in the replacement is bounded  in absolute value by 
\begin{equation}\label{eqn: E4def}
E_4:=L^{(k-1)d}  \|H\|_{\infty} \cdot \widetilde{\mathbb{E}}\left[\mathbbm{1}_{\Lambda_{k-2}}, (\tilde{\zeta}_{k-2}\cup \partial \tilde{\zeta}_{k-2})\cap (\mathcal{D}_k\cup \partial \mathcal{D}_k) \neq \varnothing ,\mathbb{P}\left(\overline{f_{k-1}}\leftrightarrow\underline{f_k}\,\text{off}\,\widetilde{\zeta}_{k-2},\mathcal{D}_k=B(\overline{f_k},K) \cap \zeta_{k-1}\right)\right].
\end{equation}
This is shown to be of lower order than the main term in Proposition \ref{prop: E4}:
\[\lim_{n\rightarrow\infty} \frac{E_4}{(n^{2-d})^k}=0.\]
The quantity \eqref{lem-rec2-4} is bounded above by
\begin{equation}
\widetilde{\mathbb{E}}\left[HX_{k-2}\mathbbm{1}_{\Lambda_{k-2}}\mathbb{E}\left(g(f_{k-1},\tilde{\zeta}_{k-2}),\overline{f_{k-1}}\leftrightarrow\underline{f_k}\,\text{off}\,\widetilde{\zeta}_{k-2}\cap B(f_{k-1};K),\mathcal{D}_k=B(\overline{f_k},K) \cap \zeta_{k-1}\right)\right].
    \label{lem-rec2-5}
\end{equation}
We denote 
\begin{equation}\label{eqn: E5def}
E_5:=\eqref{lem-rec2-5}-\eqref{lem-rec2-4},
\end{equation}
and show in Proposition \ref{prop: E5} below that $E_5$ is an error term:
\[\lim_{n\rightarrow \infty}\frac{E_5}{(n^{2-d})^{k}}=0.\]

Next, we truncate \eqref{lem-rec2-5}  to
\begin{equation}\label{eqn: lady-godiva}
\widetilde{\mathbb{E}}\left[HX_{k-2}\mathbbm{1}_{\Lambda_{k-2}}\mathbb{E}\left(g(f_{k-1},\tilde{\zeta}_{k-2}),\overline{f_{k-1}}\leftrightarrow\underline{f_k}\,\text{off}\,\widetilde{\zeta}_{k-2}\cap B(f_{k-1},K),\mathcal{D}_k=B(\overline{f_k},K) \cap C(\underline{f_k})\right)\right].
\end{equation}
In Proposition \ref{prop: only-W}, we show that 
\begin{equation}\label{eqn: E6def}
E_6:=\eqref{lem-rec2-5}-\eqref{eqn: lady-godiva}
\end{equation}
is an error:
\[\lim_{n\rightarrow\infty}  \frac{E_6}{(n^{2-d})^{k}}=0.\]

Finally, we approximate \eqref{eqn: lady-godiva} by the quantity:
\begin{equation}\label{eqn: girard}
\widetilde{\mathbb{E}}\left[HX_{k-2}\mathbbm{1}_{\Lambda_{k-2}}\mathbb{E}\left(g(f_{k-1},\tilde{\zeta}_{k-2}),\overline{f_{k-1}}\leftrightarrow\underline{f_k}\,\text{off}\,\widetilde{\zeta}_{k-2}\cap B(f_{k-1};K),\mathcal{D}_k=B(\overline{f_k},K) \cap C_{B(\underline{f_k},K^d)}(\underline{f_k})\right)\right].
\end{equation}
Here, as in Definition \ref{def: local-CA}, the notation $C_A(u)$ denotes the cluster of a vertex $u$ inside the set $A$.

It is shown in Proposition \ref{prop: localize} that
\begin{equation}\label{eqn: E7def} E_7=E_7(\mathcal{D}_k):=|\eqref{eqn: girard}-\eqref{eqn: lady-godiva}|
\end{equation}
is an error term when summed over $\mathcal{D}_k$:
\[\lim_{K\rightarrow\infty}  \limsup_{n\rightarrow\infty}  \frac{1}{(n^{2-d})^k}\sum_{\mathcal{D}_k}E_7(\mathcal{D}_k)=0.\]

Letting $K\ge L$, we partition \eqref{eqn: girard} according to the realizations of the clusters, using \eqref{eqn: V-locality}:
\begin{equation}
    \sum_{\mathcal{D}_{k-1}}\widetilde{\mathbb{E}}\left[HX_{k-2}\mathbbm{1}_{\Lambda_{k-2}}\mathbbm{1}_{\mathcal{L}(\mathcal{D}_{k-1})}\mathbb{E}\left(g(f_{k-1},\mathcal{D}_{k-1}),\overline{f_{k-1}}\leftrightarrow\underline{f_k}\,\text{off}\,\mathcal{D}_{k-1},\mathcal{D}_k=B(\overline{f_k},K) \cap  C_{B(\underline{f_k},K^d)}(\underline{f_k})\right)\right].
    \label{lem-rec2-6}
\end{equation}
We partition the inner expectation over the value of $g$:
\begin{align*}
&\mathbb{E}\left(g(f_{k-1},\mathcal{D}_{k-1}),\overline{f_{k-1}}\leftrightarrow\underline{f_k}\,\text{off}\,\mathcal{D}_{k-1},\mathcal{D}_k=B(\overline{f_k},K) \cap  C_{B(\underline{f_k},K^d)}(\underline{f_k})\right)\\
=& \sum_m m\cdot \mathbb{P}\left( \  g(f_{k-1},\mathcal{D}_{k-1}) = m,\ \overline{f_{k-1}}\leftrightarrow\underline{f_k}\,\text{off}\,\mathcal{D}_{k-1},\mathcal{D}_k=B(\overline{f_k},K) \cap C_{B(\underline{f_k},K^d)}(\underline{f_k})\right).
\end{align*}

Note that the sum over $m$ is finite. We now claim the following identity, to be proved below:
\begin{equation}\label{eqn: g-m-identity}
\begin{split}
&\mathbb{P}\left(\mathcal{D}_k=B(\overline{f_k},K) \cap  C_{B(\underline{f_k},K^d)}(\underline{f_k}), \  g(f_{k-1},\mathcal{D}_{k-1}) = m,\ \overline{f_{k-1}}\leftrightarrow\underline{f_k}\,\text{off}\,\mathcal{D}_{k-1}\right)\\
=&~(1+o_n(1))\cdot \nu_{\underline{f_k}}\left(\mathcal{D}_k = B(\overline{f_k},K)\cap W_{\underline{f_k}
}\right)
      \mathbb{P}(g(f_{k-1}, \mathcal{D}_{k-1})=m, \overline{f_{k-1}} \leftrightarrow\underline{f_k}\,\text{off}\,\mathcal{D}_{k-1}).
      \end{split}
\end{equation}

Inserting this into \eqref{lem-rec2-6}, we obtain
\begin{equation}
    (1+o_n(1)) \|H\|_{\infty}\cdot \nu_{\underline{f_k}}\left(\mathcal{D}_k = B(\overline{f_k},K)\cap W_{\underline{f_k}}\right)\sum_{\mathcal{D}_{k-1}} \mathbb{E}(g(f_{k-1}),\overline{f_{k-1}}\leftrightarrow\underline{f_k}\,\text{off}\,\mathcal{D}_{k-1})\widetilde{\mathbb{E}}(X_{k-2},\Lambda_{k-2},L(\mathcal{D}_{k-1})).
    \label{lem-rec2-8}
\end{equation}
Conditioning on the unconstrained event $\{\overline{f_{k-1}}\leftrightarrow \underline{f_k}\}$, we rewrite this as
\begin{equation}
\begin{split}
    (1+o_n(1)) \|H\|_{\infty}&\tau(\overline{f_{k-1}},\underline{f_k})\nu_{\underline{f_k}}\left(\mathcal{D}_k = B(\overline{f_k},K)\cap W_{\underline{f_k}}\right)\\
    \times& \sum_{\mathcal{D}_{k-1}} \mathbb{E}(g(f_{k-1}),\overline{f_{k-1}}\leftrightarrow\underline{f_k}\,\text{off}\,\mathcal{D}_{k-1} \mid \overline{f_{k-1}}\leftrightarrow\underline{f_k})\widetilde{\mathbb{E}}(X_{k-2},\Lambda_{k-2},L(\mathcal{D}_{k-1})).
\end{split}
    \label{lem-rec2-8o}
\end{equation}
For $R\ge K$, we upper bound the sum by
\begin{equation}
    (1+o_n(1)) \|H\|_{\infty} \sum_{\mathcal{D}_{k-1}} \mathbb{E}(g(f_{k-1}),\overline{f_{k-1}}\leftrightarrow\partial B(f_{k-1};2^R)\,\text{off}\,\mathcal{D}_{k-1} \mid \overline{f_{k-1}}\leftrightarrow\underline{f_k})\widetilde{\mathbb{E}}(X_{k-2},\Lambda_{k-2},L(\mathcal{D}_{k-1})).
    \label{lem-rec2-9}
\end{equation}
Up to factor $L^{(k-1)d}$ resulting from \eqref{rv-bdd}, the difference between the sum over $\mathcal{D}_{k-1}$ in \eqref{lem-rec2-9} and the same in \eqref{lem-rec2-8o} (ignoring the factors outside the sum) is bounded in absolute value by a lower order term, which we denote by $E_8$:
\begin{equation}\label{eqn: E8def}
\begin{split}
E_8&:=\sum_{\mathcal{D}_{k-1}} \Delta_k\, \widetilde{\mathbb{P}}(\Lambda_{k-2},L(\mathcal{D}_{k-1})),\\
\Delta_k&:=\mathbb{P}(\overline{f_{k-1}}\leftrightarrow\partial B(f_{k-1};2^R)\,\text{off}\,\mathcal{D}_{k-1} \mid \overline{f_{k-1}}\leftrightarrow\underline{f_k})\\
&~-\mathbb{P}(\overline{f_{k-1}}\leftrightarrow\underline{f_k}\,\text{off}\,\mathcal{D}_{k-1} \mid \overline{f_{k-1}}\leftrightarrow\underline{f_k}).
\end{split}
\end{equation}
We show in Proposition \ref{prop: E8} that $E_8$ is an error term:
\[\lim_{R\rightarrow\infty} \limsup_{n\rightarrow\infty} \frac{E_8}{(n^{2-d})^{k-1}}=0.\]

Using the IIC approximation \eqref{eqn: IIC-exp} in \eqref{lem-rec2-9}, we have
\begin{equation}
    (1+o_n(1))\sum_{\mathcal{D}_{k-1}} \mathbb{E}_{\nu_{\overline{f_{k-1}}}}(g(f_{k-1}),\overline{f_{k-1}}\leftrightarrow\partial B(f_{k-1};2^R)\,\text{off}\,\mathcal{D}_{k-1})\widetilde{\mathbb{E}}(HX_{k-2},\Lambda_{k-2},L(\mathcal{D}_{k-1})).
    \label{lem-rec2-10}
\end{equation}
Sending $R$ to infinity, we find that \eqref{lem-rec2-9} equals
\begin{equation}
    (1+o_n(1))\sum_{\mathcal{D}_{k-1}} \mathbb{E}_{\nu_{\overline{f_{k-1}}}}(g(f_{k-1}),\exists\,\text{a path to}\,\infty\,\text{in}\,W_{\overline{f_{k-1}}}\,\text{off}\,\mathcal{D}_{k-1})\widetilde{\mathbb{E}}(HX_{k-2},\Lambda_{k-2},L(\mathcal{D}_{k-1})).
    \label{lem-rec2-11}
\end{equation}
This completes the proof, up to the derivation of the identity \eqref{eqn: g-m-identity}.
\end{proof}

\begin{proof}[Proof of \eqref{eqn: g-m-identity}]
We first note that items (i) and (ii) in Definition \ref{def: local-var} guarantee that $g$ is measurable with respect to the restricted cluster 
\[\mathcal{C}_{k-1}:=C^{\mathcal{D}_{k-1} \cup (\mathbb{Z}^d\setminus B(\overline{f}_{k-1}, L))}(\overline f_{k-1}).\] This is the portion of the cluster of $\overline f_{k-1}$ reachable via only paths lying entirely in $B(f_{k-1}, L)$, and moreover avoiding $\mathcal{D}_{k-1}$.

Let us define a new variable $X$ to be the vector of open/closed statuses of edges which have an endpoint in $C^{\mathcal{D}_{k-1} \cup (\mathbb{Z}^d\setminus B(\overline{f}_{k-1}, L))}(\overline f_{k-1})$. When we explore the value of $X$, what we are exploring is:
\begin{itemize}
    \item edges internal to the cluster $\mathcal{C}_{k-1}$, which connect two vertices of this cluster;
    \item closed edges between a vertex of the cluster $\mathcal{C}_{k-1}$
    and a vertex which lies outside of this cluster but within $B(\overline{f}_{k-1}, L)$. These edges make up the boundary of $\mathcal{C}_{k-1}$ within $B(\overline{f}_{k-1}, L)$;
    \item exterior edges, with one endpoint in the cluster $\mathcal{C}_{k-1}$ and one endpoint outside $B(\overline{f}_{k-1}, L)$.
\end{itemize}
Let $Y$ be the set of vertices $v \notin B(\overline{f}_{k-1}, L)$ which are an endpoint of an open exterior edge. For each $m$, the occurrence (or non-occurrence) of $\{g(f_{k-1}) = m\}$ is determined by the value of $X$. We decompose based on the value $x$ of $X$, writing $Y(x)$ for the value of $Y$ when $X = x$ and $\mathcal{C}_{k-1}(x)$ similarly for the value of this cluster. Then, for $\epsilon n>K^d$, we have
\begin{align*}
    &\mathbb{P}\left(\mathcal{D}_k=B(\overline{f_k},K) \cap C_{B(\underline{f_k},K^d)}(\underline{f_k}), \  g(f_{k-1}) = m,\ \overline{f_{k-1}}\leftrightarrow\underline{f_k}\,\text{off}\,\mathcal{D}_{k-1}\right)\\
    =&\sum_{x: g(f_{k-1}) = m} \mathbb{P}\left(\mathcal{D}_k=B(\overline{f_k},K) \cap C_{B(\underline{f_k},K^d)}(\underline{f_k}), \ \overline{f_{k-1}}\leftrightarrow\underline{f_k}\,\text{off}\,\mathcal{D}_{k-1}, \, X = x\right)\\
     =&\sum_{x: \, g(f_{k-1}) = m} \mathbb{P}(X = x)\mathbb{P}\left(\mathcal{D}_k=B(\overline{f_k},K) \cap C_{B(\underline{f_k},K^d)}(\underline{f_k}), \ \overline{f_{k-1}}\leftrightarrow\underline{f_k}\,\text{off}\,\mathcal{D}_{k-1}\  \mid\  X = x\right)\\
     =&\sum_{x: \, g(f_{k-1}) = m} \mathbb{P}(X = x)\mathbb{P}\left(\mathcal{D}_k=B(\overline{f_k},K) \cap C_{B(\underline{f_k},K^d)}(\underline{f_k}), \ Y(x) \leftrightarrow\underline{f_k}\,\text{off}\,\mathcal{D}_{k-1} \cup \mathcal{C}_{k-1}(x) \  \mid\  X = x\right)\\
     =&\sum_{x: \, g(f_{k-1}) = m} \mathbb{P}(X = x)\mathbb{P}\left(\mathcal{D}_k=B(\overline{f_k},K) \cap C_{B(\underline{f_k},K^d)}(\underline{f_k}), \ Y(x) \leftrightarrow\underline{f_k}\,\text{off}\,\mathcal{D}_{k-1} \cup \mathcal{C}_{k-1}(x)\right)\\
      =&\sum_{x: \, g(f_{k-1}) = m} \mathbb{P}(X = x)\mathbb{P}\left(\mathcal{D}_k=B(\overline{f_k},K) \cap C_{B(\underline{f_k},K^d)}(\underline{f_k}) \mid  Y(x) \leftrightarrow\underline{f_k}\,\text{off}\,\mathcal{D}_{k-1} \cup \mathcal{C}_{k-1}(x)\right)\\
      &\qquad \quad  \times  \mathbb{P}\left( Y(x) \leftrightarrow\underline{f_k}\,\text{off}\,\mathcal{D}_{k-1} \cup \mathcal{C}_{k-1}(x)\right)\ .
\end{align*}

The conditional probability
\begin{equation}\label{eqn: P-pre-limit}
\mathbb{P}\left(\mathcal{D}_k=B(\overline{f_k},K) \cap C_{B(\underline{f_k},K^d)}(\underline{f_k}) \mid  Y(x) \leftrightarrow\underline{f_k}\,\text{off}\,\mathcal{D}_{k-1} \cup \mathcal{C}_{k-1}(x)\right)
\end{equation}
involves only deterministic sets. Applying Theorem \ref{thm:IICexistences} and Lemma~\ref{lem:prelocalize}, we have the approximation
\begin{align*}
    \nu_{\underline{f_k}}\left(\mathcal{D}_k = B(\overline{f_k},K)\cap W\right) \sum_{x: \, g(f_{k-1}) = m} \mathbb{P}(X = x)
      \mathbb{P}\left( Y(x) \leftrightarrow\underline{f_k}\,\text{off}\,\mathcal{D}_{k-1} \cup C^{D_{k-1} \cup (\mathbb{Z}^d\setminus B(f_{k-1}, L))}(\overline{f_{k-1}})\right)\ .
\end{align*}
After summing over $x$, we obtain
\[ \nu_{\underline{f_k}}\left(\mathcal{D}_k = B(\overline{f_k},K)\cap W\right)
      \mathbb{P}\left(g(f_{k-1},\mathcal{D}_{k-1})=m, \overline{f_{k-1}} \leftrightarrow\underline{f_k}\,\text{off}\,\mathcal{D}_{k-1}\right).\]
\end{proof}

\begin{proof}[Proof of Lemma \ref{rec3}]
Let $R>K+2$ and define
\[\mathcal{L}_R(\mathcal{D}_1):=\{C_{B(\underline{f_1},R^d)}\cap B(\overline{f}_1,K)=\mathcal{D}_1\}.\]
By Lemma \ref{lem:prelocalize}, we have
\begin{equation}\label{eqn: Rtrunc-1} \mathbb{P}\big(0\lra x, \mathbf{1}_{\mathcal{L}(\mathcal{D}_1)}\neq \mathbf{1}_{\mathcal{L}_R(\mathcal{D}_1)}\big)\le CR^{-d}\tau(0,x)
\end{equation}
uniformly in $x$ and $D_1$. Similarly,
\begin{equation}\label{eqn: Rtrunc-2}
|\nu_{\underline{f_1}}(\mathcal{L}_R(\mathcal{D}_1))-\nu_{\underline{f_1}}(W_{\underline{f_1}}\cap B(\overline{f_1},K)=\mathcal{D}_1)|\le CR^{-d}.
\end{equation}
For a fixed $R'>K_H$, explore the cluster $C_{B(R')}(0)$. That is, given a configuration $\omega$, determine the status of all edges in $\mathcal{E}(B(R'))$ with at least one endpoint in $C_{B(R')}(0)$. Denote the exploration data by $q$: it determines the value $h(q)$ of $H$, the value $C_{B(R')}(0)=q$, the boundary data 
\[Y(q)=C_{B(R')}(0)\cap \partial B(R').\]
Finally, denote the probability of the exploration state by $q$.
Given $q$, the event $0\lra \underline{f_1}$ occurs iff
\[A_{q,\underline{f_1}}=\{Y(q)\lra  \underline{f}_1 \text{ off } C(q)\}.\]
Thus we can write
\[\mathbb{E}[H, 0\lra q, \mathcal{L}_R(\mathcal{D}_1)]=\sum_q h(q)p(q)\mathbb{P}\big(\mathcal{L}_R(\mathcal{D}_1),A_{q,\underline{f_1}}\big).\]
Using Theorem \ref{thm:IICexistences}, we have
\[\mathbb{P}(\mathcal{L}_R(\mathcal{D}_1)\mid A_{q,\underline{f_1}}) = \nu_{\underline{f_1}}(\mathcal{L}_R(\mathcal{D}_1)) +o_n(1)\]
uniformly $q$, $\underline{f_1}$ and $D_1$ for fixed $R'$, so we have
\[\mathbb{E}[H, 0\lra x, \mathcal{L}_R(\mathcal{D}_1)]=\nu_{\underline{f_1}}(\mathcal{L}_R(\mathcal{D}_1))\mathbb{E}[H,0\lra x]+o_n(\tau(0,x)).\]
Finally, using Theorem~\ref{thm:IICexistences} again, we have
\[\mathbb{E}[H,0\lra x]=\tau(0,x)\big(\mathbb{E}_\nu[H]+o_n(1)\big).\]
Combining the above, we have
\[\mathbb{E}[H,0\lra x, \mathcal{L}_R(\mathcal{D}_1)]=\mathbb{E}_\nu[H]\tau(0,x)\nu_{\underline{f_1}}(\mathcal{L}_R(\mathcal{D}_1))+o_n(\tau(0,x)).\]
Now take $R\rightarrow \infty$ using \eqref{eqn: Rtrunc-1}, \eqref{eqn: Rtrunc-2}.

\end{proof}

\subsubsection{Lower order terms}\label{sec: LOT}
In this section, we estimate the lower order terms appearing in the previous computations. Throughout the section, errors for integrals weighted by $H$ are understood to be in absolute value.

\begin{proposition}\label{prop: E1}
The difference $E_1$ in \eqref{eqn: E1def} is bounded by
\begin{equation}
E_1=o_K(1) \cdot  \|H\|_{\infty} (\epsilon^{2-d})^2 \cdot L^{kd}\cdot (n^{2-d})^{k+1}. 
\end{equation}
 \begin{proof}[Proof of Proposition \ref{prop: E1}]
By \eqref{rv-bdd}, it suffices to estimate
\begin{equation}
    \widetilde{\mathbb{E}}\left[\mathbbm{1}_{\Lambda_{k-1},f_k\,\text{closed}}\mathbb{P}\left(\overline{f_k}\leftrightarrow\underline{f_{k+1}}\,\text{through}\,(\mathbb{Z}^d\setminus B(\overline{f_k},K)) \cap \bigcup_{j=0}^{k-1} \widetilde{C}(\overline{f_{j}})\right)\right].
    \label{lem-low1-1}
\end{equation}
This is bounded by
\begin{equation}
    \widetilde{\mathbb{E}}\left[\mathbbm{1}_{\Lambda_{k-1}}\sum_{y \in (\mathbb{Z}^d\setminus  B(\overline{f_k},K)) \cap \bigcup_{j=0}^{k-1} \widetilde{C}(\overline{f_{j}})} \mathbb{P}\left(\{\overline{f_k} \leftrightarrow y\} \circ \{y \leftrightarrow \underline{f_{k+1}}\}\right)\right],
    \label{lem-low1-2}
\end{equation}
which equals
\begin{equation}
    \sum_{y \not\in B(\overline{f_k},K)} \widetilde{\mathbb{E}}\left[\mathbbm{1}_{\Lambda_{k-1}}\mathbbm{1}_{y \in \bigcup_{j=0}^{k-1} \widetilde{C}(\overline{f_{j}})}\mathbb{P}\left(\{\overline{f_k} \leftrightarrow y\} \circ \{y \leftrightarrow \underline{f_{k+1}}\}\right)\right].
    \label{lem-low1-3}
\end{equation}
We use the BK inequality \eqref{eqn: BK} on the inner probability and note that 
\[\{y\in \widetilde{C}(\overline{f_{j-1}})\}\cap \{\overline{f_{j-1}}\leftrightarrow\underline{f_j}\}\subset \{\underline{f_j}\leftrightarrow y\},\]
to find that \eqref{lem-low1-3} is upper bounded by
\begin{equation}
    \sum_{j=0}^{k-1}\sum_{y \not\in B(\overline{f_k},K)} \tau(\overline{f_k},y)\tau(y,\underline{f_{k+1}}) \widetilde{\mathbb{P}}\Big[\Lambda_{j-2},\underline{f_{j}}\leftrightarrow y,\overline{f_{i-1}}\leftrightarrow\underline{f_i}\,\text{ off }\,\widetilde{\zeta}_{i-1}, j\le i \le k-1\Big].
    \label{lem-low1-4}
\end{equation}
Partition over admissible clusters $\mathcal{C} = \bigcup_{i=0}^{j-2} \widetilde{C}(\overline{f_{i}})$. This enables independence: $\Lambda_{j-2}$ depends only on $\mathcal{C}$ and its edge boundary, while the latter two indicator functions depend only on bonds outside $\mathcal{C}$ and its edge boundary. Resolve the partition and bound by
\begin{equation}
    \mathbb{P}(\Lambda_{j-2})\sum_{y \not\in B(\overline{f_k},K)} \tau(\overline{f_k},y)\tau(y,\underline{f_{k+1}}) \widetilde{\mathbb{P}}(\underline{f_j}\leftrightarrow y, \overline{f_{i-1}}\leftrightarrow\underline{f_i}\,\text{ off }\,\widetilde{\zeta}_{i-1}, j\le i \le k-1),
    \label{lem-low1-5}
\end{equation}
where we have denoted
\[\tilde{\zeta}_i=\bigcup_{\ell=0}^i \tilde{C}(\overline{f_\ell}).\]
A similar argument gives
\[\widetilde{\mathbb{P}}(\underline{f_j}\leftrightarrow y, \overline{f_{i-1}}\leftrightarrow\underline{f_i}\,\text{ off }\,\widetilde{\zeta}_{i-1}, j\le i \le k-1)\le\tilde{\mathbb{P}}(\underline{f_j}\leftrightarrow y, \overline{f_{j-1}}\leftrightarrow\underline{f_j})\tilde{\mathbb{P}}(\overline{f_{i-1}}\leftrightarrow\underline{f_i}\,\text{ off }\,\widetilde{\zeta}_{i-1}, j+1\le i \le k-1)\]
Applying the BK inequality \eqref{eqn: BK} in the form
\begin{equation}\label{eqn: triple-pt}\mathbb{P}(a\leftrightarrow b, a\leftrightarrow c)\le \sum_z \mathbb{P}(a\leftrightarrow z)\mathbb{P}(b\leftrightarrow z)\mathbb{P}(c\leftrightarrow z)
\end{equation}
to the probability
\[\tilde{\mathbb{P}}(\underline{f_j}\leftrightarrow y, \overline{f_{j-1}}\leftrightarrow\underline{f_j})\]
we find that \eqref{lem-low1-5} is bounded by
\begin{equation}
    \left(\prod_{i=0}^{j-1} \tau(\overline{f_i},\underline{f_{i+1}})\prod_{i=j+1}^{k-1} \tau(\overline{f_i},\underline{f_{i+1}})\right)\left(\sum_{y \not\in B(\overline{f_k},K)}\sum_{z \in \mathbb{Z}^d} \tau(z,\underline{f_j})\tau(\overline{f_k},y)\tau(y,\underline{f_{k+1}})  \tau(\overline{f_{j-1}},z)\tau(z,y)\right).
    \label{lem-low1-6}
\end{equation}
The product is well-controlled by the two-point asymptotic \eqref{eqn: HS-scaling}. We treat the more delicate case $j=k-1$ in detail, the simpler cases $j\neq k-1$ being similar.

To control the sum \eqref{lem-low1-6}, we consider cases. \\

{\bf Case A.} If both $|z - \overline{f_{k-1}}|$ and $|y - \underline{f_{k+1}}|$ are larger than $\epsilon n/4$, we may apply the two-point asymptotic \eqref{eqn: two-pt} to find
\begin{equation}\label{eqn: product-tpt}
\tau(\overline{f_{k-1}},z)\tau(y,\underline{f_{k+1}})\le C\epsilon^{4-2d}n^{4-2d}.
\end{equation}
Recall Aizenman and Newman's triangle condition \cite{AN}:
\begin{equation}\label{eqn: Delta}
\nabla:=\sum_{x,y}\tau(0,x)\tau(x,y)\tau(y,0)<\infty.
\end{equation}
This follows directly from the two-point function bound. From \eqref{eqn: Delta}, the sum 
\[\sum_{z, y\in\mathbb{Z}^d} \tau(z,\underline{f_k}) 
 \tau(z,y)\tau(\overline{f_k},y)\]
 is absolutely convergent.
 It follows that when $K\rightarrow \infty$, then
\begin{equation}\label{eqn: triangle-decay}
\sum_{z\in\mathbb{Z}^d}\sum_{y\notin B(\underline{f_k},K)} \tau(z,\underline{f_k}) 
 \tau(z,y)\tau(\overline{f_k},y)\rightarrow 0.
\end{equation}
From \eqref{eqn: product-tpt} and \eqref{eqn: triangle-decay}, we obtain the upper bound
\begin{equation}
o(1) \cdot (\epsilon^{2-d})^2 \cdot (n^{2-d})^2 \cdot \prod_{j=0}^{k-2} |\underline{f_{j+1}}-\overline{f_j}|^{2-d},
\label{lem-low1-7}
\end{equation}
as $K \rightarrow \infty$, confirming lower-order status.\\

{\bf Case B.} Suppose without loss of generality that $|z - \overline{f_{k-1}}| \leq \epsilon n/4$, but $|y - \underline{f_{k+1}}| > \epsilon n/4$. By the triangle inequality,
\begin{equation}
    |z - \underline{f_k}| \geq |\underline{f_k}-\overline{f_{k-1}}| - |z - \overline{f_{k-1}}| \geq \epsilon n - \epsilon n/4 \geq  (3/4) \epsilon n.
    \label{lem-low1-8}
\end{equation}
Using these two bounds, we use the two-point asymptotic \eqref{eqn: two-pt} to estimate \eqref{lem-low1-6} by
\[(\epsilon^{2-d})^2(n^{2-d})^2\left(\prod_{i=0}^{k-2} \tau(\overline{f_i},\underline{f_{i+1}})\right) \sum_{y \not\in B(\overline{f_k},K)} \sum_{z \in \mathbb{Z}^d} \tau(\overline{f_{k-1}},z)\tau(z,y)\tau(\overline{f_k},y).\]
Using the translation invariance, the double sum is bounded by
\begin{equation}\label{eqn: prepare-triangle}
\sum_{y,z \in \mathbb{Z}^d} \tau(0,z)\tau(z,y)\tau(y,\overline{f_{k-1}}-\overline{f_k}).
\end{equation}
We now recall Barsky and Aizenman's result \cite[Lemma 2.1]{BA} that, on $\mathbb{Z}^d$, the condition \eqref{eqn: Delta} implies the decay of the \emph{open triangle}:
\begin{equation}\label{eqn: open-T}
\lim_{R\rightarrow\infty} \sup_{|w|>R}\sum_{x,y\in \mathbb{Z}^d} \tau(0,x)\tau(x,y)\tau(y,w)=0.
\end{equation}
This implies that \eqref{eqn: prepare-triangle} tends to zero as $n\rightarrow \infty$.

{\bf Case C.} If both $|z - \overline{f_{k-1}}| \leq \epsilon n/4$ and $|y - \underline{f_{k+1}}| \leq \epsilon n/4$, then 
\[|z - \underline{f_k}| \geq  |\underline{f_k}-\overline{f_{k-1}}|-|z-\overline{f_{k-1}}|\ge (3/4)\epsilon n,\]
and 
\[|\overline{f_k}-y|\geq |\overline{f_k}-\underline{f_{k+1}}|-|y-\underline{f_{k+1}}|\ge (3/4)\epsilon n,\]
so the sum in \eqref{lem-low1-6} is bounded by
\[C(\epsilon^{2-d})^2n^{4-2d}\sum_{y \not\in B(\overline{f_k},K)} \sum_{z \in \mathbb{Z}^d} \tau(\overline{f_{k-1}},z)\tau(z,y)\tau(y,\underline{f_{k+1}}).\]
Using translation invariance and the decay of the open triangle \eqref{eqn: open-T} as in case B., we obtain the desired result.
\end{proof}
\end{proposition}

\begin{proposition}
\label{prop: E2}
Let $R\ge K$. The quantity $E_2$ defined in \eqref{eqn: E2def} is bounded  in absolute value by 
\[E_2= \|H\|_{\infty} C2^{-R}\cdot L^{kd} \cdot (n^{2-d})^{k+1}.\]
\begin{proof}[Proof of Proposition \ref{prop: E2}]
After applying \eqref{rv-bdd} to bound the quantities $X_k$, it suffices to bound the probability
\begin{equation}
    \mathbb{P}(\overline{f_k}\leftrightarrow \partial B(f_k;2^R)\,\text{off}\,\mathcal{D}_k\,\text{but}\,\overline{f_k}\leftrightarrow\underline{f_{k+1}}\,\text{only through}\,\mathcal{D}_k).
    \label{lem-low2-1}
\end{equation}
This event implies the occurrence of
\begin{equation}
    \{\overline{f_k}\leftrightarrow\partial B(f_k;2^R)\} \circ \{\partial B(\overline{f_k},K) \leftrightarrow \underline{f_{k+1}}\}.
    \label{lem-low2-2}
\end{equation}
Indeed, choosing a path $\gamma$ from $\overline{f_k}$ to $\partial B(f_k,2^R)$ avoiding $\mathcal{D}_k$, and letting $\tilde{\gamma}$ be the portion of a path from $\overline{f_k}$ to $\underline{f_{k+1}}$ between its last exit from $\partial B(f_k,K)$ and $\underline{f_{k+1}}$, we claim that $\gamma$ and $\tilde{\gamma}$ do not intersect. If they did, concatenating the portion of $\gamma$ from $\overline{f_k}$ to the first intersection and the portion of $\tilde{\gamma}$ from that intersection to $\underline{f_{k+1}}$, we obtain a path avoiding $\mathcal{D}_k$ which joins $\overline{f_k}$ and $\underline{f_{k+1}}$, which gives a contradiction.
Applying BK and \eqref{eqn: KN}, we get a bound of
\begin{equation}
    C \cdot 2^{-2R} \cdot K^{d-1} \cdot (|f_{k+1}-f_k|-K)^{2-d}.
    \label{lem-low2-3}
\end{equation}
Keeping $K$ small compared to $n$ tells us that the gap between \eqref{lem-rec1-6} and \eqref{lem-rec1-7} is bounded by
\begin{equation}
    C \cdot 2^{-2R} \cdot K^{d-1} \cdot \prod_{i=0}^{k} \tau(\overline{f_i},\underline{f_{i+1}}),
    \label{lem-low2-4}
\end{equation}
from which the main result follows immediately.
\end{proof}

\end{proposition}

\begin{proposition}\label{prop: E3}
The difference $E_3$ defined in \eqref{eqn: E3def} is bounded by
\begin{equation}
\begin{split}
    E_3&\le CK^{d-1} \textbf{}(\epsilon^{2-d})^{k+1}  n^{4-d}\cdot (n^{2-d})^k\\
    &= o(1) K^{d-1} \textbf{}(\epsilon^{2-d})^{k+1}(n^{2-d})^k,
\end{split}
\end{equation}
where $C$ denotes a constant independent of the parameters $K$, $n$ and $\epsilon$ and the $o(1)$ factor tends to zero as $n\rightarrow\infty$.
\begin{proof}[Proof of Proposition \ref{prop: E3}]
Suppose we have a bond whose occupation status affects both $\{\mathcal{C}\,\text{is a cluster}\}$ and $\{\mathcal{D}_k\,\text{is a cluster}\}$. Then this bond must have one endpoint in $\mathcal{C}$ and one in $\mathcal{D}_k$. Since $\mathcal{D}_k\cup\partial\mathcal{D}_k \subseteq B(\overline{f_k},K+1)$ by definition, this implies $\mathcal{C} \cap \partial B(\overline{f_k},K+1) \neq \varnothing$. Thus, summing over such $\mathcal{C}$ and $\mathcal{D}_k$, we can upper bound
\begin{equation}
    \mathbb{P}\left(\Lambda_{k-2}, \overline{f_{k-1}}\leftrightarrow\underline{f_k}\,\text{off}\,\bigcup_{j=0}^{k-2} C(\overline{f_{j}}),\mathcal{D}_k=B(\overline{f_k},K)\cap \bigcup_{j=0}^{k-1} C(\overline{f_{j}}),\bigcup_{j=0}^{k-2} C(\overline{f_{j}}) \cap \partial B(\overline{f_k},K+1)\neq\varnothing\right)
    \label{lem-low3-1}
\end{equation}
by resolving the partition over $\mathcal{D}_k$ and using \eqref{eqn: triple-pt}. \\

If the event in the probability in \eqref{lem-low3-1} occurs, then for some $i \leq k-2$, we have the occurrence of  $\{\overline{f_i} \leftrightarrow z, z \leftrightarrow \underline{f_{i+1}}\}$, for some vertex
\[z\in \bigcup_{j=0}^{k-2} C(\overline{f_{j}}) \cap \partial B(\overline{f_k},K+1).\] The remaining connections between $f_i$ must be disjoint to respect the conditions of $\Lambda_{k-2}$. We apply \eqref{eqn: triple-pt} to the probability of the event $\{\overline{f_i} \leftrightarrow z, z \leftrightarrow \underline{f_{i+1}}\}$ to obtain
\begin{equation}
    \sum_{i=0}^{k-2} \prod_{j=0,j\neq i}^{k-1} |\underline{f_{j+1}}-\overline{f_j}|^{2-d} \sum_z \sum_{y \in \mathbb{Z}^d} \tau(\overline{f_i},y)\tau(y,z)\tau(y,\underline{f_{i+1}}).
    \label{lem-low3-3}
\end{equation}
where $z \in \partial B(\overline{f_k},K+1)$. Split the $y$-sum into the far and near regimes, i.e. over $y$ such that $|y - \underline{f_{i+1}}|$ is larger than $\epsilon n/4$, and not, respectively. \\

{\bf Case A. Far regime.} We sum over all $y \in \mathbb{Z}^d$ in the far regime. Bounding $\tau(y,\underline{f_{i+1}})$ in \eqref{lem-low3-3} by $C \cdot \epsilon^{2-d} \cdot n^{2-d}$, we get
\begin{equation}
    C \cdot \epsilon^{2-d} \cdot n^{2-d} \sum_{i=0}^{k-2}  \prod_{j=0,j\neq i}^{k-1} |\underline{f_{j+1}}-\overline{f_j}|^{2-d} \sum_z \sum_y \tau(\overline{f_i},y)\tau(y,z).
\end{equation}
To bound this sum, we apply the convolution bound from Proposition 1.7 in \cite{HHS}, noting that the sum is over $z \in \partial B(\overline{f_k},K+1)$. This directly yields the bound
\begin{equation}\label{eqn: A-bound}
    C \cdot K^{d-1} \cdot (\epsilon^{2-d})^{k+1} \cdot n^2 \cdot (n^{2-d})^{k+1}  
\end{equation}
which is lower order than the leading term, which is order $(n^{2-d})^k$. \\

{\bf Case B. Near regime.} Consider the near-regime part of \eqref{lem-low3-3}, i.e., the sum over all $y$ with $|y - \underline{f_{i+1}}| \leq \epsilon n/4$. Since the $f_i$ are at distance at least $\epsilon n$ from each other, we have $|y - z| > \epsilon n/4$ and $|y - \overline{f_i}| > \epsilon n/4$. The sum of $\tau(y,\underline{f_{i+1}})$ over such $y$ can be bounded by $C \cdot \epsilon^2 \cdot n^2$. Applying the two-point asymptotic to the remaining factors and summing, we get the bound
\begin{equation}
    C \cdot K^{d-1} \cdot (\epsilon^{2-d})^{k+1} \cdot \epsilon^2 \cdot n^2 \cdot (n^{2-d})^{k+1},
\end{equation}
which is asymptotically of order \eqref{eqn: A-bound} when $n\rightarrow\infty$.
\end{proof}
\end{proposition}

\begin{proposition}\label{prop: E4}
The quantity $E_4$ introduced in \eqref{eqn: E4def} satisfies the bound
\[E_4\le C \cdot K^{d-1} \cdot (\epsilon^{2-d})^{k+1} \cdot \epsilon^2 \cdot n^{4-d} \cdot (n^{2-d})^{k}.\]
\begin{proof} 
We can proceed as in the proof of Proposition \ref{prop: E3}. Consider
\begin{equation}
    \sum_{\mathcal{D}_k} \widetilde{\mathbb{E}}\left[\mathbbm{1}_{\Lambda_{k-2}}\mathbbm{1}_{\bigcup_{j=0}^{k-2} \widetilde{C}(\overline{f_{j}})\cap \partial B(\overline{f_k},K+1)\neq\varnothing}\mathbb{P}\left(\overline{f_{k-1}}\leftrightarrow\underline{f_k}\,\text{off}\,\bigcup_{j=0}^{k-2} \widetilde{C}(\overline{f_{j}}),\mathcal{D}_k=B(\overline{f_k},K) \cap \bigcup_{j=0}^{k-1} C(\overline{f_{j}})\right)\right].
\end{equation}
We carry out the same procedure as before. Resolve the partition over $\mathcal{D}_k$, drop the restriction on the connection in the inner probability, and follow the line of reasoning after \eqref{lem-low3-3} to conclude.
\end{proof}    
\end{proposition}

\begin{proposition}\label{prop: E5}
The difference $E_5$, defined in \eqref{eqn: E5def}, is bounded  in absolute value by
\[E_5=o(1) \cdot  (\epsilon^{2-d})^2 \cdot L^{(k-1)d}  \|H\|_{\infty}\cdot (n^{2-d})^{k}.\]
\begin{proof}[Proof of Proposition \ref{prop: E5}]
Applying \eqref{rv-bdd} to the quantity $E_5$, we obtain a factor $L^{(k-1)d} \|H\|_{\infty}$ times
\begin{align*}
    &\sum_{\mathcal{D}_k} \widetilde{\mathbb{E}}\left[\mathbbm{1}_{\Lambda_{k-2},f_{k-1}\,\text{closed}}\mathbb{P}\big(\overline{f_{k-1}}\leftrightarrow\underline{f_k}\,\text{through}\,\bigcup_{j=0}^{k-2} \widetilde{C}(\overline{f_{j}})\cap (\mathbb{Z}^d\setminus B(f_{k-1};K)),f_k\,\text{closed},\mathcal{D}_k=B(\overline{f_k},K) \cap \bigcup_{j=0}^{k-1} C(\overline{f_{j}})\big)\right]\\
    =&~\widetilde{\mathbb{E}}\left[\mathbbm{1}_{\Lambda_{k-2},f_{k-1}\,\text{closed}}\mathbb{P}\big(\overline{f_{k-1}}\leftrightarrow\underline{f_k}\,\text{through}\,\bigcup_{j=0}^{k-2} \widetilde{C}(\overline{f_{j}})\cap (\mathbb{Z}^d\setminus B(f_{k-1};K)),f_k\,\text{closed}\right]
    \label{lem-low1-9}.
\end{align*}
This is estimated as in the proof of Proposition \ref{prop: E1}, replacing $k$ by $k-1$.
\end{proof}
\end{proposition}

\begin{proposition}\label{prop: only-W}
The quantity $E_6$ defined in \eqref{eqn: E6def} is of order $n^{(2-d)(k+1)}$: the
expectation
\begin{equation}
\widetilde{\mathbb{E}}\left[HX_{k-2}\mathbbm{1}_{\Lambda_{k-2}}\mathbb{E}\left(g(f_{k-1},\tilde{\zeta}_{k-2}),\overline{f_{k-1}}\leftrightarrow\underline{f_k}\,\text{off}\,\widetilde{\zeta}_{k-2}\cap B(f_{k-1};K),\mathcal{D}_k=B(\overline{f_k},K) \cap \zeta_{k-1} \right)\right] \label{eqn: zeta-added}
\end{equation}
appearing in \eqref{eqn: P-pre-limit} equals
\begin{equation}
\begin{gathered}
\cdot  \widetilde{\mathbb{E}}\left[HX_{k-2}\mathbbm{1}_{\Lambda_{k-2}}\mathbb{E}\left(g(f_{k-1},\tilde{\zeta}_{k-2}),\overline{f_{k-1}}\leftrightarrow\underline{f_k}\,\text{off}\,\widetilde{\zeta}_{k-2}\cap B(f_{k-1};K),\mathcal{D}_k=B(\overline{f_k},K) \cap C(\underline{f_k})\right)\right]\\
+  L^{(k-1)d} \|H\|_{\infty} O(n^{(2-d)k}).
\end{gathered}
\end{equation}

\begin{proof}
Since
\[\mathbb{P}(\Lambda_{k-2})\le Cn^{(2-d)(k-1)}\]
and
\[|g|\le CL^d,\]
it will suffice to show that in the probability
\[\mathbb{P}\left(\overline{f_{k-1}}\leftrightarrow\underline{f_k}\,\text{off}\,\widetilde{\zeta}_{k-2}\cap B(f_{k-1},K),\mathcal{D}_k=B(\overline{f_k},K) \cap \zeta_{k-1} \right)\]
we can replace
\[\zeta_{k-1}=C(\overline{f_{k-1}})\cup C(\overline{f_{k-2}}) \cup \dots \cup C(\overline{f_0})\]
in \eqref{eqn: zeta-added} by $C(\overline{f_{k-1}})=C(\underline{f_k})$ at the cost of a negligible error as $n\rightarrow\infty$. Let $F$ be the event that, for some $0 \leq i \leq k-2$, $C(\overline{f_i}) \cap \mathcal{D}_k \neq \varnothing$, and $C(\overline{f_i}) \neq C(\overline{f_{k-1}})$. Using the BK inequality, we have:
\begin{align}
    \nonumber \mathbb{P}(F, \mathcal{D}_k = B(\overline{f_k},K) \cap \zeta_{k-1}, \overline{f_{k-1}}\leftrightarrow\underline{f_k}) &\leq \sum_{i=0}^{k-2}\mathbb{P}(\cup_{z \in B(\overline{f_k},K)}\{ \overline{f_i}\leftrightarrow z\} \circ\{\underline{f_k}\leftrightarrow \overline{f_{k-1}}\}) \\
    &\leq C K^{d-1} \cdot n^{2-d} \cdot n^{2-d}.\label{eqn: final-ratio}
\end{align}
\end{proof}
\end{proposition}

\begin{proposition}\label{prop: E8}
For $R\ge K$ term $E_8$ in \eqref{eqn: E8def} is bounded as follows
\[E_8\le C2^{-R}\cdot (n^{2-d})^{k-1}.\]
\begin{proof}[Proof of Proposition \ref{prop: E8}]
The proof is identical to that of the previous Proposition \ref{prop: E2}.
\end{proof}
\end{proposition}

\begin{proposition}\label{prop: localize}
The term $E_7$ defined in \eqref{eqn: E7def} has the estimate
\[\sum_{\mathcal{D}_k}E_7(\mathcal{D}_k)\le CL^d  \|H\|_{\infty} K^{-d}n^{(-d+2)k}.\]
    \begin{proof}
Consider the difference between the inner expectations of \eqref{eqn: lady-godiva} and \eqref{eqn: girard}:
\begin{align*}
&\big|\mathbb{E}\left(g(f_{k-1},\tilde{\zeta}_{k-2}),\overline{f_{k-1}}\leftrightarrow\underline{f_k}\,\text{off}\,\widetilde{\zeta}_{k-2}\cap B(f_{k-1},K),\mathcal{D}_k=B(\overline{f_k},K) \cap C(\underline{f_k})\right)\\
-~&\mathbb{E}\left(g(f_{k-1},\tilde{\zeta}_{k-2}),\overline{f_{k-1}}\leftrightarrow\underline{f_k}\,\text{off}\,\widetilde{\zeta}_{k-2}\cap B(f_{k-1},K),\mathcal{D}_k=B(\overline{f_k},K) \cap C_{K^d}(\underline{f_k})\right)\big|\\
\le~& CL^d \mathbb{P}(\overline{f_{k-1}}\leftrightarrow\underline{f_k}, B(\overline{f_k},K) \cap C(\underline{f_k})\neq B(\overline{f_k},K) \cap C_{K^d}(\underline{f_k}), \mathcal{D}_k=B(\overline{f_k},K) \cap C_{K^d}(\underline{f_k}))\\
+~&CL^d \mathbb{P}(\overline{f_{k-1}}\leftrightarrow\underline{f_k}, B(\overline{f_k},K) \cap C(\underline{f_k})\neq B(\overline{f_k},K) \cap C_{K^d}(\underline{f_k}), \mathcal{D}_k=B(\overline{f_k},K) \cap C(\underline{f_k})).
\end{align*}
We apply Lemma~\ref{lem:prelocalize} to bound each of the probabilities in the last display by $C K^{-d} n^{2-d}$.
Combining this with the estimate
\[\tilde{\mathbb{E}}(X_{k-2},\Lambda_{k-2})\le Cn^{(2-d)(k-1)},\]
we obtain the desired bound after summing over $\mathcal{D}_k$.
\end{proof}

\end{proposition}

\subsection{Convolution bound for the near-regime}\label{sec: convolution}
In this section, we derive Lemma \ref{convolution-bound}, the convolution bound used in Proposition \ref{nearregime-thm}.
\begin{proof}[Proof of Lemma \ref{convolution-bound}]
We show the above bound inductively. \\

{\bf Base case.} Let $k = 1$. Consider
\begin{equation}\label{eqn: to-bound}
    \sum_{x_1 \in B(2\|y\|_\infty)} \mathbbm{1}_{|x_1| < \epsilon\|y\|_\infty}\langle x_1\rangle^{2-d}\langle y-x_1\rangle^{2-d}.
\end{equation}
Provided $\epsilon\le \frac{1}{2\sqrt{d}}$, it follows from the triangle inequality that $|y-x_1| \geq \frac{1}{2}|y|$. So the sum \eqref{eqn: to-bound} is bounded by
\begin{equation}
    2^{d-2} \langle y\rangle^{2-d} \sum_{|x_1|<\epsilon\|y\|_\infty} \langle x_1\rangle^{2-d}.
\end{equation}
Sum over dyadic annuli to obtain the final bound, proving the statement for $k = 1$ and $i = 1$:
\begin{equation}\label{eqn: e-small-bd}
    C \epsilon^2 \langle y\rangle^{4-d}.
\end{equation}
The estimate \eqref{eqn: e-small-bd} holds for $\epsilon$ small enough (depending only on the dimension). If $\epsilon\ge \epsilon_0$ , we instead bound \eqref{eqn: to-bound} ignoring the condition on $|x_1|$, using the general convolution estimate
\begin{equation}\label{eqn: conv-est}
\sum_{z\in\mathbb{Z}^d}\langle x-z\rangle^{-\alpha}\langle z-y\rangle^{-\beta} \le C\langle x-y\rangle^{-\alpha-\beta+d}
\end{equation}
for $\alpha+\beta>d$ to obtain the estimate
\[C\langle y\rangle^{4-d}\le \frac{C}{\epsilon_0^2}\epsilon^2\langle y\rangle^{4-d},\]
so the result is proved in case $k=1$, $i=1$. For the case $k = 1$ and $i = 2$, the above argument goes through identically, completing the base case. \\

{\bf Inductive step.} We assume the statement has been shown for $k$ and show it for $k+1$. Since convolutions commute, we may assume $i = 1$. \\

We therefore are tasked with bounding 
\begin{align} 
&\sum_{\substack{x_1, \dots, x_k \in B(2 \|y\|_\infty) \\ |x_1| < \epsilon \|y\|_{\infty}}} \langle x_1\rangle^{2-d} \langle x_2 - x_1\rangle^{2-d} \dots \langle x_k - x_{k-1}\rangle^{2-d} \langle y - x_k\rangle^{2-d} \nonumber \\
= &\sum_{\substack{x_1, \dots, x_k \in B(2 \|y\|_\infty) \\ |x_1| < \epsilon \|y\|_{\infty}}} \left[\mathbbm{1}_{|x_2 - x_1| < 2 \epsilon\|y\|_\infty} + \mathbbm{1}_{|x_2 - x_1| \geq 2 \epsilon\|y\|_\infty} \right] \langle x_1\rangle^{2-d} \langle x_2 - x_1\rangle^{2-d} \cdots \langle x_k - x_{k-1}\rangle^{2-d} \langle y - x_k\rangle^{2-d}\ .\label{eqn: two-indic}
\end{align}
We bound the terms corresponding to the two indicator functions in \eqref{eqn: two-indic} separately.

For the first term in \eqref{eqn: two-indic}, where $|x_2 - x_1| \le 2 \epsilon \|y\|_\infty$, isolate the sum over $x_1$ and bound uniformly in $x_2$:
\[\sum_{\substack{x_1 \in B(2 \|y\|_\infty) \\ |x_1| < \epsilon \|y\|_{\infty} \\ |x_1 - x_2| < 2 \epsilon \|y\|_\infty}}  \langle x_1\rangle^{2-d} \langle x_2 - x_1\rangle^{2-d} \ .\]
Note that the inequalities in the subscript imply also $\|x_2\|_\infty \le 3 \epsilon \|y\|_{\infty}$. Provided $3\epsilon$ is small enough, we can therefore apply the inductive hypothesis in the case $k = 1$ to upper bound the above by
\[C \langle x_2\rangle^{4-d} \leq C (3\epsilon)^2 \|y\|^2_\infty \langle x_2\rangle^{2-d}\ . \]
Plugging back in, the sum of the first term is at most
\[ C (3\epsilon)^2 \|y\|^2_\infty \sum_{\substack{x_2, \dots, x_k \in B(2 \|y\|_\infty) \\  |x_2| < 3 \epsilon \|y\|_{\infty}}} \langle x_2\rangle^{2-d} \langle x_2 - x_1\rangle^{2-d} \dots \langle x_k - x_{k-1}\rangle^{2-d} \langle y - x_k\rangle^{2-d},\]
and applying the induction hypothesis (to the case $k-1$), this is at most
\[ C^{k} (3\epsilon)^4 \langle y\rangle^{2(k+1)-d}\leq C^{k+1} \epsilon^2 \langle y\rangle^{2(k+1)-d}.\]

For the second term, we again first sum over $x_1$. Since in this case $|x_2 - x_1| \geq |x_2| / 2$, the relevant factors are
\[ \sum_{\substack{x_1 \in B(2 \|y\|_\infty) \\ |x_1| < \epsilon \|y\|_{\infty} \\ |x_1 - x_2| \geq 2 \epsilon \|y\|}}  \langle x_1\rangle^{2-d} \langle x_2 - x_1\rangle^{2-d} \leq  2^{d-2} |x_2|^{2-d} \sum_{\substack{x_1 \in B(2 \|y\|_\infty) \\ |x_1| < \epsilon \|y\|_{\infty}}}  \langle x_1\rangle^{2-d}\ . \]
Summing yields the bound
\[  C_1 2^{d-2} \epsilon^2 \langle x_2\rangle^{2-d}  \langle y\rangle^2\ . \]
Plugging this back into the sum over the remaining $x_i$, we have to bound
\[ C_1 2^{d-2} \epsilon^2   \langle y\rangle^2 \sum_{\substack{x_2, \dots, x_k \in B(2 \|y\|_\infty) }} \langle x_2\rangle^{2-d}  \cdots \langle x_k - x_{k-1}\rangle^{2-d} \langle y - x_k\rangle^{2-d}.\]
We use the $\epsilon = 2$ case of our inductive hypothesis, leading to the bound
\[  C_1 2^{d-2} \epsilon^2  C^k |y|^{2(k+1) - d} \leq  C^{k+1} \epsilon^2 \langle y\rangle^{2(k+1)-d} \]
assuming $C$ is large relative to $C_1$. Pulling both terms together completes the proof.
\end{proof}

\section{Control of bubbles}\label{sec: control}
In this section, we show that we can replace the distance and resistance across bubbles, as defined in Section \ref{sec: additive}, by the truncated quantities defined in Section \ref{sec: local-v}.
\subsection{Outside the box}
The following gives an estimate for the contribution of edges outside $[-Mn,Mn]^d$:
\begin{proposition}\label{prop: outside-M}
There is a constant $C$ independent of $M$ and $n$ such that:
    \begin{equation}\label{eqn: M-trunc-est}
    n^{d-4}\sum_{f\in \mathcal{E}(\mathbb{Z}^d\setminus [-Mn,Mn]^d)} \mathbb{E}[D(bubble(\overline{f})), f \ \text{pivotal for } 0\leftrightarrow n\mathbf{e}_1]\le CM^{4-d}.
    \end{equation}
    The same estimate holds with $D(bubble(\overline{f}))$ replaced by $R_{\mathrm{eff}}(\overline{f})$ or 1, after possibly adjusting the constant.
\begin{proof}
Since 
\[1\le R_{\mathrm{eff}}(bubble(\overline{f}))\le D(bubble(\overline{f})),\]
it will suffice to estimate the quantity:
\begin{equation}
   \label{eqn: bubble-E-sum} 
n^{-2}\times n^{d-2} \times \sum_{f\notin \mathcal{E}([-Mn,Mn]^d)} \mathbb{E}[D(bubble(\overline{f})), f \ \text{pivotal}].
\end{equation}

Summing over the possible heads of the bubble of $f$ (Recall the definition in Section \ref{sec: bubble-def}), we find the upper bound 
\begin{equation} \label{eqn: int-bubble-trunc} \leq C n^{d-4}  \sum_{z \in \mathbb{Z}^d} \mathbb{E}[\mathrm{dist}(f, z),\,  \text{$f$ open pivotal}, head(bubble(\overline{f}))=z ]\ . 
\end{equation}
If $z$ is the head of $f$'s bubble, $\overline{f}$ has two disjoint connections to $z$. We upper bound $\mathrm{dist}(f, z)$ by the number of vertices on one of these connections:
\[\mathrm{dist}(f, z)\le \#\{y\in \mathbb{Z}^d: 0\leftrightarrow y \text{ and } y\leftrightarrow z  \text{ disjointly}\}.\]
This leads to the upper bound
\begin{equation}\label{eqn: 5-fold} C  \sum_{z, y \in \mathbb{Z}^d} \mathbb{P}(0 \leftrightarrow \underline{f}) \mathbb{P}(\overline{f} \leftrightarrow y) \mathbb{P}(\overline{f} \leftrightarrow z) \mathbb{P}(y \leftrightarrow z) \mathbb{P}(z \leftrightarrow n\mathbf{e}_1)\ . 
\end{equation}
We sum over $y$ first and use the bound
\begin{equation}\label{eqn: y-z-est}
\sum_{y\in \mathbb{Z}^d} \langle \overline{f} - y\rangle^{2-d} \langle y - z\rangle^{2-d} \leq C\langle \overline{f}-z\rangle^{4-d}.
\end{equation}
Together with the two-point function bound \eqref{eqn: HS-scaling}, we find that \eqref{eqn: 5-fold} is bounded by
\begin{equation}\label{eqn: z-sum}
C \sum_{z\in \mathbb{Z}^d}  \langle \overline{f}-z\rangle^{6-2d} \mathbb{P}(0 \leftrightarrow \underline{f})  \mathbb{P}(z \leftrightarrow n\mathbf{e}_1)\ . 
\end{equation}
Summing \eqref{eqn: z-sum} over $f\notin \mathcal{E}([-Mn,Mn]^d)$, we find
\begin{equation}
\label{eqn: f-sum}
\begin{split}
&\sum_{f\in \mathcal{E}(\mathbb{Z}^d\setminus [-Mn,Mn]^d)} \mathbb{P}(0 \leftrightarrow \overline{f}) \sum_{z \in \mathbb{Z}^d}\langle \overline{f}-z\rangle^{6-2d}\mathbb{P}(z \leftrightarrow n\mathbf{e}_1)\\
\le~& C\sum_{f\in \mathcal{E}(\mathbb{Z}^d\setminus [-Mn,Mn]^d)}\langle \overline{f}\rangle^{2-d}\sum_{z\in \mathbb{Z}^d}\langle \overline{f}-z\rangle^{6-2d}\langle z-n\mathbf{e}_1\rangle^{2-d}\\
\le~& C\sum_{f\in \mathcal{E}(\mathbb{Z}^d\setminus [-Mn,Mn]^d)}\langle \overline{f}\rangle^{2-d}\langle \overline{f}-n\mathbf{e}_1\rangle^{2-d}
\end{split}
\end{equation}
In the final step, we have used the estimate
\begin{equation}
\sum_{x\in \mathbb{Z}^d}
\langle u-x\rangle^{-\alpha} \langle x-v\rangle^{-\beta} \le C\langle u-v\rangle^{-\beta}
\end{equation}
if $\alpha>d$, $\beta>0$.

Summing \eqref{eqn: f-sum} over $f\in \mathcal{E}(\mathbb{Z}^d\setminus [-Mn,Mn]^d)$, we get
\[\sum_{f\in \mathcal{E}(\mathbb{Z}^d\setminus [-Mn,Mn]^d)}\langle \overline{f}\rangle^{2-d}\langle \overline{f}-n\mathbf{e}_1\rangle^{2-d}\le C(Mn)^{4-d}.\]
\end{proof}
\end{proposition}

\subsection{Large bubbles}
Next, we control the effect of large bubbles.

\begin{proposition}\label{prop: large-bubble-trunc}
We have the estimate, uniformly in $n$ and $L$ sufficiently large:

\begin{equation}\label{eqn: undisclosed}
n^{d-4}\sum_{f\in \mathbb{Z}^d} \mathbb{E}[D(bubble(\overline{f})),\,\mathrm{diam}(bubble(\overline{f}))\ge L/2, f \ \text{pivotal for } 0\leftrightarrow n\mathbf{e}_1]\le CL^{6-d}(\log L)^3.
\end{equation}
\end{proposition}
\begin{proof}

To simplify notation, we let
\[D(\overline{f}):=D(bubble(\overline{f})).\]
We decompose the inner expectation according to the diameter of the bubble:
\begin{equation} \sum_{k = \log (L/2)}^{\infty} \mathbb{E}[D(\overline{f}) ; 2^k \leq \mathrm{diam}(bubble(f)) <  2^{k+1},  \text{$f$ open pivotal for } 0\leftrightarrow n\mathbf{e}_1]\ . 
\end{equation}

Next, we decompose the expectation according to whether or not there is some vertex $y$ in the bubble of $\overline{f}$ such that 
\[\mathrm{dist}(f, y) > 2^{2k} k^3.\] 
 On this event we upper bound $D(\overline{f})$ by $2^{dk}$ deterministically and use the estimate (7) from \cite{CHS}
\begin{equation}\label{eqn: CHS-bound}
\mathbb{P}\big(\mathrm{dist}(\overline{f},y)>\lambda 2^{2k} \mid 0 \sa{B(\overline{f};2^k)}  y\big)\le e^{-c\lambda}.
\end{equation}
We see
\begin{equation}\label{eqn: dist-control}
\begin{split}
    &\mathbb{E}[D(\overline{f}) ; 2^k \leq \mathrm{diam}(bubble(f)) <  2^{k+1},  \text{$f$ open pivotal}, \text{$\exists y$  in $Ann(2^k, 2^{k+1}; f)$ s.t. } \mathrm{dist}(\overline{f}, y) > 2^{2k} k^3 ]\\
    \leq ~& 2^{dk} \sum_{y \in Ann(2^k , 2^{k+1}; f)} \mathbb{P}[\{0 \leftrightarrow f\} \circ \{y \leftrightarrow n\mathbf{e}_1\} \circ\{y \stackrel{B(f; 2^{k+1})}{\longleftrightarrow} f \text{ by an open path of length $> 2^{2k} k^3$}]\\
    \leq ~&  2^{dk} \sum_{y \in Ann(2^k , 2^{k+1}; \overline{f})}  \langle f\rangle^{2-d}\langle y - n\mathbf{e}_1\rangle^{2-d}  \mathbb{P}( y \stackrel{B(f; 2^{k+1})}{\longleftrightarrow} f \text{ by an open path of length $> 2^{2k} k^3$})\\
    \leq ~& 2^{dk}\langle f\rangle^{2-d}2^{-c k^3}  \sum_{y \in Ann(2^k , 2^{k+1}; f)} \langle y-n\mathbf{e}_1\rangle^{2-d} \\
    \le ~& 2^{2kd} 2^{-ck^3}\langle f\rangle^{2-d}\langle f-n\mathbf{e}_1\rangle^{2-d}.
\end{split}
\end{equation}

It suffices to bound the remaining term, where $\mathrm{dist}(x,y)\le 2^{2k}k^3$
for any $y\in Ann(2^k,2^{k+1};\overline{f})$.
We decompose the sum based on the location of $head(bubble(\overline{f}))$:
\begin{align*}
    &\mathbb{E}[D(\overline{f}); 2^k \leq \mathrm{diam}(bubble(f)) <  2^{k+1},  \text{$f$ open pivotal}, \text{ no vertex $y$  in $Ann(2^k, 2^{k+1}; f)$ has } \mathrm{dist}(\overline{f}, y) > 2^{2k} k^3 ]\\
    \leq &~2^{2k} k^3 \sum_{|z-\overline{f}| \le 2^{k+1}}  \mathbb{P}(2^k \leq \mathrm{diam}(bubble(\overline{f})) <  2^{k+1}, head(bubble(\overline{f}))=z,  \text{$f$ open pivotal}).
\end{align*}

\paragraph{Case A: $|z-\overline{f}|> 2^{k-3}$.} We bound the sum over $z\in Ann(2^{k-3},2^{k+1};\overline{f})$. In this case, we have the bound by
\begin{align*}
&2^{2k}k^3 \sum_{z\in Ann(2^{k-3}, 2^{k+1}; \overline{f})}\mathbb{P}(0\leftrightarrow \overline{f}\circ f\Leftrightarrow z \circ z\leftrightarrow n\mathbf{e}_1)\\
\le &2^{2k}k^3  2^{(4-2d)k} \langle f\rangle^{2-d}\sum_{z\in Ann(2^{k-3}, 2^{k+1}; \overline{f})}\langle  z- n\mathbf{e}_1\rangle^{2-d}\\
\le & 2^{2k}k^3  2^{(4-2d)k} \langle f\rangle^{2-d}2^{kd}\langle \overline{f}-n\mathbf{e}_1\rangle^{2-d}\\
\le& 2^{(6-d)k}k^3 \langle f\rangle^{2-d} \langle \overline{f}-n\mathbf{e}_1\rangle^{2-d}.
\end{align*}
Summing over $k$, we obtain the result under the condition $d>6$.

\paragraph{Case B: $|z-\overline{f}|\le 2^{k-3}$.}It remains to consider the sum over $|z-\overline{f}|\le 2^{k-3}$.  If $2^k\le \mathrm{diam}(bubble(\overline{f}))<2^{k+1}$, there is a $w\in Ann(\overline{f};2^k,2^{k+1})\cap bubble(\overline{f})$. There are two edge-disjoint connections $\eta_1$ and $\eta_2$ from $z$ to $\overline{f}$, forming a cycle $\eta$. The vertex $w$ has two connections to $\overline{f}$ that first meet $\eta$ at two vertices $a_1$ and $a_2$.

\paragraph{Case B.1: $a_1$ and $a_2$ in the same segment.} If $a_1$ and $a_2$ both lie on the same segment, $\eta_1$ or $\eta_2$ of $\eta$, then there is an open arc $\eta'$
\[\overline{f}\rightarrow a_1\rightarrow w\rightarrow a_2\rightarrow z\]
disjoint from either $\eta_1$ or $\eta_2$. By considering the first point $u$ of $\eta'$ outside of $B(\overline{f};2^{k-3})$  and the first point $v$ inside $B(\overline{f};2^{k-3})$ appearing along $\eta'$ after $u$, we find the existence of points $u,v$ such that
\[A_1:=\{0\leftrightarrow\overline{f}\}\circ \{\overline{f}\leftrightarrow u\}\circ \{u\leftrightarrow v\} \circ \{v \leftrightarrow z\}\circ \{\overline{f}\leftrightarrow z\}\circ \{z\leftrightarrow n\mathbf{e}_1\}\circ \{u\leftrightarrow B_{2^{k-1}}(u)\}\circ \{v\leftrightarrow B_{2^{k-1}}(v)\}.\]
Applying the BK inequality \eqref{eqn: BK}, we have
\begin{align*}
\mathbb{P}(A_1)&\le C\langle \overline{f}\rangle^{2-d}2^{-4k}\sum_{z:\,|z-\overline{f}|\le 2^{k-3}}\sum_{u,v}\langle \overline{f}-u\rangle^{2-d}\langle u-v\rangle^{2-d}\langle v-z\rangle^{2-d}\langle z-\overline{f}\rangle^{2-d}\langle z-n\mathbf{e}_1\rangle^{2-d}\\
&\le C\langle \overline{f}\rangle^{2-d}2^{-4k}\sum_{z:\,|z-\overline{f}|\le 2^{k-3}} \langle \overline{f}-z\rangle^{8-2d}\langle z-n\mathbf{e}_1\rangle^{2-d}.
\end{align*}
Note the bound
\begin{equation}\label{eqn: d-cases}
\sum_{z:\,|z-\overline{f}|\le 2^{k-3}} \langle \overline{f}-z\rangle^{8-2d}\langle z-n\mathbf{e}_1\rangle^{2-d}\leq C\langle \overline{f}-n\mathbf{e}_1\rangle^{2-d}\times \begin{cases} 1, & d>8\\
k, & d=8 \\
2^{k(8-d)},& \quad d< 8\
\end{cases}.
\end{equation}
Multiplying by $k^3 2^{2k}$, we find the estimate
\begin{align*}
    &\sum_{k\ge \log L} k^32^{2k}\mathbb{P}(A_1)\\
    \le & C\langle \overline{f}\rangle^{2-d}\langle \overline{f}-n\mathbf{e}_1\rangle^{2-d}\sum_{k\ge \log L} k^3 2^{2k}2^{k(8-d)-4k}\\
    \le & CL^{d-6} (\log L)^3 \langle \overline{f}\rangle^{2-d}\langle \overline{f}-n\mathbf{e}_1\rangle^{2-d}.
\end{align*}
Summing over $f$ and multiplying by $n^{d-4}$, we obtain an estimate of 
\[CL^{d-6}(\log L)^3.\]

\paragraph{Case B.2: $a_1$ and $a_2$ in different segments.} We now treat the case where $a_1\in \eta_1$ and $a_2\in \eta_2$ (or vice-versa.) 

\paragraph{Case B.2.1: $a_1,a_2\in B(\overline{f};2^{k-1})$.} If both $a_1$ and $a_2$ lie inside $B(\overline{f}; 2^{k-1})$, then we have the event $A_2$ defined by
\[\{0\leftrightarrow\overline{f}\}\circ \{\overline{f}\leftrightarrow a_1\}\circ \{a_1\leftrightarrow z\}\circ \{\overline{f}\leftrightarrow a_2\}\circ \{a_2\leftrightarrow z\}\circ \{z\leftrightarrow n\mathbf{e}_1\}\circ \{a_1\leftrightarrow B_{2^{k-1}}(a_1)\}\circ \{a_2\leftrightarrow B_{2^{k-1}}(a_2)\}.\]
Using the one-arm exponent \eqref{eqn: KN}, we have the estimate
\begin{align*}
&\mathbb{P}(A_2)\\
\le~&C2^{-4k}\langle \overline{f}\rangle^{2-d}\langle z-n\mathbf{e}_1\rangle^{2-d}\sum_{a_1,a_2} \langle \overline{f}-a_1\rangle^{2-d}\langle a_1-z\rangle^{2-d}\langle \overline{f}-a_2\rangle^{2-d}\langle a_2-z\rangle^{2-d}\\
\le~&C2^{-4k}\langle \overline{f}\rangle^{2-d}\langle z-n\mathbf{e}_1\rangle^{2-d}\langle \overline{f}-z\rangle^{8-2d}.
\end{align*}
To take the sum over $z$, we use \eqref{eqn: d-cases} again, and conclude as previously.

\paragraph{Case B.2.2: $a_1\notin B(\overline{f};2^{k-1})$.} It remains to consider the case where one of $a_1$ and $a_2$, say $a_1$, lies outside $B(\overline{f},2^{k-1})$, then we have
the event
\[\{0\leftrightarrow \overline{f}\}\circ\{\overline{f}\leftrightarrow a_1\}\circ\{a_1\leftrightarrow z\} \circ \{a_1\leftrightarrow a_2\} \circ\{\overline{f}\leftrightarrow a_2\}\circ \{a_2\leftrightarrow z\} \circ \{z \leftrightarrow n\mathbf{e}_1\}.\]
Since $|\overline{f}-a_1|,|z-a_1|\ge 2^{k-2}$, we have
\[\mathbb{P}(\overline{f}\leftrightarrow a_1 \circ a_1\leftrightarrow z )\le C2^{-2(d-2)k}.\]
We again distinguish two cases.

\paragraph{Case B.2.2.1: $a_2\in B(\overline{f},2^{k-2})$.} If $a_2\in B(\overline{f},2^{k-2})$, then we also have $|a_1-a_2|\ge 2^{k-2}$ and consequently 
\[\mathbb{P}(a_1\leftrightarrow a_2)\le C2^{-(d-2)k}.\]
In this case, we estimate as follows
\begin{align*}
&\mathbb{P}(2^k \leq \mathrm{diam}(bubble(\overline{f})) <  2^{k+1}, head(bubble(\overline{f}))=z,  \text{$f$ open pivotal})\\
    \leq~& C\langle \overline{f}\rangle^{2-d}2^{-(d-2)k}\sum_{a_1\in Ann(\overline{f};2^{k},2^{k+1})} 2^{-2kd+4k}\sum_{ z\in B(\overline{f};2^{k-1})}\sum_{a_2} \langle \overline{f}-a_2\rangle^{2-d}\langle a_2-z\rangle^{2-d}\langle z-n\mathbf{e}_1\rangle^{2-d}\\
    \leq~&C\langle \overline{f}\rangle^{2-d}2^{-(d-2)k}\sum_{a_1\in Ann(\overline{f};2^{k},2^{k+1})} 2^{-2kd+4k}\sum_{ z\in B(\overline{f};2^{k-1})}\langle \overline{f}-z\rangle^{4-d}\langle z-n\mathbf{e}_1\rangle^{2-d}\\
    \leq~&C k2^{-2kd+10k}\langle \overline{f}\rangle^{2-d}\langle \overline{f}-n\mathbf{e}_1\rangle^{2-d}.
\end{align*}
In the final step, we have bounded
\[\langle \overline{f}-z\rangle^{4-d}\le C2^{4k}\langle\overline{f}-z\rangle^{-d}\]
before performing the sum over $z$
\[\sum_{ z\in B(\overline{f};2^{k-1})}\langle \overline{f}-z\rangle^{-d}\langle z-n\mathbf{e}_1\rangle^{2-d}\le Ck\langle f-n\mathbf{e}_1\rangle^{2-d}\]
Multiplying by $n^{d-4}2^{2k}k^3$ and summing over $f$ and $k\ge \log_2( L/2)$ gives an estimate of
\[CL^{12-2d}(\log L)^4=CL^{2(6-d)}(\log L)^4.\]
\paragraph{Case B.2.2.2: $a_2\notin B(\overline{f},2^{k-2})$.} Finally, we consider the case where $|a_2-\overline{f}|\ge 2^{k-2}$. In this case, we use the estimate
\[\mathbb{P}(\overline{f}\leftrightarrow a_2\circ a_2\leftrightarrow z)\le C\langle \overline{f}-a_2\rangle^{-d+2}\langle a_2-z\rangle^{-d+2}\le 2^{-2(d-2)k}.\]
We decompose according to the distance between $a_1$ and $a_2$ and use
\[\#\{a_2: 2^\ell \le  |a_1-a_2|\le 2^{\ell+1}\}\le C2^{d\ell}.\]
\begin{align*}
&\mathbb{P}(2^k \leq \mathrm{diam}(bubble(\overline{f})) <  2^{k+1}, head(bubble(\overline{f}))=z,  \text{$f$ open pivotal})\\
    \leq~& C\langle \overline{f}\rangle^{2-d}2^{-2(d-2)k}\sum_{a_1\in Ann(\overline{f};2^{k},2^{k+1})} 2^{4k-2kd}\sum_{\ell=1}^{k}2^{-\ell(d-2)}2^{d\ell} \sum_{ z\in B(\overline{f};2^{k-1})}\langle z-n\mathbf{e}_1\rangle^{2-d}\\
    \leq~& C\langle \overline{f}\rangle^{2-d} 2^{10k-2kd}\langle \overline{f}-n\mathbf{e}_1\rangle^{2-d}.
\end{align*}
Multiplying by $2^{2k}k^3$ and summing, we get an estimate of
\[CL^{2(6-d)}(\log L)^3.\]
\end{proof}

\subsection{Truncation}\label{sec: truncate}
In this section, we replace the quantity 
\begin{equation}\label{eqn: dist-chem}
    D(bubble(\overline{f}))\mathbbm{1}[f \text{ pivotal for } 0\lra n\mathbf{e}_1]
\end{equation} by a truncated distance 
\begin{equation}\label{eqn: dist-trunc}
\mathrm{dist}_{bubble_V(\overline{f})}\big(\overline{f}, S_L(f,V)\big)\mathbbm{1}[f \text{ pivotal for } 0\lra n\mathbf{e}_1]
\end{equation} 
with $V=C^f(0)$ and $S_L(f,V)$ as defined in \eqref{eqn: SLVf}.
\begin{proposition}\label{prop: local-trunc}
We have:
\[n^{-2}\sum_{f\in [-Mn,Mn]^d} \mathbb{E}\Big( D(bubble(\overline{f})), D(bubble(\overline{f}))\neq \mathrm{dist}_{bubble(\overline{f})}\big(\overline{f},S_L(f,C^f(0))\big), f \text{ pivotal }\mid 0\lra n\mathbf{e}_1\Big)=O(L^{-1/2}),\]
uniformly in $M$.

\end{proposition}
\begin{proof}
Let 
\[\mathrm{Diff}_L =\{ D(bubble(\overline{f}))\neq \mathrm{dist}_{bubble(\overline{f})}\big(\overline{f},S_L(f,C^f(0))\big)\}.\]
By Proposition \ref{prop: large-bubble-trunc}, we have
\[n^{-2}\sum_{f\in [-Mn,Mn]^d} \mathbb{E}\Big( D(bubble(\overline{f})), \mathrm{Diff}_L, \mathrm{diam}(bubble(\overline{f}))>L^{\frac{1}{2}}, f \text{ pivotal }\mid 0\lra n\mathbf{e}_1\Big)\le CL^{\frac{6-d}{2}},\]
so it suffices to consider the complementary sum, where the diameter of each bubble is limited to $L^{\frac{1}{2}}$. Reproducing the estimates in \eqref{eqn: dist-control}, we can moreover assume that $D(bubble(\overline{f}))\le CL(\log L)^3$, at the cost of an error of order $2^{-c(\log L)^3}$.

It suffices to estimate the quantity
\[n^{-2}L(\log L)^3\sum_{f\in [-Mn,Mn]^d} \mathbb{P}\Big( \mathrm{Diff}_L, f \text{ pivotal }\mid 0\lra n\mathbf{e}_1\Big).\]

The distance \eqref{eqn: dist-chem} is 
\begin{equation}
\label{eqn: dist-not-trunc}
\mathrm{dist}(\overline{f},\{x\in \mathbb{Z}^d\setminus C^f(0) : x \leftrightarrow n\mathbf{e}_1 \text{ off }  bubble(\overline{f})\})
\end{equation}
while the truncation is 
\[\mathrm{dist}(\overline{f},\{x\in \mathbb{Z}^d\setminus C^f(0) : x\leftrightarrow \partial B_L(\overline{f}) \text{ off }  bubble(\overline{f})\}).\]
We use the following, to be proved later.
\begin{lemma}\label{lem: distances-differ}
Suppose $f\in \mathcal{E}(\mathbb{Z}^d)$ is pivotal for the connection from $0$ to $n\mathbf{e}_1$, and  the quantities \eqref{eqn: dist-not-trunc}  and \eqref{eqn: dist-trunc} are \emph{not} equal. Then, the event
\begin{equation}\label{eqn: triple-circ}
    \{0\lra \underline{f}\}\circ \{\overline{f}\lra n\mathbf{e}_1\}\circ\{\overline{f}\lra B(\overline{f},L)\}
\end{equation}
occurs.
\end{lemma}

Applying Lemma \ref{lem: distances-differ}, we obtain
\begin{equation}\label{eqn: dist-zwy}
\begin{split}
&\mathbbm{1}[\mathrm{Diff}_L, f \text{ pivotal for } 0\lra n\mathbf{e}_1\big]\\
\le~& \mathbbm{1}[0\lra \underline{f}\circ \overline{f}\lra n\mathbf{e}_1 \circ \overline{f}\lra B(\overline{f},L)].
\end{split}
\end{equation}

By \eqref{eqn: dist-zwy} and the BK inequality, we have
\begin{equation}
\label{eqn: diff-bound}
\begin{split}
    &\mathbb{P}\Big(\mathrm{Diff}_L, f \text{ pivotal for } 0\leftrightarrow n\mathbf{e}_1\Big)\\
    \le~& \mathbb{P}(0\leftrightarrow \underline{f})\mathbb{P}(\overline{f}\lra n\mathbf{e}_1)\mathbb{P}(\overline{f}\leftrightarrow B(\overline{f},L))\\
    \le~&CL^{-2}\langle f\rangle^{2-d} \langle \overline{f}-n\mathbf{e}_1\rangle^{2-d}.
    \end{split}
\end{equation}
In the final step, we have used the estimate
\[\mathbb{P}(\overline{f}\leftrightarrow B(\overline{f},L))\le CL^{-2},\]
which follows from the one-arm estimate \eqref{eqn: KN}.

Summing the above over $f$, we obtain
\begin{align*}
    &CL^{-2}\sum_{f\in [-Mn,Mn]^d}\langle \underline{f}\rangle^{2-d} \langle \overline{f}-n\mathbf{e}_1\rangle^{2-d}\\
    \le &~CL^{-2}\sum_{f\in \mathbb{Z}^d} \langle \underline{f}\rangle^{2-d} \langle \overline{f}-n\mathbf{e}_1\rangle^{2-d}\\
    \le &~CL^{-2}n^{4-d},
\end{align*}
uniformly in $M$, from which it follows easily that
\begin{equation}
    \label{eqn: bubble-prop}
    n^{-2}\sum_{f\in [-Mn,Mn]^d} \mathbb{P}\Big( \mathrm{Diff}_L, f \text{ pivotal }\mid 0\lra n\mathbf{e}_1\Big)=O(L^{-2}),
\end{equation}
and thus
\[n^{-2}L(\log L)^3\sum_{f\in [-Mn,Mn]^d} \mathbb{P}\Big( \mathrm{Diff}_L, f \text{ pivotal }\mid 0\lra n\mathbf{e}_1\Big)=O(L^{-1/2}).\]
\end{proof}

\begin{proof}[Proof of Lemma \ref{lem: distances-differ}]
    We begin by showing that the conditions of the lemma imply that there are sites $z\neq w$ in $bubble(\overline{f})$ such that
\begin{enumerate}
\item $z$ is the head of $bubble(\overline{f})$, and so is connected to $n\mathbf{e}_1$ off $bubble(\overline{f})$,
\item $w$ is connected to $\partial B(\overline{f},L)$ off $bubble(\overline{f})$,
\item the two connections are edge-disjoint.
\end{enumerate}

    Let $z=\mathrm{head}(bubble(\overline{f}))$. By \eqref{eqn: dist-to-head}, we have
    \[D(bubble(\overline{f}))=\mathrm{dist}_{bubble(\overline{f})}(\overline{f},z).\]
    If the distances differ, then there is $w\in bubble(\overline{f})$ realizing the minimum in the distance \[\mathrm{dist}_{bubble(\overline{f})}\big(\overline{f},S_L(f,C^f(0))\big)\] with $z\neq w$. 
    
    By definition, the vertices $z$ and $w$ have connections to $n\mathbf{e}_1$ and $B(\overline{f},L)$, respectively, off $bubble(\overline{f})$. These connections must be vertex-disjoint, since otherwise they would not lie off $bubble(\overline{f})$, as can be seen by considering the first vertex in the connection $z\leftrightarrow n\mathbf{e}_1$ that intersects with a connection from $w$ to $B(\overline{f}, L)$. This vertex has two edge-disjoint connections to $\overline{f}$, and thus is in $bubble(\overline{f})$, a contradiction, so the vertices $z$ and $w$ with properties 1., 2. and 3. exist.

    We now show that \eqref{eqn: triple-circ} occurs. Choose two edge-disjoint open paths $\eta_1$ and $\eta_2$ between $\overline{f}$ and the head $z$. These two paths form a circuit of open edges which we denote by $\eta=\eta_1\cup\eta_2$. 
    
    The vertex $w$ whose existence is noted above has two  edge-disjoint connections to $\overline{f}$, both lying inside $bubble(\overline{f})$. Choose one, and call it $\gamma$. Call the portion of this open path from $w$ until the first vertex belonging to $\eta$, possibly $\overline{f}$, $\gamma'$.  Let $a$ be the first vertex of $\gamma'$ in $\eta$.
    
    The vertex $a$ could coincide with $\overline{f}$. If this is the case, $w$ has a connection to $\overline{f}$ which is edge-disjoint from both $\eta_1$ and $\eta_2$, and so $\overline{f}$ is connected to $B(\overline{f},L)$ edge-disjointly from the concatenation of $\eta_1$ with the connection from $z$ to $n\mathbf{e}_1$ guaranteed by the first part of the lemma.

    Second, if $a$ coincides with $z$, then we can concatenate $\eta_1$ with $\gamma'$ and the connection from $w$ to $B(\overline{f},L)$ off $bubble(\overline{f})$ on the one hand and $\eta_2$ with the connection from $z$ to $n\mathbf{e}_1$ on the other hand to obtain two edge-disjoint open paths.
    
    In the remaining case, $a$ lies in one  of the segments $\eta_1 \setminus \{z,\overline{f}\}$ or $\eta_2\setminus \{z,\overline{f}\}$. Say $a\in \eta_1$. Then, following $\eta_1$ from $\overline{f}$ to $a$, and then $\gamma'$ from $a$ to $w$, and then the connection from $w$ to $B(\overline{f},L)$, we obtain a path from $\overline{f}$ to $B(\overline{f},L)$ which is necessarily edge-disjoint from $\eta_2$ concatenated with the connection from $z$ to $n\mathbf{e}_1$ in 1. above. 
\end{proof}
\begin{figure}[htbp]  
  \centering
  \includegraphics[width=0.50\textwidth]{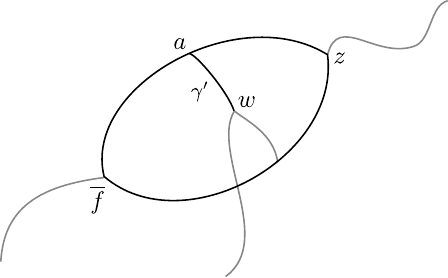}
  \caption{An illustration of the proof of Lemma \ref{lem: distances-differ}}
  \label{fig:differ}
\end{figure}

\section{Limit Statements}

\subsection{Brownian Motion} \label{section: BM}
In the expected integrated super-Brownian scaling limit, the path from $0$ to $n\mathbf{e}_1$ scales to a  Brownian path, with the time along the path representing chemical distance. See Hara and Slade \cite[Section 4]{HS} for partial results on this scaling limit. 

In this section, we identify the distribution corresponding to the limiting moments obtained in Theorem \ref{thm: local-limit}. Up to scaling, it is clear that it suffices to explain how to interpret the integrals $I_k(M)$ in \eqref{eqn: IM}. 
\begin{lemma}\label{eq}
    The quantity 
    \[k!\cdot z_d^k\cdot I_k(M)\]
    is the $k$th moment of the total occupation time of the set $[-M,M]^d$ by a Brownian motion started at $0$ conditioned to hit $\mathbf{e}_1=(1,0,\ldots,0)$ (or, equivalently, conditioned to hit any unit vector). The quantity $z_d$ was defined in \eqref{eqn: zd-def}.
\end{lemma}

For $d>1$, Brownian motion does not hit isolated points. Brownian motion conditioned on the zero-probability event that it hits $\mathbf{e}_1$ eventually is defined by the well-known $h$-transform procedure, with $h$ given by the Green's function for the Laplacian $\frac{1}{2}\Delta$ on $\mathbb{R}^d$. See Doob's classic paper \cite{Doob} for general $h$ transforms of Brownian motion. The particular example of interest to us here appears in \cite[Section 6]{Doob}. See also \cite[Chapter IV.39]{RW} for a more modern treatment of $h$-transforms using SDEs. 

For $d\ge 3$, the Green's function is given by
\begin{equation}
    \label{eqn: G}
G(x,y)=z_d\cdot |x-y|^{2-d}.
\end{equation}
This function is harmonic away from $x=y$. In fact, $G$ solves the distributional equation
\begin{equation}
    -\frac{1}{2}\Delta_x G(\cdot,y)=\delta_y.
\end{equation}
Let 
\[h(x)=\frac{G(x,\mathbf{e}_1)}{G(0,\mathbf{e}_1)}=|x-\mathbf{e}_1|^{2-d}.\]
The function $h$ is a positive harmonic function away from $x=\mathbf{e}_1$. As a consequence, given a $d$-dimensional standard Brownian motion $W_t$, we can define a change of measure by setting, on $\mathcal{F}_t=\sigma(W_s, 0\le s \le t)$ by
\[\frac{\mathrm{d}\mathbb{Q}}{\mathrm{d}\mathbb{P}}\Big|_{\mathcal{F}_t}= h(W_t).\]
By a standard application of Girsanov's formula, this measure is the distribution of the process $X_t$ defined by the SDE
\begin{equation}\label{eqn: Xt}
\begin{split}
\mathrm{d}X_t&=\mathrm{d}B_t+ \nabla\log h(X_t)\,\mathrm{d}t \\
&=\mathrm{d}B_t+(2-d)\frac{X_t-\mathbf{e}_1}{|X_t-\mathbf{e}_1|^2}\,\mathrm{d}t,\\
\end{split}
\end{equation}
with the initial data $X_0=0$. Here $B_t$ is a standard Brownian motion in $\mathbb{R}^d$. The equation \eqref{eqn: Xt} can be solved by standard methods up to a random time 
\[T_d^r=\inf\{t>0:|X_t-\mathbf{e}_1|\le r\},\]
since the drift coefficient is smooth and bounded up to that time. See \cite{LG} for such matters. The time $T_d$ for $X_t$ to reach $\mathbf{e}_1$ is then given by
\begin{equation}
\label{eqn: T0-def}
T_d=\lim_{r\downarrow 0}T_d^r.
\end{equation}
The limit exists by monotonicity. The process $X_t$ for $0\le t<T_d$ is the natural way to define a Brownian motion started at $x=0$ and conditioned to hit $\mathbf{e}_1$ eventually.

Applying It\^o's formula to \eqref{eqn: Xt}, we derive an SDE for the modulus $R_t=|X_t-\mathbf{e}_1|$ 
\[\mathrm{d}R_t=\mathrm{d}B_t+\frac{3-d}{2}\frac{1}{R_t}\,\mathrm{d}t, \quad R_0=|x|,\]
with $B_t$ a standard Brownian motion. $R_t$ is a Bessel process of dimension $4-d$. 
We have 
\[\mathbb{P}(T_d<\infty)=1,\]
see \cite[Proposition 2.1]{Law}.
The random variable $T_d$ has density
\begin{equation}\label{eqn: fdt}
f_d(t)=\frac{1}{2^{\frac{d}{2}-1}\Gamma(\frac{d}{2}-1)}t^{-\frac{d}{2}}e^{-1/2t}.    
\end{equation}
See \cite[Proposition 2.9]{Law}. This is an inverse Gamma distribution.

 We recall the definition of $\psi_d$ from the statement of Theorem~\ref{thm: avoid-dist}:
\[\psi_d(s)=\E\exp(-sT_d), \qquad s \geq 0; \]
in other words, $\psi_d$ is the Laplace transform of the hitting time $T_d$. Suppose that we fix a $\lambda > 0$ and kill the process $(X_t)$ at exponential rate $\lambda$ --- in other words, $X_t$ dies at a random time $T^*$ distributed as an exponential with rate $\lambda$, independent of the process $(X_t)$. Let $(X^*_t)$ be the process $(X_t)$ conditioned to hit $\mathbf{e}_1$ before $T^*$, and let $T_d^*$ be the time this process first hits $\mathbf{e}_1$. The distribution of the variable $T_d^*$ then has density $C \exp(-\lambda T_d)$ with respect to $T_d$, where $C$ is the appropriate normalizing constant:
\[C^{-1} = \int_0^\infty \frac{1}{2^{\frac{d}{2}-1}\Gamma(\frac{d}{2}-1)}t^{-\frac{d}{2}}e^{-1/2t - \lambda t} \, \mathrm{d}t = \psi_d(\lambda)\ .  \]
In particular, the Laplace transform of the distribution of $T_d^*$ is given by 
\begin{equation}
    \label{eq:lapstar}
    \E[\exp(- s T_d^*)] = \frac{\psi_d(s + \lambda)}{\psi_d(\lambda)}\ .
\end{equation}

\begin{proposition}
Let $f:(\mathbb{R}^d)^n\rightarrow \mathbb{R}$ be a bounded Borel function. We have the identity
\begin{equation}\label{eqn: identity}
\mathbb{E}^x[f(X_{t_1},\ldots,X_{t_n}), T_d>t]=\frac{1}{h(x)}\mathbb{E}^x[f(W_{t_1},\ldots, W_{t_n})h(W_t)],
\end{equation}
for $t_1<\ldots<t_n\le t$.
\begin{proof}
    Since $h$ is harmonic away from $\mathbf{e}_1$, this follows from It\^o's formula and Girsanov's theorem by a standard argument. See \cite[Chapter IV.39]{RW} for a similar computation.
\end{proof}
\end{proposition}
The next proposition relates the moments \eqref{eqn: IM} to the occupation measure of the process \eqref{eqn: Xt}.
\begin{proposition}\label{lem: moments}
    Let $X_t$ denote the process defined by \eqref{eqn: Xt}. Let $A\subset \mathbb{R}^d$ be a bounded Borel set and $k$ be an integer. Then
    \begin{align*}
    \mathbb{E}\left[\left(\int_0^{T_d} 1_A(X_t)\,\mathrm{d}t\right)^k\right]&=\frac{k!}{z_d}\int_A\cdots\int_A G(0,x_1)G(x_1,x_2)\cdots G(x_{k-1},x_k)G(x_k,\mathbf{e}_1)\,\mathrm{d}x_1\cdots\mathrm{d}x_k\\
    &=k! z_d^k\int_A\cdots\int_A |x_1|^{-d+2}|x_2-x_1|^{-d+2}\cdots |x_k-x_{k-1}|^{-d+2}|x_k-\mathbf{e}_1|^{-d+2}\,\mathrm{d}x_1\cdots\mathrm{d}x_k.
    \end{align*}
\begin{proof}
    Expanding the $k$th power, we have
    \begin{align*}
    & \mathbb{E}\left[\int_0^\infty \cdots \int_0^\infty 1_A(X_{t_1})\cdots 1_A(X_{t_k})\cdot 1_{t_1<T_d}\cdots 1_{t_k<T_d}\,\mathrm{d}t_1\cdots\mathrm{d}t_k\right]\\
    =~&k!\int\cdots\int_{0<t_1<\cdots<t_k<\infty} \mathbb{E}[1_A(X_{t_1})\cdots 1_A(X_{t_k})1_{t_k<T_d}]\,\mathrm{d}t_1\cdots\mathrm{d}t_k.
    \end{align*}
    Using the identity \eqref{eqn: identity}, the inner expectation is
    \begin{align*}
        &\mathbb{E}[1_A(X_{t_1})\cdots 1_A(X_{t_k}), t_k<T_d]\\
        =~&\mathbb{E}[1_A(W_{t_1})\cdots 1_A(W_{t_k}) \,h(W_{t_k})]\\
        =~&\int_A\cdots \int_A p_{t_1}(0,x_1)p_{t_2-t_1}(x_1,x_2)\cdots p_{t_k-t_{k-1}}(x_{k-1},x_k)|x_k-\mathbf{e}_1|^{-d+2}\,\mathrm{d}x_1\cdots \mathrm{d}x_k,
    \end{align*}
    where
    \[p_t(x,y)=\frac{1}{(2\pi t)^{\frac{d}{2}}}\exp(-\frac{|x-y|^2}{2t})\]
    is the heat kernel associated with $\frac{1}{2}\Delta$ on $\mathbb{R}^d$.
    We now claim that
    \begin{equation}\label{eqn: multiple-greens}
        \begin{split}
        &\int_{0<t_1<\cdots<t_k<\infty} p_{t_1}(x,x_1)p_{t_2-t_1}(x_1,x_2)\cdots p_{t_k-t_{k-1}}(x_k,x_{k-1})\,\mathrm{d}t_1\cdots\mathrm{d}t_k\\
        =~& G(x,x_1)G(x_1,x_2)\cdots G(x_{k-1},x_k).
        \end{split}
    \end{equation}
    To show \eqref{eqn: multiple-greens}, we integrate the time variables successively:
    \[\int_{t_{j-1}}^\infty p_{t_j-t_{j-1}}(x_j,x_{j-1})\,\mathrm{d}t_j=\int_0^\infty p_\tau(x_j,x_{j-1})\,\mathrm{d}\tau=G(x_j,x_{j-1}),\]
    We have used the identity
    \[\int_0^\infty p_\tau(x,y)\,\mathrm{d}\tau = G(x,y).\]
    This is essentially the probabilistic \emph{definition} of the Green's function, but we can check it from the definition \eqref{eqn: G}, starting from the heat kernel:
    \[\int_0^\infty e^{-\frac{|x|^2}{2\tau}}\,\frac{\mathrm{d}\tau}{(2\pi \tau)^{d/2}}=\frac{1}{(2\pi)^{d/2}}|x|^{-d+2}\int_0^\infty \tau^{-d/2}e^{-1/2\tau}\,\mathrm{d}\tau=z_d|x|^{-d+2},\]
    where 
    \[z_d=\frac{\Gamma(\frac{d}{2}-1)}{2\pi^{\frac{d}{2}}}.\]
    Here we have scaled the variable of integration by $|x|^2$ and used the definition of the Gamma function. Inserting \eqref{eqn: multiple-greens} into the integrals above, we obtain the result.
\end{proof}
\end{proposition}
Lemma \ref{eq} follows immediately from the previous result applied to $A=[-M,M]^d$. We also note the following.
\begin{proposition} \label{prop: X-lim}
    Let 
    \begin{equation}\label{eqn: XM-def}
    X(M):=\int_0^{T_d}1_{[-M,M]^d}(X_t)\,\mathrm{d}t.
    \end{equation}
    Then $X(M)$ converges in distribution to the random variable $T_0$, with density \eqref{eqn: fdt}: for any bounded Lipschitz function $f$,
    \[\mathbb{E}[f(X(M))]=\mathbb{E}[f(T_d)]+o_{M\rightarrow \infty}(1).\]
    \begin{proof}
        $X(M)$ is monotone in $M$, so in fact $X(M)$ converges almost surely to $T_0$.
    \end{proof}
\end{proposition}

\subsection{Proof of Theorem \ref{thm: main}}\label{sec: final-proof}
 Recall from the statement of Theorem \ref{thm: main} that we define sequences of random variables $D_n$, $R_n$  by
    \[D_n:=\mathrm{dist}(0,n\mathbf{e}_1)\]
    and 
    \[R_n:=R_{\mathrm{eff}}(0,n\mathbf{e}_1),\] and let $P_n$ denote the number of pivotal edges for connection  $0\leftrightarrow n\mathbf{e}_1$ ($P_n$ is set to zero if no connection exists.)
    We further introduce the normalization constants
    \begin{align*}
      a_D&:=\frac{z_d}{2d\alpha_d\mathbf{c}\beta},\\
      a_R&:=\frac{z_d}{2d\alpha_r\mathbf{c}\beta },\\
      a_P&:= \frac{z_d}{2d\alpha_p\mathbf{c}\beta }.
    \end{align*}

Let $A_n = \{0 \lra n \mathbf{e}_1 \text{ and there exists a pivotal edge}\}$. Note that $\mathbb{P}(0 \lra n \mathbf{e}_1) / \mathbb{P}(A_n) \to 1$, so it suffices to show the statement claimed in the lemma conditional on $A_n$ instead of $\{0 \lra n \mathbf{e}_1\}$.  Let $Y_n$ denote any of the three quantities $D_n$, $R_n$ or $P_n$ and let $a_*$ denote the corresponding normalization constant. We begin by showing the convergence in distribution of $Y_n / n^2$ for this fixed choice of $Y_n$; we then extend this to the full joint convergence in the theorem statement. By the usual argument, it suffices to show that the sequence
\[\mathbb{E}[f(n^{-2}Y_n)\mid A_n]\]
converges to 
\[\mathbb{E}[f(a_*^{-1} T_d)]\]
for any bounded, 1-Lipschitz function.

Let
\begin{align}
D_n^M&=\sum_{e\in B(Mn)}\big(1+D(bubble(\bar{e}))\big)\mathbbm{1}_{e \text{ pivotal for } \{0\leftrightarrow n \mathbf{e}_1\}} \label{eqn: LnM},\\
R_n^M&=\sum_{e\in B(Mn)}\big(1+r_{bubble(\bar{e})}\big)\mathbbm{1}_{e \text{ pivotal for } \{0\leftrightarrow n \mathbf{e}_1\}} \label{eqn: RnM},\\
P_n^M&=\sum_{e\in B(Mn)}\mathbbm{1}_{e \text{ pivotal for } \{0\leftrightarrow n \mathbf{e}_1\}}. \label{eqn: PnM}
\end{align}
These are truncated versions of the quantities $D_n$, $R_n$ and $P_n$ where only the contribution from pivotals inside the box $[-Mn,Mn]^d$ is considered. By Proposition \ref{prop: outside-M},
\begin{equation}\label{eqn: M-trunc-main-estimate}
\mathbb{E}[|D_n^M-D_n|]\le  Cn^{-d+4}M^{-d+4},
\end{equation}
with similar estimates for $R_n$ and $P_n$. 

Let $Y_n^M$ be any of the truncated quantities introduced in \eqref{eqn: LnM}, \eqref{eqn: RnM}, \eqref{eqn: PnM}, so that
\begin{align*}
Y_n&=\sum_{e\in \mathcal{E}(\mathbb{Z}^d)}h(\bar{e})\mathbbm{1}_{e \text{ pivotal for } \{ 0\leftrightarrow n \mathbf{e}_1\}},\\
Y_n^M&=\sum_{e\in B(Mn)}h(\bar{e})\mathbbm{1}_{e \text{ pivotal for } \{ 0\leftrightarrow n \mathbf{e}_1\}},
\end{align*}
with $h(\bar{e})=1+D(bubble(\bar{e}))$, $h(\bar{e})=1+r_{bubble(\bar{e})}$, or $h=1$.
It follows from \eqref{eqn: M-trunc-main-estimate} and the corresponding estimates for $R_n$ and $P_n$ that
\begin{equation}n^{-2}\mathbb{E}[|Y_n^M-Y_n|\mid 0\leftrightarrow n\mathbf{e}_1]\le CM^{-d+4}.
\end{equation}
In particular, for a 1-Lipschitz function $f$, we have
\begin{equation}\label{eqn: M-trunc-result}
\mathbb{E}[f(n^{-2}Y_n)\mid 0\leftrightarrow n\mathbf{e}_1] = \mathbb{E}[f(n^{-2}Y_n^M)\mid 0\leftrightarrow n\mathbf{e}_1]+O(M^{-d+4}).
\end{equation}

\subsubsection*{$L$-truncation}
We introduce the decomposition    
    \begin{align*}
   Y_n^M&= \sum_{e\in B(Mn)}h(\bar{e})\mathbbm{1}_{e \text{ pivotal for } 0\leftrightarrow n \mathbf{e}_1}\\
   &=\sum_{e\in B(Mn)} h(\bar{e})\mathbbm{1}_{diam(bubble(\bar{e}))>L/2}+\sum_{e\in B(Mn)} h(\bar{e})\mathbbm{1}_{0<diam(bubble(\bar{e}))\le L/2}\\
   &=:Y^M_{n,>L}+Y^M_{n,L}
    \end{align*}
    In the second equality, we take $\diam(bubble(\overline{e}))>0$ to include the event that $e$ is pivotal for $0\leftrightarrow n\mathbf{e}_1$.
    By Proposition \ref{prop: large-bubble-trunc} in Section \ref{sec: control}, we have 
    \begin{equation}\label{eqn: P-truncate}
    n^{-2}\mathbb{E}[Y^M_{n,>L} \mid 0\lra n\mathbf{e}_1]\le C(\log L)^3L^{6-d}
    \end{equation}
    uniformly in $n$ and $M$.

    Let $f:\mathbb{R}\rightarrow\mathbb{R}$ be a bounded, 1-Lipschitz function. Write:
    \begin{equation*}
      \mathbb{E}[f(n^{-2}Y^M_{n})\mid 0\leftrightarrow  n\mathbf{e}_1]=\mathbb{E}[f(n^{-2}Y^M_{n,L})\mid 0\leftrightarrow  n\mathbf{e}_1]+\mathbb{E}[f(n^{-2}Y^M_n)-f(n^{-2}Y^M_{n,L})\mid 0\leftrightarrow n\mathbf{e}_1].
    \end{equation*}
    Then by \eqref{eqn: P-truncate}, the second term is bounded by
    \[C\|f\|_{\mathrm{Lip}} (\log L)^3 L^{6-d},\]
   so we have
    \begin{equation}\label{eqn: dist-trunc-bound}
      \mathbb{E}[f(n^{-2}Y^M_{n})\mid 0\leftrightarrow  n\mathbf{e}_1]=\mathbb{E}[f(n^{-2}Y^M_{n,L})\mid 0\leftrightarrow  n\mathbf{e}_1]+O(L^{-2}).
    \end{equation}

    Next, we replace $h(e_i)$ by the truncated quantities introduced in Section \ref{sec: local-v}. This step is only needed if we are not considering the number of pivotals. Let $h_L$ denote either of the truncated quantities \eqref{eqn: L-bubble} or \eqref{eqn: L-resistance} and 
    \[Y^{M,\mathrm{loc}}_{n,L}:=\sum_{e=(z,w)\in \mathcal{E}(\mathbb{Z}^d)} h_L(e)\mathbf{1}_{\{e \text{ pivotal for } 0\lra n\mathbf{e}_1\}}.\]
    We now write
    \begin{equation*}
    \begin{split}
    &\mathbb{E}[f(n^{-2}Y^M_{n,L})\mid 0\leftrightarrow  n\mathbf{e}_1]\\
    =~&\mathbb{E}[f(n^{-2}Y^{M,
    \mathrm{loc}
    }_{n,L})\mid 
    0\leftrightarrow  n\mathbf{e}_1]+\mathbb{E}[f(n^{-2}Y^M_{n,L})-f(n^{-2}Y^{M,
    \mathrm{loc}
    }_{n,L})\mid 
    0\leftrightarrow  n\mathbf{e}_1]
    \end{split}
    \end{equation*}
    The second term is bounded by
    \[
    |\mathbb{E}[f(n^{-2}Y^M_{n,L})-f(n^{-2}Y^{M,
    \mathrm{loc}
    }_{n,L})\mid 
    0\leftrightarrow  n\mathbf{e}_1]|\le \|f\|_{\mathrm{Lip}}\mathbb{E}[|Y^M_{n,L}- Y^{M,
    \mathrm{loc}
    }_{n,L}|\mid  0\leftrightarrow  n\mathbf{e}_1].\]
    
    By Proposition \ref{prop: local-trunc}, in Section \ref{sec: truncate}, this is $O(L^{-1/2})$, so we have the approximation:
    \begin{equation}\label{eqn: pre-trunc}
        \mathbb{E}[f(n^{-2}Y^M_{n,L})\mid 0\leftrightarrow  n\mathbf{e}_1] =\mathbb{E}[f(n^{-2}Y^{M,
    \mathrm{loc}
    }_{n,L})\mid 
    0\leftrightarrow  n\mathbf{e}_1]+o_{L\rightarrow\infty}(1).
    \end{equation}

\begin{proposition} \label{prop: moment-conv}
    Let $\nu_n^{M,L}$ be the distribution of 
    \[\frac{z_d Y^{M,
    \mathrm{loc}
    }_{n,L}}{2d \mathbf{c}\alpha_{h_L}\beta n^2}\] 
    under $\mathbb{P}(\,\cdot \mid 0\leftrightarrow n\mathbf{e}_1)$. Then $\nu_n^{M,L}$ converges weakly to the law of $X(M)$ defined in \eqref{eqn: XM-def}.
\end{proposition}
\begin{proof}
    Theorem \ref{thm: local-limit} implies convergence of the moments of $Y^{\mathrm{loc}}_{M,L}$, convergence of moments and Lemma \ref{lem: moments} identifies the limiting distribution.
\end{proof}

\begin{proposition}\label{prop: alphaL} The quantities $\alpha_{r,L}$ and $\alpha_{d,L}$ in \eqref{eqn: L-bubble} and \eqref{eqn: L-resistance} converge as $L\rightarrow\infty$. We denote the limits by $\alpha_r$ and $\alpha_d$, respectively.
\end{proposition}
\begin{proof}
We give the proof for the case of the distance. The resistance is handled similarly. We begin by writing
\[ \alpha_{d,L}=\mathbb{E}_{\nu_{\mathbf{e}_1}} \otimes \mathbb{E}_{\tilde{\nu}_{0}}[\mathbbm{1}_{\exists\,\text{a path to}\,\infty\,\text{in}\,W_{\mathbf{e}_1}\,\text{off}\,\widetilde{W}_{0}} \cdot (1+\mathrm{dist}_{bubble_{\widetilde{W}_{0}}(\vec{\mathbf{e}}_1)}\big(\mathbf{e}_1,S_L({\vec{\mathbf{e}}_1},\widetilde{W}_0)\big))\mathbbm{1}_{\mathrm{diam}(bubble(\mathbf{e}_1))\le L/2}],\]
with $\vec{\mathbf{e}}_1=(0,\mathbf{e}_1)$.

By Monotone Convergence, it suffices to note that the integrand 
\begin{equation}\label{eqn: monotone-lim}
\mathrm{dist}_{bubble_{\widetilde{W}_{0}}(\vec{\mathbf{e}}_1)}\big(\mathbf{e}_1,S_L({\vec{\mathbf{e}}_1},\widetilde{W}_0)\big)\mathbbm{1}_{\mathrm{diam}(bubble(\mathbf{e}_1))\le L/2}
\end{equation}
increases to 
\begin{equation}\label{eqn: S_infty}
\mathrm{dist}_{bubble_{\widetilde{W}_{0}}(\vec{\mathbf{e}}_1)}\big(\mathbf{e}_1,S_\infty({\vec{\mathbf{e}}_1},\widetilde{W}_{0})\big)
\end{equation}
almost surely with respect to the product measure. (See \eqref{eqn: SLVf} and \eqref{eqn: SinftyV} for the definitions of $S_L$ and $S_\infty$.) To see this, note that
\begin{align*}
&\beta\cdot \tilde{\mathbb{E}}[\mathbb{P}(\mathrm{diam}(bubble(\vec{\mathbf{e}}_1))> L/2, \tilde{C}(0)\cap C(\mathbf{e}_1)=\emptyset\mid \mathbf{e}_1\lra n\mathbf{e}_1)\mid 0\lra  -m\mathbf{e}_1]\\
=~&\frac{\mathbb{P}(\vec{\mathbf{e}}_1 \text{ pivotal  for } -m\mathbf{e}_1\leftrightarrow n\mathbf{e}_1, \mathrm{diam}(bubble(\vec{\mathbf{e}}_1))> L/2)}{\mathbb{P}(0\lra  -m\mathbf{e}_1)\mathbb{P}(\mathbf{e}_1 \lra n\mathbf{e}_1)}.
\end{align*}
A straightforward modification of Lemma \ref{lem: distances-differ} shows that the event in the probability in the numerator implies
\[\{-m\mathbf{e}_1\lra \mathbf{e}_1 \}\circ \{\mathbf{e}_1\lra n\mathbf{e}_1\}\circ \{\mathbf{e}_1\lra \partial B(\mathbf{e}_1,L/2)\},\]
so
\[\mathbb{P}(\vec{\mathbf{e}}_1 \text{ pivotal  for } -m\mathbf{e}_1\leftrightarrow n\mathbf{e}_1, \mathrm{diam}(bubble(\vec{\mathbf{e}}_1))> L/2)\le CL^{-2}n^{-d+2}m^{-d+2}.\]
From this, we infer that 
\[\tilde{\mathbb{E}}[\mathbb{P}(\mathrm{diam}(bubble(\vec{\mathbf{e}}_1))> L/2, \tilde{C}(0)\neq C(\mathbf{e}_1)\mid \mathbf{e}_1\lra n\mathbf{e}_1)\mid 0\lra  -m\mathbf{e}_1]\le CL^{-2}\]
uniformly in $n,m\in \mathbb{Z}_+$, so
\[(\nu_{\mathbf{e}_1}\otimes \tilde{\nu}_0)\big(\mathrm{diam}(bubble(\mathbf{e}_1))> L/2\big)\rightarrow 0\]
as $L\rightarrow \infty$, so indeed
\[\mathbbm{1}_{\mathrm{diam}(bubble(\mathbf{e}_1))\le L/2}\uparrow 1\]
$(\nu_{\mathbf{e}_1}\otimes \tilde{\nu}_0)$-almost surely. 

As for the distance $\mathrm{dist}_{bubble(W_{\mathbf{e}_1})}\big(\mathbf{e}_1,S_L({\vec{\mathbf{e}}_1},W_{\mathbf{e}_1})\big)$, it is increasing in $L$ since
\begin{equation}\label{eqn: increasing-SL} 
S_L(\vec{\mathbf{e}_1},\widetilde{W}_{0})\subset S_{L'}(\vec{\mathbf{e}_1},\widetilde{W}_{0})
\end{equation}
for $L'\le L$. That it increases $(\nu_{\mathbf{e}_1}\otimes \tilde{\nu}_0)$-almost surely to its limit \eqref{eqn: S_infty} follows from \eqref{eqn: diff-bound}. Let 
\begin{align*}A_\infty&= \cap_L \{\mathrm{dist}_{bubble_{\widetilde{W}_{0}}(\vec{\mathbf{e}}_1)}\big(\mathbf{e}_1,S_L({\vec{\mathbf{e}}_1},\widetilde{W}_{0})\big) \neq  \mathrm{dist}_{bubble_{\widetilde{W}_{0}}(\vec{\mathbf{e}}_1)}\big(\mathbf{e}_1,S_\infty({\vec{\mathbf{e}}_1},\widetilde{W}_{0})\big) \}\\
&:=\cap_L A_L,
\end{align*}
Then, as in the proof of \eqref{eqn: diff-bound} in Proposition \ref{prop: local-trunc}  we have
\[\tilde{\mathbb{E}}[\mathbb{P}(A_L\mid \mathbf{e}_1\lra n\mathbf{e}_1)\mid 0\lra -m\mathbf{e}_1 ] \le CL^{-2}\]
uniformly in $n$ and $m$, so 
\[(\nu_{\mathbf{e}_1}\otimes \tilde{\nu}_0)(A_\infty)=0,\]
confirming the almost-sure monotone convergence of \eqref{eqn: monotone-lim} to \eqref{eqn: S_infty}.
\end{proof}

\subsubsection{Conclusion and Proof of Theorem~\ref{thm: main}}
We begin with some preliminary results summarizing our work in the last subsection. Our first result is a weaker, univariate version of Theorem \ref{thm: main}:
\begin{lemma}\label{lem:weak}
Recall the variables $X_n$, $Y_n$, and $Z_n$ defined at \eqref{eq:XYZ}. We have:
\begin{equation}
    \label{eq:weakthm1}
    \begin{gathered}
        \text{Conditional on $\{0 \lra n \mathbf{e}_1\}$, each of the variables $X_n$, $Y_n$, and $Z_n$}\\
        \text{converges in distribution to a variable having Lebesgue density \eqref{eqn: limit-density}}.
    \end{gathered}
\end{equation}
    \end{lemma}
\begin{proof}
Let $f$ be a bounded, $1$-Lipschitz function. Let $a_{*,L}$ be obtained from $a_*$ by replacing $\alpha_d$, $\alpha_r$ by the appropriate truncated versions $\alpha_{d,L}$ and $\alpha_{r,L}$, while $a_{P,L}=a_P$. Combining \eqref{eqn: M-trunc-result}, \eqref{eqn: dist-trunc-bound} and \eqref{eqn: pre-trunc}, we have
\[\mathbb{E}[f(n^{-2} Y_n)\mid 0\leftrightarrow n\mathbf{e}_1]=\mathbb{E}[f(n^{-2}Y_{n,L}^{M,\mathrm{loc}})\mid 0\leftrightarrow n\mathbf{e}_1]+O(M^{-d+4})+o_{L\rightarrow\infty} (1).\]
By Proposition \ref{prop: moment-conv}, 
\[\mathbb{E}[f(n^{-2}Y_n)\mid 0\leftrightarrow n\mathbf{e}_1]=\mathbb{E}[f(a_{*,L}^{-1}\cdot X(M))]+o_{n\rightarrow\infty}(1),\]
where $X(M)$ is the random variable \eqref{eqn: XM-def}. 
From Proposition \ref{prop: alphaL}, we have
\[\mathbb{E}[f(a_{*,L}^{-1}\cdot  X(M))]=\mathbb{E}[f(a_*^{-1}\cdot X(M))]+o_{L\rightarrow\infty}(1).\]
Finally, by Proposition \ref{prop: X-lim}, we find
\[\mathbb{E}[f(a_*^{-1}\cdot  X(M))]=\mathbb{E}[f(a_*^{-1}\cdot  T_d)]+o_{M\rightarrow\infty}(1).\]
Taking the limits $n\rightarrow\infty$, $L\rightarrow \infty$ and $M\rightarrow \infty$ in that order concludes the proof of Lemma~\ref{lem:weak}.\end{proof}

We now extend \eqref{eq:weakthm1} to the full joint convergence claimed in Theorem~\ref{thm: main}. We do this by showing that in the limit, sums of local random variables along pivotal edges as considered in Theorem~\ref{thm: local-limit} are fully correlated. Fix $M$ and $L$, let $g^{(1)}, g^{(2)}$ be two
families of local variables (satisfying (iv) jointly), and let
$Y^{(1)}, Y^{(2)}$ denote the corresponding sums over pivotal edges in
$B(Mn)$ of the local approximations as constructed above. Write
$\alpha_i := \alpha_{g^{(i)}}$ and suppose $\alpha_2>0$.
\begin{proposition}\label{prop: diff-sq}
With the notation above,
\begin{equation}\label{eq:L2}
\lim_{n\to\infty}
\E\Big[\Big(n^{-2}\Big(Y^{(1)}-\tfrac{\alpha_1}{\alpha_2}\,Y^{(2)}\Big)
\Big)^2 \;\Big|\; 0\lra n\mathbf{e}_1\Big] = 0 .
\end{equation}
\end{proposition}
\begin{proof}[Proof of Proposition \ref{prop: diff-sq}]
We expand the square in \eqref{eq:L2} and use linearity of (conditional expectation) to split the sum in \eqref{eq:L2} into three terms, each of which has the form
\[\text{const} \times \frac{1}{n^4} E\Big[ \sum_{f_1, f_2 \in B(M n)^2} g^{(i)}(f_1, \zeta_{0}^*) g^{(j)}(f_2, \zeta_{1}^*); f_1, f_2 \text{ open pivotals for }  0\lra n\mathbf{e}_1\Big]. \]
The three terms just referenced correspond to the cases $i = j = 1$, $i = j = 2$, and $i = 1, j = 2$.
Each of these three sums can immediately be dealt with using our limiting results for sums of local variables since, the product of two families of local
variables is again a family of local variables. Applying the $H \equiv 1$ case of Theorem \ref{thm: local-mixed} with $k=2$ to the case $i \neq j$ and Theorem~\ref{thm: local-limit} with $k = 2$ to the case $i = j$, and dividing by
$\tau(0,ne_1)=(1+o(1))\mathbf{c}n^{2-d}$, we have
\[
\E\big[n^{-4}\,Y^{(i)}Y^{(j)}\mid 0\lra n\mathbf{e}_1\big]
= 2\,\mathbf{c}^{2}\beta^2(2d)^2\,\alpha_i\alpha_j\, I_2(M) + O(n^{-2}),
\qquad i,j\in\{1,2\}.
\]
The coefficient is the same for all pairs $(i,j)$ apart from the factor
$\alpha_i\alpha_j$, so with $r=\alpha_1/\alpha_2$ the left side of
\eqref{eq:L2} equals
\[
2\,\mathbf{c}^{2}\beta^2(2d)^2 I_2(M)\,\big(\alpha_1 - r\alpha_2\big)^2 + o(1)
= o(1). \qedhere
\]
\end{proof}

 We are now ready to conclude the proof of our main theorem.
\begin{proof}[Proof of Theorem  \ref{thm: main}]
It suffices to prove the joint convergence with $Z_n$; the argument for
$Y_n$ is identical, and the ratio statements follow since the limit $T_d$ is
almost surely positive. By Propositions \ref{prop: outside-M} and \ref{prop: local-trunc}, together with \eqref{eqn: pre-trunc}
applied to the coordinate
$D_n$ (no truncation being needed for $P_n$), for $F:\mathbb{R}^2\to
\mathbb{R}$ bounded and $1$-Lipschitz,
\[
\E\big[F(X_n,Z_n)\mid 0\lra ne_1\big]
= \E\Big[F\Big(\frac{z_d\, P^M_n}{2d\alpha_p \mathbf{c}\beta n^2},\,
\frac{z_d\, Y^{M,\mathrm{loc}}_{n,L}}{2d\alpha_{d,L} \mathbf{c}\beta n^2}\Big)
\,\Big|\, 0\lra n\mathbf{e}_1\Big]
+ O(M^{-d+4}) + o_{L\to\infty}(1),
\]
uniformly in $n$. 

By Proposition \ref{prop: diff-sq} applied to the pair
$(g_{d,L}, \mathbf 1)$, the second coordinate equals
$\alpha_{d,L}/\alpha_p$ times the first plus a term tending to $0$ in
conditional $L^2$. After renormalization by $\alpha_{d,L}$ the two
coordinates differ by $o(1)$ in conditional probability. 

By Proposition \ref{prop: moment-conv},
the first coordinate converges in distribution to $X(M)$;
by Slutsky's theorem the pair converges in distribution
$(X(M),X(M))$. Taking $n\to\infty$, then $L\to\infty$ using Proposition \ref{prop: alphaL}, then $M\to\infty$
using Proposition \ref{prop: X-lim} ( $\alpha_{d,L}$ replaced by $\alpha_d$), 
 together with the tightness of
$n^{-2}P_n$ (Theorem 3 with $k=1$) gives the claim.
\end{proof}

\section{Proof of Theorem \ref{thm: avoid-dist}}
Theorem \ref{thm: avoid-dist} will follow from the following
\begin{proposition}\label{prop: avoidH}
    Let $H$ be a bounded random variable measurable with respect to $C_{B(K_H)}(0)$, $K_H$ fixed.  Then, for every $\lambda\ge 0$ and $y\neq 0$, with $y_n=\lfloor ny \rfloor$,
    \[\lim_{n\rightarrow\infty}\mathbb{E}[H e^{-a\lambda n^{-2} D(0,y_n)}\mid 0\lra y_n]=\mathbb{E}_\nu[H]\psi_d(\lambda |y|^2).\]
    The convergence is uniform over $\epsilon<|y|<\epsilon^{-1}$ for $\epsilon>0$.
\end{proposition}
\begin{proof}
We use the argument from the previous section with small modifications. Instead of $Y_n = D_n$, we set 
\[ Y_n = H D_n = \sum_e H \ (1 + D(bubble(\overline{e}))) \mathbf{1}_{e \text{ pivotal for }\{0 \lra n \mathbf{e}_1\} }\ .\]
We define the truncations $Y_n^M$, $Y_{n, > L}^m$, and $Y_{n, L}^M$ analogously, with each sum exactly as before, except now multiplied by $H$; then we localize as before to produce an analogous $Y_{n, L}^{M, \text{ loc}}$. By the 
the general version of Theorem \ref{thm: local-limit} including the factor $H$, the moments of this $Y_{n, L}^{M, \text{ loc}}$ converge to the same limit as before, except multiplied by a factor of $\E_{\nu}[H]$.

In other words, the variable $Y_{n, L}^{M, \text{ loc}} / \E_\nu[H]$ converges in distribution to the law of the same $X(M)$ as defined in \eqref{eqn: XM-def}. The remainder of the argument from the previous section now applies with essentially no modification.
\end{proof}

 We continue working towards Theorem~\ref{thm: avoid-dist} using the following reduction:
\begin{proposition}\label{prop: laplace-conv}
Let $R_0>0$ be fixed and $F \subset B(R_0)$ be a finite set. Define $y_n=\lfloor n y \rfloor$. For $R>R_0$, define
\[A^F_{y_n}\subset E_R^F:=\{\text{there is an open path } 0\lra \partial B(R) \text{ in } B(R)\setminus F\}.\]
For any bounded Lipschitz function $f$,
\begin{equation}
    \big|\mathbb{E}[f(n^{-2} D^F(0,y_n)), A_{y_n}^F]-\mathbb{E}[1_{E^F_R}f(n^{-2} D(0,y_n)), 0\lra y_n]\big|\le C(R^{-1}\|f\|_{\infty}+ n^{-2}R^d\|f\|_{\mathrm{Lip}})\tau(0,y_n).
\end{equation}
\begin{proof}
Let
    \[\mathcal{B}_{R,n}:=\bigcup_{z\in\partial B(R)}\{B(R_0)\lra \partial B(R)\}\circ \{B(R_0)\lra z\}\circ \{z\lra y_n\}.\]
    By the BK inequality \eqref{eqn: BK}, we have
    \[\mathbb{P}(\mathcal{B}_{R,n})\le \sum_{z\in \partial B(R)}\mathbb{P}(B(R_0)\lra \partial B(R))\mathbb{P}(B(R_0)\lra z)\mathbb{P}(z\lra y_n).\]
    Next, we bound the factor by \eqref{eqn: KN} and the remaining by \eqref{eqn: HS-scaling}, we get the bound
    \begin{equation}\label{eqn: BRn-prob}
    \begin{split}
    &C_{R_0}\sum_{x\in B(R_0)}\sum_{z\in \partial B(R)} R^{-2}\langle x-z\rangle^{-d+2}\langle z-y_n\rangle^{-d+2}\\
    \le~&C_{R_0}R^{-2}\tau(0,y_n) \sum_{z\in \partial B(R)} \langle z\rangle^{-d+2}\\
    \le~& C_{R_0}R^{-1}\tau(0,y_n).
    \end{split}
    \end{equation}
On the event
    \[E_R^F\cap \{0\lra y_n\}\setminus \mathcal{B}_{R,n},\]
we construct a path realizing $A^F_{y_n}$ with length close to $D(0,y_n)$. Let $\gamma$ be a geodesic $0\rightarrow y_n$. If $\gamma\cap F=\emptyset$, we are done. Otherwise, let $v$ be the last vertex of $F$ on $\gamma$, and let $\gamma'$ be the portion of $\gamma$ from $v$ to $y_n$. Select a path $\pi$ from 0 to $\partial B(R)$ in $B(R^d)\setminus F$. Its existence is guaranteed on $E_R^F$. $\gamma'$ and $\pi$ cannot be disjoint on $\mathcal{B}_{R,n}^c$, so we can select a vertex $u\in \pi\cap \gamma'$, and consider the concatenation $\sigma$ of the portion of $\pi$ between $0$ and $u$ and the part of $\gamma' $ between $u$ and $y_n$. Then $\sigma$ is a path connecting $0$ to $y_n$ avoiding $F$ and
    \[|\sigma|\le |\mathcal{E}(B(R^d))|+D(0,y_n).\]
 This implies
 \begin{equation}\label{eqn: D-diff}
 D(0,y_n)\le D^F(0,y_n)\le D(0,y_n)+|\mathcal{E}(B(R^d))|
 \end{equation}
 on $E_R^F\cap \{0\lra y_n\}\setminus \mathcal{B}_{R,n}$, in particular
 \[E_R^F\cap \{0\lra y_n\}\setminus \mathcal{B}_{R,n} \subset A_{y_n}^F.\]
By \eqref{eqn: BRn-prob}, we have
\begin{align*} \mathbb{E}[1_{E^F_R}f(n^{-2} D(0,y_n)), 0\lra y_n]&= \mathbb{E}[1_{E^F_R\setminus \mathcal{B}_{R,n}}f(n^{-2} D(0,y_n)), 0\lra y_n]+ O(R^{-1}\|f\|_{L^\infty})\tau(0,y_n),\\
\mathbb{E}[f(n^{-2} D^F(0,y_n)),A_{y_n}^F]&=\mathbb{E}[1_{E^F_R\setminus \mathcal{B}_{R,n}}f(n^{-2} D^F(0,y_n)), 0\lra y_n]+ O(R^{-1}\|f\|_{L^\infty})\tau(0,y_n)
\end{align*}
Subtracting the two expressions and using \eqref{eqn: D-diff}, we obain the stated result.
 
\end{proof}
\end{proposition}
We apply Proposition \ref{prop: avoidH} with $H=\mathbf{1}_{E_R^F}$ and Proposition \ref{prop: laplace-conv}  with $f(t)=e^{-\lambda a n^{-2} t}$. For fixed $R$, uniformly in $\epsilon n\le |y_n|\le \epsilon^{-1}n$ we have
\begin{align*}
&\frac{\mathbb{E}[e^{-\lambda a n^{-2} D^F (0,y_n)}, A^F_{y_n}]}{\tau(0,y_n)}\\
=~&\frac{\mathbb{E}[1_{E_R^F} e^{-\lambda a n^{-2} D(0,y_n)}, 0\lra y_n]}{\tau(0,y_n)}+O(R^{-1})\\
=~&\nu(1_{E_R^F})\psi_d(\lambda |y|^2)+O(R^{-1}+R^dn^{-2})+o_n(1).
\end{align*}
Finally, use $\nu(E_R^F)\downarrow \nu(\Xi(F))$ to find
\begin{equation}\lim_{n\rightarrow \infty} \frac{\mathbb{E}[e^{-\lambda a n^{-2} D^F (0,y_n)}, A^F_{y_n}]}{\tau(0,y_n)} = \nu(\Xi(F))\psi_d(\lambda |y|^2)
\end{equation}
 by taking $n\rightarrow \infty$ and then $R\rightarrow\infty$. In particular, setting $\lambda=0$ gives
 \begin{equation}\lim_{n\rightarrow \infty} \frac{\mathbb{P}(A^F_{y_n)})}{\tau(0,y_n)}=\nu(\Xi(F)).
 \end{equation}
 Theorem \ref{thm: avoid-dist} follows from the last two limits.
 \qed

\section{Proof of Corollary \ref{cor: subcritical}}
We prove the corollary by analyzing the Laplace transform of the chemical distance. Letting $\lambda > 0$ be arbitrary, we study the asymptotic behavior of
\begin{align}\label{eq:laplace1}
    \E_p \left[\exp(-\lambda \mathrm{dist}(0,y) / |y|^2); \, 0 \lra y \right] &= \overline{\E} \left[\exp(-\lambda \mathrm{dist}_p(0,y) / |y|^2); \, 0 \sa{p} y \right].
\end{align}
Taking the convention $\dist(0, y) = \infty$ off $\{0 \lra y\}$, the restriction to the event $\{0 \lra y\}$ (respectively $\{0 \sa{p} y\}$) is superfluous but occasionally mnemonically useful.
Letting $\omega$ be a typical sample point and $\Sigma_{p_c}$ the sigma-algebra generated by the configuration of $p_c$-open edges, \eqref{eq:laplace1} is identical to
\begin{equation} \label{eq:laplace2} \overline{\E} \left[\overline{\E} \left[\exp(-\lambda \mathrm{dist}_p(0,y) / |y|^2) \mid \Sigma_{p_c} \right];\, 0 \sa{p_c} y  \right]\ . \end{equation}

We focus for the moment on the inner expectation in \eqref{eq:laplace2}. Note that $\dist_{p_c}(0,y) \leq \dist_{p}(0,y)$ $\overline \Pp$-a.s. It follows that 
\begin{equation}
    \label{eq:lapub}
    \eqref{eq:laplace2} \leq \overline{\E} \left[\overline{\E} \left[\exp(-\lambda \mathrm{dist}_{p_c}(0,y) / |y|^2)  \mathbf{1}_{\{0 \sa{p} y\}}\mid \Sigma_{p_c} \right];\, 0 \sa{p_c} y  \right],
\end{equation}
and so
\begin{equation}
    \label{eq:lapube}
    \eqref{eq:laplace1} \leq \overline{\E} \left[\exp(-\lambda \mathrm{dist}_{p_c}(0,y) / |y|^2)  \overline \Pp(0 \sa{p} y \mid \Sigma_{p_c}) \right] .
\end{equation}
We will show a lower bound that matches \eqref{eq:lapub} up to negligible error terms, then analyze the quantity on the right-hand side of \eqref{eq:lapub} and take expectations to get an analogue of \eqref{eq:lapube}. Analyzing the right-hand side of \eqref{eq:lapube} will allow us to complete the proof. We use the identity \eqref{eq:deltasum} to decompose the chemical distance into a sum over bubbles. We denote by $B_i$ the $i$th bubble in the configuration at $p_c$, with initial and terminal vertices $u_i$ and $v_i$, so that
\[\dist_{p_c}(0, y) = \sum_i[1 + \dist_{p_c}(u_i, v_i)], \]
and we write
\[\dist_p(u_i, v_i) = \dist_{p_c}(u_i, v_i) + \Delta_i \]
to control how the distances across bubbles change between the two configurations. (Of course, each $B_i$ may be split into multiple bubbles when moving from $p_c$ to $p$; we control this effect only coarsely and indirectly through bounds on $\Delta_i$.)

We reexpress the conditional expectation from \eqref{eq:laplace2} and then remove a couple of error terms. For $L \geq 1$ a large positive integer fixed relative to $y$ and hence $p$:
\begin{align}
    &\overline{\E} \left[\exp(-\lambda \mathrm{dist}_p(0,y) / |y|^2) \mid \Sigma_{p_c} \right]\nonumber\\
    = &\overline{\E} \left[\left. \exp\left(-\lambda \left[\mathrm{dist}_{p_c}(0,y) + \sum_i \Delta_i\right] / |y|^2\right)\  \right| \Sigma_{p_c} \right]\nonumber\\
    \geq  &\overline{\E} \left[\left. \exp\left(-\lambda \left[\mathrm{dist}_{p_c}(0,y) + \sum_{i: \, \mathrm{\diam}(B_i) \leq L} \Delta_i\right] / |y|^2\right) \mathbf{1}_{E_L^c \cap G_L^c}\  \right| \Sigma_{p_c} \right] - \overline \Pp(E_L \mid \Sigma_{p_c}) - \overline \Pp(G_L \mid \Sigma_{p_c})\ . \label{eq:smallterms}
\end{align}
Here $E_L$ is the event that some edge in a measurably chosen geodesic $\Gamma$ for the chemical distance between $u_i$ and $v_i$ in the $p_c$-open configuration is closed in the $p$-open configuration, for some $i$ such that $\diam(B_i) > L$. The event $G_L$ is the event that at least two edges in some $B_i$ with diameter at most $L$ are closed in the $p$-open configuration. 

We bound the terms on the right-hand side of \eqref{eq:smallterms}, by expressing disconnection probabilities in terms of critical expectations; this is an extension of an idea appearing in Kesten and Zhang \cite[Section 2]{KZ} (see also \cite[Section 5.3]{CHS}).
Letting $P_{y}$ denote the number of open pivotal edges in the $p_c$-open configuration for the event $\{0 \sa{p_c} y\}$, we have 
\begin{align*}
    \overline \Pp(G_L \mid \Sigma_{p_c}) &\leq \sum_{\substack{i \leq P_y(\omega): \\ \diam(B_i) \leq L}} \overline \Pp(\text{ at least two edges of $B_i$ closed in $p$-open configuration} \mid \Sigma_{p_c})\\
    &\leq \sum_{\substack{i \leq P_y(\omega): \\ \diam(B_i) \leq L}}\min\left\{1, \binom{\#B_i}{2} \left(1- \frac{p}{p_c}\right)^2\right\}\\
&\le  \left(\frac{m}{p_c |y|^2} \right) \sum_{\substack{i \leq P_y(\omega): \\ \diam(B_i) \leq L}}  \#B_i \min\left\{1,\left(\frac{m}{p_c |y|^2} \right) \#B_i\right\}\ .
\end{align*}

This tends to $0$ as $|y|\rightarrow \infty$ by a moment computation. Indeed, if $\mathrm{diam}(B_i)\le L$, then $\# B_i\le C L^d$, and
\begin{align*}
&\overline \Pp(G_L \cap \{0 \sa{p_c} y\}) = \overline \E[\overline \Pp(G_L \mid \Sigma_{p_c}); \, 0 \sa{p_c} y]\\
\leq &\left(\frac{m}{p_c |y|^2} \right) \E_{p_c}\left[\sum_{\substack{i = 1}}^{P_y(\omega)}  \#B_i \min\left\{1,\left(\frac{m}{p_c |y|^2} \right) \#B_i\right\}  \mathbf{1}_{\{\mathrm{diam}(B_i)\le L\}}\,;\, 0\lra  y  \right]\\
\le~& C\frac{m L^{2d}}{|y|^4}\mathbb{E}_{p_c}[P_y; 0\lra  y  ]\ .
\end{align*}
By the same diagrammatic argument as in the proof of Proposition~\ref{prop: outside-M}, the expectation in the last line above is at most of order $|y|^{4-d}$ as $|y| \to \infty$. It follows that 
\begin{align}\label{eq:GLbd}
\Pp(G_L \cap \{0 \sa{p_c} y\}) \le \frac{C m L^{2d}}{|y|^d}\ ,
\end{align}
and in the order of limits $|y| \to \infty$ followed by $L \to \infty$ that we will take in controlling \eqref{eq:smallterms}, the right-hand side of \eqref{eq:GLbd} is of smaller order than the first term of \eqref{eq:smallterms} (which is of order $|y|^{2-d}$).

We control the expectation of the other error term from \eqref{eq:smallterms}; recalling that we fixed a random variable $\Gamma$ representing a shortest open path from $0$ to $y$ in the $p_c$-open configuration, 
\begin{align}
    &\overline \Pp(E_L) = \overline \E[\overline \Pp(E_L \mid \Sigma_{p_c})]\nonumber\\
    \leq & \overline \E\left[ \sum_{e} \mathbf{1}_{e \in \Gamma, \, e \text{ in a bubble of diameter $\geq L$}}\right]\nonumber\\
    \leq & C L^{6-d} (\log L)^3 |y|^{4-d}\ , \label{eq:ELbd}
\end{align}
where we have straightforwardly used Proposition~\ref{prop: large-bubble-trunc} in the last step. This is also clearly smaller order than the first term of \eqref{eq:smallterms} in the parameter regime we will consider.

We now return to the first, positive term of \eqref{eq:smallterms}, rewriting it as
\begin{equation}\label{eqn: factor} \exp(-\lambda \mathrm{dist}_{p_c}(0,y) / |y|^2)\ \overline{\E} \left[\left. \exp\left(-\lambda \left[ \sum_{i: \, \mathrm{\diam}(B_i) \leq L} \Delta_i\right] / |y|^2\right), 0 \sa{p_c} y , E_L^c \cap G_L^c\  \right| \Sigma_{p_c} \right].
\end{equation}
On $E_L^c \cap G_L^c$, an open connection from $0$ to $y$ exists in the $p$-open configuration unless some open pivotal in the $p_c$-open configuration is closed in the $p$-open configuration. 
Using the inequality
\[1-e^{-x}\le x, \quad x>0\]
we have
\begin{align*}
&\overline{\E} \left[\left. \exp\left(-\lambda \left[ \sum_{i: \, \mathrm{\diam}(B_i) \leq L} \Delta_i\right] / |y|^2\right) E_L^c \cap G_L^c, 0 \sa{p} y\  \right| \Sigma_{p_c} \right]\\
\ge ~&\overline{\Pp} \left(\left. E_L^c \cap G_L^c, 0 \sa{p} y\  \right| \Sigma_{p_c} \right)- \frac{\lambda L^d }{|y|^2}\sum_i \overline{\Pp} \left(\left. \begin{array}{c}\text{ there is at least one edge deletion in } B_i, \\ \text{all $p_c$-open pivotals open in $p$} \end{array} \right| \Sigma_{p_c} \right)\\
\ge ~& \overline{\Pp} \left(\left. E_L^c \cap G_L^c, 0 \sa{p} y\  \right| \Sigma_{p_c} \right)- \frac{C\lambda L^{2d}P_y \overline \Pp( \text{all $p_c$-pivotals open in $p$} \mid \Sigma_{p_c})}{|y|^4}.
\end{align*}

It follows that the display \eqref{eqn: factor} is lower-bounded by
\[\exp(-\lambda \mathrm{dist}_{p_c}(0,y) / |y|^2)\ \left(\overline{\Pp} \left(\left. 0 \sa{p} y , E_L^c \cap G_L^c\  \right| \Sigma_{p_c} \right)-\frac{C\lambda L^{2d}P_y\overline \Pp( \text{all $p_c$-pivotals open in $p$} \mid \Sigma_{p_c}) }{|y|^4}\right)\ .\]

Note that for a given $p_c$-open configuration such that $\{0 \sa{p_c} y\}$ occurs, the event $\{ 0 \sa{p} y \} \cap E_L^c \cap G_L^c$ occurs exactly when all $p_c$-open pivotals are also open in $p$. Putting the last display back into \eqref{eq:smallterms}, taking expectations, and applying \eqref{eq:GLbd} and \eqref{eq:ELbd} yields that
\begin{equation}
\label{eq:laplb}
\begin{gathered}
  \eqref{eq:laplace1}  \geq \overline{\E} \left[ \exp(-\lambda \mathrm{dist}_{p_c}(0,y) / |y|^2)  \overline \Pp( \text{all $p_c$-pivotals open in $p$} \mid \Sigma_{p_c}) \right]\Big(1-\frac{C\lambda L^{2d}P_y}{|y|^4}\Big)\\
  - C L^{6-d} (\log L)^3 |y|^{4-d}  - \frac{C m L^{2d}}{|y|^d}\ .
  \end{gathered}
\end{equation}
This essentially matches \eqref{eq:lapube}, as will be clear when we finish our analysis of the positive term. Note that
\begin{align*}
     \overline \Pp( \text{all $p_c$-pivotals open in $p$} \mid \Sigma_{p_c}) &= \left( \frac{p}{p_c} \right)^{P_y} = \left( 1 - \frac{m}{p_c |y|^2} \right)^{P_y}\\
     &= \exp\left( P_y \log \left[ 1 - \frac{m}{p_c |y|^2} \right] \right)\\
     &\geq \exp\left( - P_y \left[ \frac{m}{p_c |y|^2} + C \left(\frac{m}{p_c |y|^2} \right)^2 \right] \right)\ .
\end{align*}
In particular, the first term of \eqref{eq:laplb} is lower bounded by
\begin{equation}
    \label{eq:laplb2}
    \overline{\E} \left[ \exp(-\lambda \mathrm{dist}_{p_c}(0,y) / |y|^2 -  P_y \left[ \frac{m}{p_c |y|^2} + C \left(\frac{m}{p_c |y|^2} \right)^2 \right])\left[1 - \frac{C \lambda L^d P_y }{|y|^2} \right] \right]
\end{equation}

The distributional convergence of Theorem~\ref{thm: main} gives that the vector $|y|^{-2} (\dist_{p_c}(0,y), \alpha_d \alpha_p^{-1} P_y)$ converges in distribution under $\Pp_{p_c}(\cdot \mid 0 \lra y)$ to a vector whose components are a.s.~equal. We apply this in the last expression and take the dual limit as $|y| \to \infty$ followed by $L \to \infty$, and we apply the same procedure to \eqref{eq:lapube}. We arrive at the conclusion that
\begin{equation}
    \label{eq:laplim}
    \lim_{|y| \to \infty} |y|^{d-2} \E_p \left[\exp(-\lambda\  \mathrm{dist}(0,y) / |y|^2)\right] = \lim_{|y| \to \infty} |y|^{d-2} \E_{p_c} \left[ \exp\left(- \left(\lambda + \frac{m \alpha_p}{\alpha_d} \right) \frac{\dist(0,y)}{|y|^2}  \right) \right]\ .
\end{equation}
We replace $\lambda$ in \eqref{eq:laplim} and apply our result in Theorem~\ref{thm: avoid-dist} to see exactly the result claimed at \eqref{eq:lapsubchem} in the statement of the corollary.

The remaining results claimed in the statement of the corollary now follow from the characterization of $T_d$ and $\psi_d$ at at \eqref{eqn: laplace-Td}. \qed

\bigskip
{\bf Acknowledgements.} We are grateful to Manuel Cabezas and Alexander Fribergh for comments on earlier versions of this manuscript. We should also like to thank David Croydon and Noe Kawamoto for pointing out a mistake in the way the truncation was done in Sections 3 and 5 in said previous version.

The research of S.~C.~was supported by NSF grant DMS-2154564. The research of J.~H.~was supported by NSF grant DMS-1954257. The research of P.S. was supported by NSF grants DMS-2154090 and DMS-2238423, and a Simons Fellowship. This work was completed while P.S. was in residence at SLMath.

\bigskip
{\bf Data Availability Statement.} We do not analyse or generate any datasets, because our work proceeds within a theoretical and mathematical approach.


\begin{thebibliography}{9}
\bibitem{AP} Antal, P. and Pisztora, A. \emph{On the chemical distance for supercritical Bernoulli percolation.} Ann. Probab., 24, 1996.
\bibitem{AN}  Aizenman, M. and Newman, C., \emph{Tree graph inequalities and critical behavior in percolation models}. J. Stat. Phys. 36, 1984.
\bibitem{BA} Barsky, D. J. and Aizenman, M., \emph{Percolation Critical Exponents Under the Triangle Condition}, Ann. Probab.,  19, 1991.
\bibitem{BK} van den Berg, J. and Kesten, H.  \emph{Inequalities with applications to percolation and reliability.} J. Appl. Probab.
 22, 1985.
 \bibitem{BRH} Blanc-Renaudie, A. and Hutchcroft, T., \emph{Super-Brownian limits and the k-point function for high-
dimensional percolation}. Preprint arXiv:2607.2238

\bibitem{CCFK} Cabezas, M, Croydon, D., Fribergh, A., Kawamoto, N. \emph{Range convergence and sharp one-arm asymptotics for the critical percolation cluster in high dimensions}, arXiv:2607.24561

Authors: M. Cabezas, D. Croydon, A. Fribergh, N. Kawamoto
\bibitem{CHS} Chatterjee, S., Hanson, J. and Sosoe, P.. \emph{Subcritical connectivity and some exact tail exponents in high dimensional percolation.} Commun. Math. Phys. 403, 2023.
\bibitem{CCHS} Chatterjee S., Chinmay, P., Hanson, J. and Sosoe, P., \emph{Robust construction of the incipient infinite cluster in high dimensional critical percolation}. Preprint arXiv:2502.10882, 2025.
\bibitem{CCHS1} Chatterjee, S., Chinmay, P., Hanson, J. and Sosoe, P., \emph{Convergence of $k$-point functions in high
dimensional percolation.} Preprint, 2026.
\bibitem{DHS} Damron, M., Hanson, J. and Sosoe, P. \emph{Strict Inequality for the Chemical Distance Exponent in Two‐Dimensional Critical Percolation}. Comm. Pure Appl. Math. 74, 2021.
\bibitem{Doob} Doob, J.L., \emph{Conditional brownian motion and the boundary limits of harmonic functions} Bull. Soc. Math. France, 85, 1957.
\bibitem{DMPS} Duminil-Copin, H., Markar, A., Panis, R. and Slade, G. \emph{A random walk approach to high-dimensional critical phenomena}. Preprint arXiv:2605.21438.
\bibitem{FH} Fitzner, R. and van der Hofstad, R. \emph{Mean-field behavior for nearest-neighbor percolation in $d>10$}, Electron. J. Probab. 22, 2017.
\bibitem{GM} Grimmett, G., Marstrand, J.M. \emph{The Supercritical Phase of Percolation is Well Behaved}, Proc. R. Soc. Lond. A Math. Phys. Sci. 430, 1990.
\bibitem{KZ} Kesten, H., Zhang, Y., \emph{The tortuosity of occupied crossings of a box in critical percolation}, Journal of Statistical Physics \textbf{70} no. 3-4, 1993.
\bibitem{KN} Kozma, G. and Nachmias, A. \emph{Arm exponents in high dimensional percolation.} J. Amer. Math. Soc. 24, 2011.
\bibitem{KN2} Kozma, G. and Nachmias, A.  \emph{The Alexander-Orbach conjecture holds
in high dimensions}, Invent. Math., 2009.
\bibitem{HS0} Hara, T. and Slade, G., \emph{Mean-Field Critical Behaviour for Percolationin High Dimensions}. Commun.~Math.~Phys.~128, 1990.
\bibitem{HS} Hara, T. and Slade, G.,
\emph{The incipient infinite cluster in high-dimensional percolation}. Electronic Research Announcements of the AMS, Volume 4, 1998.
\bibitem{HS1} Hara, T. and Slade, G.,
\emph{The scaling limit of the incipient infinite cluster in high-dimensional percolation. I. Critical exponents}. J. Stat. Phys. 99, 2000.
\bibitem{HS2} Hara, T. and Slade, G.,
\emph{The scaling limit of the incipient infinite cluster in high-dimensional percolation. II. Integrated super-Brownian excursion}. J. Math. Phys. 41, 2000.
\bibitem{Hara} Hara, T., \emph{Decay of correlations in nearest-neighbor self-avoiding walk, percolation, lattice trees and animals.} Ann. Probab. 36, 2008.
\bibitem{HTWB} Havlin, S., Trus, B., Weiss, G.H. and Ben-Avraham, D., \emph{The chemical distance distribution in percolation clusters}, J. Phys. A. 1985.
\bibitem{HH} Heydenreich, M. and van der Hofstad, R. \emph{Progress in high-dimensional percolation and random graphs}. Cham: Springer, 2017.
\bibitem{HHH} Heydenreich, M., van der Hofstad, R. and  Hulshof, T. \emph{High-dimensional incipient infinite clusters revisited}. J. Stat. Phys. 155, 2014.
\bibitem{HHH2} Heydenreich, M., van der Hofstad, R. and  Hulshof, T. \emph{Random Walk on the High-Dimensional IIC}. Commun.~Math.~Phys.~329, 2014.
\bibitem{HHS} Hara, T., Slade, G. and van der Hofstad, R. \emph{Critical two-point functions and the lace expansion for spread-out high-dimensional percolation and related models.} Ann. Probab. 31, 2003.
\bibitem{vdHJ} van der Hofstad, R.~and Járai, A.A., \emph{The incipient infinite cluster for high-dimensional unoriented percolation.} J.~Statist.~Phys. ~114, 2004.
\bibitem{Hutch1} Hutchcroft, T., \emph{Critical long-range percolation I: High effective dimension}, preprint arXiv:2508.18807.
\bibitem{Hutch2} Hutchcroft, T., \emph{Critical long-range percolation II: Low effective dimension}, preprint arXiv:2508.18808.
\bibitem{Hutch3} Hutchcroft, T., \emph{Critical long-range percolation III: The upper critical dimension}, preprint arXiv:2508.18809.
\bibitem{HMS} Hutchcroft, T., Michta, E., and Slade, G. \emph{High-dimensional near-critical percolation and the torus plateau}. Ann. Probab. 51, 2023.
\bibitem{Law} Lawler, G. \emph{Notes on the Bessel Process}, \url{https://www.math.uchicago.edu/~lawler/bessel18new.pdf}.
\bibitem{LG} LeGall, J.F., \emph{Brownian Motion, Martingales and Stochastic Calculus}. Graduate Texts in Mathematics. Springer, 2016.
\bibitem{LP} Lyons, R., and Peres, Y. \emph{Probability on trees and networks}. Vol. 42. Cambridge University Press, 2017.
\bibitem{SL1} Liu, Y. and Slade, G. \emph{Gaussian deconvolution and the lace expansion.} Probab. Theory
Related Fields, 195, 2026.
\bibitem{SL2} Liu, Y. and Slade, G. \emph{Gaussian deconvolution and the lace expansion for spread-out models}. Ann. Inst. Henri Poincar\'e Probab. Stat., 62, 2026.
\bibitem{SL3} Liu, Y. and Slade, G. \emph{Crossover from subcritical to critical decay: random walk, self-avoiding walk, percolation}, preprint arXiv:2605.11545.
\bibitem{RW} Williams, D. and Rogers, L.C.G. \emph{Diffusions, Markov Processes and Martingales}, Cambridge University Press, Volume 2. It\^o Calculus. 
\bibitem{WP-bessel} Wikipedia contributors, \emph{Bessel function}, Wikipedia, The Free Encyclopedia, 16 June 2026, 05:39 UTC, \url{https://en.wikipedia.org/w/index.php?title=Bessel_function&oldid=1359607442}
\end{thebibliography}
\end{document}